\documentclass[11pt,reqno]{amsart}
\usepackage{CJK,CJKpunct}
\usepackage{subeqnarray}
\usepackage{cases}
\usepackage{stmaryrd}
\usepackage{color}
\usepackage{amsfonts}
\usepackage{mathrsfs}
\usepackage{amssymb,amsmath,amsthm}
\usepackage{epsfig}
\usepackage{graphicx}
\usepackage{appendix}
\usepackage{lineno}
\usepackage{enumerate}
\usepackage{comment}
\usepackage{caption}
\usepackage{cleveref}
\newcommand{\dps}{\displaystyle}
\newtheorem{theorem}{\indent Theorem}[section]

\newtheorem{remark}{\indent Remark}[section]

\newcommand{\ba}{\begin{array}}\newcommand{\ea}{\end{array}}
\newcommand{\be}{\begin{eqnarray}}\newcommand{\ee}{\end{eqnarray}}
\newcommand{\beq}{\begin{equation*}}\newcommand{\eeq}{\end{equation*}}
\newcommand{\bex}{\begin{eqnarray*}}
\newcommand{\eex}{\end{eqnarray*}}

\def\bq{\begin{equation}}
\def\eq{\end{equation}}
\def\beq{\begin{equation*}}
\def\eeq{\end{equation*}}
\def\br{\begin{eqnarray}}
\def\er{\end{eqnarray}}
\def\brr{\bq\begin{array}{r@{}l}}
\def\err{\end{array}\eq}
\def\bry{\beq\begin{array}{r@{}l}}
\def\ery{\end{array}\eeq}
\def\brl{\bq\begin{array}{l}}
\def\erl{\end{array}\eq}
\def\bryl{\beq\begin{array}{l}}
\def\eryl{\end{array}\eeq}

\font\tenbi=cmmib10   at 11 pt
\font\sevenbi=cmmib10 at 9pt
\font\fivebi=cmmib7 at 6pt
\newfam\bifam
\textfont\bifam=\tenbi \scriptfont\bifam=\sevenbi  \scriptscriptfont\bifam=\fivebi

\def\bi{\fam\bifam\tenbi}
\font\sixtdb=msbm10 at 16 pt \font\tendb=msbm10 at 12 pt  \font\sevendb=msbm7
\newfam\dbfam
\textfont\dbfam=\sixtdb

\textfont\dbfam=\tendb \scriptfont\dbfam=\sevendb

\def\Dt {\triangle t}

\def\x{{\bi x}}

\def\Dt {\tau}

\title[Linear  adaptive  BDF2 scheme for the PFC equation]
{A linear adaptive second-order backward differentiation formulation scheme for the phase field crystal equation$^*$}
\author[Dianming Hou and Zhonghua Qiao]
{Dianming Hou$^{1}$
\quad
Zhonghua Qiao$^{2}$
}
\thanks{\hskip -12pt
${}^*$The work of D. Hou is supported by NSFC grant 12001248, the NSF of the Jiangsu Province grant
BK20201020, the NSF of Universities in Jiangsu Province of China grant 20KJB110013 and the Hong Kong Polytechnic University grant 1-W00D. Z. Qiao's work is partially supported by the Hong Kong Research Grant Council RFS grant RFS2021-5S03 and GRF grant 15302919, the Hong Kong Polytechnic University internal grant 4-ZZLS, and the CAS AMSS-PolyU Joint Laboratory of Applied Mathematics.
\\
$^{1}$School of Mathematics and Statistics, Jiangsu Normal University, 221116 Xuzhou, China. Current address: Department of Applied Mathematics, The Hong Kong Polytechnic University, Hung Hom, Kowloon, Hong Kong.
  (dmhou@stu.xmu.edu.cn).\\
$^{2}$Corresponding author. Department of Applied Mathematics, The Hong Kong Polytechnic University, Hung Hom, Kowloon, Hong Kong. (zqiao@polyu.edu.hk).\\
}
\begin{document}
\graphicspath{{figures/},}
\date {\today}
\maketitle

\begin{abstract}
In this paper, we present and analyze a linear fully discrete second order scheme with variable time steps for the phase field crystal  equation.
More precisely, we construct a linear adaptive time stepping scheme based on the second order backward differentiation formulation (BDF2) and use the Fourier spectral method for the spatial discretization.
The scalar auxiliary variable approach is employed to deal with the nonlinear term, in which we only adopt a first order method to approximate the auxiliary variable.
This treatment is extremely important in the derivation of the unconditional energy stability of the proposed adaptive BDF2 scheme.
However, we find for the first time that this strategy will not affect the second order accuracy of the unknown phase function $\phi^{n}$ by setting the positive constant $C_{0}$ large enough such that $C_{0}\geq 1/\Dt.$
The energy stability of the adaptive BDF2 scheme is established with a mild constraint on the adjacent time step radio $\gamma_{n+1}:=\Dt_{n+1}/\Dt_{n}\leq 4.8645$.
Furthermore, a rigorous error estimate of the second order accuracy of $\phi^{n}$ is derived for the proposed scheme on the nonuniform mesh by using the uniform $H^{2}$ bound of the numerical solutions.
Finally, some numerical experiments are carried out to validate the theoretical results and demonstrate the efficiency of the fully discrete adaptive BDF2 scheme.
\end{abstract}

\section{Introduction}
The phase field crystal (PFC) model was firstly proposed in \cite{EG04,EKHG02} and has been frequently used for modelling the crystal growth at the atomic sale in space but on diffusive scale in time.
It can simulate nucleated crystallites at arbitrary locations and orientations
with the elastic and plastic deformations.
The PFC model can account for incorporating the periodic nature of the crystal lattices through the following Swift-Hohenberg type free energy, defined by
\bq
E(\phi)=\int_{\Omega}\Big(\frac{1}{2}\phi(\Delta+\beta)^{2}\phi+\frac{1}{4}\phi^{4}-\frac{\varepsilon}{2}\phi^{2}\Big)d\x,
\eq
which is minimized by periodic density fields. Here $\Omega\subset\mathbb{R}^{d}, d=2$ or $3$, $\phi$ is the atom density field, $\beta$ and $\varepsilon$ are two positive constants such that $0<\varepsilon<\beta^{2}$. The operator $(\beta+\Delta)^{2}\phi$ is defined by
$
(\Delta+\beta)^{2}\phi=\Delta^{2}\phi+2\beta\Delta\phi+\beta^{2}\phi.
$
Thus the conserved phase field crystal equation reads
\brr\label{prob}
\begin{cases}
\dps\frac{\partial \phi}{\partial t}
=\Delta\mu,& \mbox{in}\quad \Omega\times(0,T),\\
\dps\mu=(\Delta+\beta)^{2}\phi+\phi^{3}-\varepsilon\phi,& \mbox{in}\quad \Omega\times(0,T),\\
\dps \phi(\x,0)=\phi_{0}(\x), &\mbox{in}\quad \Omega.
\end{cases}
\erl

Due to the gradient structure of \eqref{prob}, the above PFC system satisfies the following energy dissipation law
 \bq\label{EDlaw}
 \frac{d}{dt}E(\phi)
 =-\|\nabla\mu\|^{2}\leq0.
 \eq
This indicates the free energy $E(\phi)$ always decreases in time. In numerical simulations, it is crucial to design numerical schemes preserving the above energy dissipation law at the discrete level,  see e.g., \cite{BLWW13,Ell93,Eyr98,GLWW14,DJLQ_rev,Yan16,Zha17,Shen17_1,Shen17_2,JLQ22_1,JLQ22_2,LL21} and references therein. These schemes are often called energy stable schemes in practice. A direct implicit/explicit treatment for the nonlinear term in \eqref{prob} usually yields a severe stability restriction on the time step sizes to preserve the discrete form of the energy law \eqref{EDlaw}. In the last two decades, many efforts have been made to relax this restriction. For example, Wise, Wang and Lowengrub developed a first order convex splitting scheme for the PFC equation in \cite{WWL09}.
The unconditional energy stability and the corresponding error estimate of their proposed scheme were rigorously established.
 The second order convex splitting method for the PFC model based on the Crank-Nicolson formula was also studied in their subsequent work \cite{HWWL09}. They have proved that the Crank-Nicolson convex splitting scheme was unconditionally energy stable in the sense of the discrete free energy bounded by the initial one, which indicates the uniform $L^{\infty}$ bound of the numerical solutions. Then the corresponding error estimate of the proposed scheme could be easily derived with the help of the uniform bound of the numerical solutions, seeing \cite{DFWWZ18}.
 In \cite{GX16}, Guo and Xu presented first order and second order unconditionally energy stable schemes based on a local discontinuous Galerkin method with the convex splitting approach. An unconditionally energy stable BDF2 convex splitting scheme has been developed for the PFC model by Li, Mei and You in \cite{LMY18}.
 They rigorously established the mass conservation, the unconditional energy stability and the convergence of the proposed numerical scheme.
 Though all the above convex splitting schemes are uniquely solvable due to the gradient of a strictly convex function,
 their resulted nonlinear systems must be solved at each time step, which is relatively time-consuming.
 Thus it is highly desired to develop linear high order energy stable schemes for the PFC model \eqref{prob}.
 In \cite{YH17}, Yang et al. proposed two linear second order unconditionally energy stable schemes for the PFC equation using the invariant energy quadratization (IEQ) approach.
 Recently, Li and Shen have adopted the scalar auxiliary variable (SAV) approach to construct a linear Crank-Nicolson scheme for the PFC model in \cite{LS20}, where the unconditional energy stability and the corresponding error estimate of the fully discrete scheme were rigorously proved.

Another important feature of the PFC model is that its evolution process usually takes rather long time before it reaches the steady state. Moreover, in this process, the solution of the PFC equation usually undergoes both fast and slow changes at different time. Therefore, it is reasonable to employ some time adaptive strategies in the simulation, where unconditionally energy stable numerical schemes with variable time steps are very needed. Combining the second order Crank-Nicolson convex splitting scheme, an adaptive time-stepping strategy based on the variation of the free energy was adopted for the PFC model in \cite{ZMQ13}, where both the
steady-state solution and the dynamical changes of the solution have been resolved accurately and efficiently.
As known, the BDF2 schemes usually achieve stronger stability than the one based on the Crank-Nicolson formula, especially for the equations with strong stiffness \cite{HAX19,HX21_3,DYL18}.
Recently, an adaptive fully implicit BDF2 time-stepping method was studied for the PFC model in \cite{LJZ22}. Under a restriction on the adjacent time-step ratio $\gamma_{n}\leq3.561$, the energy stability and the corresponding $L^{2}$ error estimate of the proposed adaptive scheme were rigorously derived by using the discrete orthogonal convolution kernels and convolution inequalities. However, these adaptive time-stepping schemes in \cite{LJZ22,ZMQ13} are also nonlinear and the nonlinear iteration is unavoidable at each time step.
Very recently, combining the nonuniform BDF2 formula and the SAV approach, we successfully derived a linear second order unconditionally energy stable time stepping scheme for gradient flows with the adjacent time-step ratio $\gamma_{n}\leq 4.8465$ in \cite{HQ21}. Though it avoids solving the nonlinear system for the phase function $\phi^n$ by explicit treatment for the nonlinear term, there still exits a nonlinear algebraic equation to be solved for the auxiliary variable at each time step. 

The goal of this paper is to develop and analyze a linear second order adaptive time-stepping scheme based on the BDF2 formula for the PFC model \eqref{prob} using the SAV approach proposed in \cite{Shen17_1,Shen17_2}.
The linear BDF2 scheme based on the SAV approach has been proved unconditionally energy stable on the uniform temporal mesh for the PFC model in \cite{YH17}. However, it seems to be rather difficult to prove similar results on nonuniform temporal meshes. In this paper, we construct a linear unconditionally energy stable BDF2 scheme for the PFC model on the nonuniform temporal mesh for the first time
by using a first order approximation on the auxiliary variable. The first order treatment of the auxiliary variable will not affect the second order accuracy of the unknown atom density field $\phi^n$ by setting the positive constant $C_{0}$ large enough such that $C_{0}\geq 1/\Dt$.
Moreover, the constructed adaptive BDF2 scheme avoids solving the nonlinear algebraic equation for the auxiliary variable as needed in our previous work \cite{HQ21}. More precisely, there are only two sixth-order equations with constant coefficients to be solved at each time step, which can be solved efficiently by exiting fast Poisson solvers.
Furthermore, under a mild restriction on the adjacent time-step ratio $\gamma_{n}\leq 4.8465$, we rigorously establish the unconditional energy stability and the corresponding error estimate of the proposed fully discrete  BDF2 scheme on the nonuniform temporal mesh.

The rest of the paper is organized as follows. In \cref{sec:sect2}, we review the SAV method proposed in \cite{Shen17_1,Shen17_2}, and construct the linear BDF2 scheme on the nonuniform temporal mesh. The unconditional energy stability for the proposed scheme is rigorously proved in the sense of a modified discrete energy.
In \cref{sec:sect3}, a rigorous error estimate of the fully discrete adaptive BDF2 scheme is derived for the PFC model, which shows the second order accuracy of the numerical solution $\phi^{n}$ in time with larger enough constant $C_{0}$ such that $C_{0}\geq1/\Dt.$
Several numerical examples are given in \cref{sec:sect4} to validate the theoretical results and the efficiency of the proposed method.
Finally, the paper ends with some concluding remarks in \cref{sec:conclusions}.


\section {Scalar auxiliary variable approach}\label{sec:sect2}
\setcounter{equation}{0}

We first introduce some notations which will be used throughout the paper. Denote $\|\cdot\|_{m}$ the standard $H^{m}$ norm. The norm and inner product of $L^{2}(\Omega):=H^{0}(\Omega)$ are denoted by $\|\cdot\|$ and $(\cdot,\cdot)$, respectively.
For the sake of simplicity, we consider the periodic boundary condition for \eqref{prob}.

The key idea of the SAV approach \cite{HAX19,Shen17_1} is to transform the nonlinear term to a simple quadratic form by introducing a scalar auxiliary variable $r(t)$, defined by
\bq\label{SAV}
r(t)=\sqrt{E_{1}(\phi)+C_{0}}:=\sqrt{\int_{\Omega}F(\phi)d\x+C_{0}},
\eq
where $F(\phi)=\frac{1}{4}\phi^{4}-\frac{S+\varepsilon}{2}\phi^{2}$ with $S\geq0$, and $C_{0}$ is a positive constant such that $E_{1}(\phi)+C_{0}>0.$
In this paper,  we further assume that  $C_{0}$ is large enough such that
\bq\label{eq_c}
\frac{C_{0}}{2}\leq E_{1}(\phi)+C_{0}\leq 2C_{0}.
\eq
Then, an equivalent system of the original problem \eqref{prob} with scalar auxiliary variable is given as
 \brr\label{re_prob}
 \begin{cases}
 \dps\frac{\partial \phi}{\partial t}=\dps \Delta\mu,\\[9pt]
 \dps\mu\dps=(\Delta+\beta)^{2}\phi+S\phi+\frac{r(t)}{\sqrt{E_{1}(\phi)+C_{0}}}F'(\phi),\\[15pt]
 \dps\frac{d r}{d t}\dps=\frac{1}{2\sqrt{E_{1}(\phi)+C_{0}}}\int_{\Omega}F'(\phi)
\frac{\partial \phi}{\partial t}d\x.
 \end{cases}
 \err
\eqref{re_prob} satisfies the following energy dissipation law 
 \bq\label{mod_energ}
 \dps\frac{d\overline{E}(\phi,r)}{dt}=-\|\nabla\mu\|^{2}\leq0,
 \eq
 where $\overline{E}(\phi,r):=\frac{1}{2}\|(\Delta+\beta)\phi\|^{2}+\frac{S}{2}\|\phi\|^{2}+r^{2}$.
 Moreover, we have
 \bq\label{H2c}
 \overline{E}(\phi,r)\leq {\color{black}\overline{E}}(\phi_{0}(\x),r(0)),
 \eq
 which shows the $H^{2}$ bound of the solution of \eqref{re_prob}.
 From the definition of the auxiliary variable $r(t)$ in \eqref{SAV}, it follows that $\overline{E}(\phi,r)=E(\phi)+C_{0}$, which indicates that the energy dissipation law \eqref{mod_energ} is the same as \eqref{EDlaw}. 

We are now ready to construct efficient time stepping
 schemes for \eqref{re_prob} to numerically approximate the solution $\phi$  on the nonuniform temporal mesh.  Let $\Dt_{n}:=t_{n}-t_{n-1}, n=1,2\dots, M$ be the time step size, $\gamma_{n+1}={\Dt_{n+1}}/{\Dt_{n}}$ be the adjacent time--step ratio,
 and $\tau=\max\{\Dt_{n},n=1,2,\dots, M\}$ be the maximum time step size of the temporal mesh.
\subsection{{\color{black}A} second-order BDF scheme on nonuniform temporal mesh}
We firstly recall the first order scheme for solving \eqref{re_prob} reported in \cite{LS20}
 \begin{subequations}\label{eq1}
 \begin{align}
 \dps{}\frac{\phi^{n+1}-\phi^{n}}{\Dt_{n+1}}&\dps=\Delta \mu^{n+1},\label{eq1_1}\\
 \dps\mu^{n+1}&\dps=(\Delta+\beta)^{2}\phi^{n+1}+S\phi^{n+1}+\frac{r^{n+1}}{\sqrt{E_{1}^{n}+C_{0}}}F'(\phi^{n}),\label{eq1_2}\\
 \dps\frac{r^{n+1}-r^{n}}{\Dt_{n+1}}&\dps=\frac{1}{2\sqrt{E_{1}^{n}+C_{0}}}\int_{\Omega}F'(\phi^{n})
\frac{\phi^{n+1}-\phi^{n}}{\Dt_{n+1}}d\x,\label{eq1_3}
 \end{align}
 \end{subequations}
where $\phi^{n}$ is an approximation to $\phi(t_n)$ and $E^{n}_{1}:=E_{1}(\phi^{n})$.
As stated in \cite{LS20}, the scheme \eqref{eq1} is unconditionally energy stable in the sense that the following discrete energy law holds
\bq\label{eq2}
\overline{E}(\phi^{n+1},r^{n+1})-\overline{E}(\phi^{n},r^{n})\leq0.
\eq

To construct a second-order BDF scheme for \eqref{re_prob}, we firstly introduce a second order approximation $F^{j+\sigma}_{2}\phi$ to $\phi'(t)$ at $t=t_{j+\sigma}$ as follows, seeing also \cite{HQ21,HX22_1}. Here, $t_{j+\sigma}:=t_{n}+\sigma\Dt_{n+1}$ with $\frac{1}{2}\leq\sigma\leq1.$
\bq\label{BDF_s}
\dps F^{j+\sigma}_{2}\phi:=\frac{1}{\Dt_{j+1}}\Big[\frac{1+2\sigma\gamma_{j+1}}{1+\gamma_{j+1}}\phi^{j+1}-(1+(2\sigma-1)\gamma_{j+1})\phi^{j}+\frac{(2\sigma-1)\gamma^{2}_{j+1}}{1+\gamma_{j+1}}\phi^{j-1}\Big],
\eq
where  $\phi^{j}$ is an approximation to $\phi(t_{j})$.
Note that when $\sigma=1$, it becomes the classical second order BDF with variable time steps:
\be\label{BDF2b}
F^{j+1}_{2}\phi=\frac{1}{\Dt_{j+1}}\Big[\frac{1+2\gamma_{j+1}}{1+\gamma_{j+1}}\phi^{j+1}-(1+\gamma_{j+1})\phi^{j}+\frac{\gamma^{2}_{j+1}}{1+\gamma_{j+1}}\phi^{j-1}\Big],
\ee
 and Crank-Nicolson formula
\bex
F^{j+1/2}_{2}\phi=\frac{\phi^{j+1}-\phi^{j}}{\Dt_{j+1}}
\eex
for $\sigma=\frac{1}{2}$.

Now, we are ready to construct the second-order BDF scheme on nonuniform temporal mesh for the PFC system \eqref{re_prob},
 \begin{subequations}\label{BDF_2}
 \begin{align}
 &\dps F^{n+\sigma}_{2}\phi=\Delta \mu^{n+\sigma},\label{BDF_2_1}\\
 &\dps\mu^{n+\sigma}=(\Delta+\beta)^{2}\phi^{n+\sigma}+S\phi^{n+\sigma}+\frac{r^{n+1}}{\sqrt{E_{1}^{n}+C_{0}}}F'(\phi^{*,n+\sigma}),\label{BDF_2_2}\\
 &\dps\frac{r^{n+1}-r^{n}}{\Dt_{n+1}}=\frac{1}{2\sqrt{E_{1}^{n}+C_{0}}}\int_{\Omega}F'(\phi^{*,n+\sigma})
\frac{\phi^{n+1}-\phi^{n}}{\Dt_{n+1}}d\x, \label{BDF_2_3}
 \end{align}
 \end{subequations}
 where $\phi^{n+\sigma}:=\sigma\phi^{n+1}+(1-\sigma)\phi^{n}$ and $\phi^{*,n+\sigma}:=\phi^{n}+\sigma\gamma_{n+1}(\phi^{n}-\phi^{n-1})$.
  Though only the first order finite difference approximation is used in \eqref{BDF_2_3}, we will see in \cref{sec:sect3}, the proposed scheme \eqref{BDF_2} is actually of second order accuracy for $\phi$ with $C_{0}\geq 1/\tau$.

Before rigorously proving the energy stability of the scheme \eqref{BDF_2}, we
need the following essential inequality, seeing also \cite{HQ21,HX22_1},
\brr\label{ideq1}
\dps \delta_{t}\phi^{n+1}\cdot F^{n+\sigma}_{2}\phi\geq\dps\Big[g(\gamma_{n+2})+\frac{1}{2}G(\gamma_{n+1},\gamma_{n+2})\Big]\frac{| \delta_{t}\phi^{n+1}|^{2}}{\tau_{n+1}}
-g(\gamma_{n+1})\frac{| \delta_{t}\phi^{n}|^{2}}{\tau_{n}},
\err
where $\delta_{t}\phi^{n+1}=\phi^{n+1}-\phi^{n}$. And in \eqref{ideq1},  for any $0<s,z\leq\gamma_{**}(\sigma)$, $$g(\gamma_{n+1}):=\frac{(2\sigma-1)\gamma_{n+1}^{\frac{3}{2}}}{2(1+\gamma_{n+1})}>0$$ and $$G(s,z):=\frac{2+4\sigma s-(2\sigma-1)s^{\frac{3}{2}}}{1+s}-\frac{(2\sigma-1)z^{\frac{3}{2}}}{1+z}\geq0.$$  Here $\gamma_{**}(\sigma)$ is the positive root of $G(z,z)=0$. Moreover,
the root function $\gamma_{**}(\sigma)$ is decreasing for $\sigma\in[\frac{1}{2},1]$ with $\gamma_{**}(1)\approx4.8645$ and $\gamma_{**}(\sigma)\rightarrow+\infty,$ as $\sigma\rightarrow\frac{1}{2}.$ Thus $\gamma_{**}(\sigma)\geq 4.8645$ for any $\sigma\in[\frac{1}{2},1].$
Furthermore, for $4\leq\gamma_{*}(\sigma)<\gamma_{**}(\sigma)$, it holds
\beq
G(s,z)\geq \min\{G(0,\gamma_{*}(\sigma)),G(\gamma_{*}(\sigma),\gamma_{*}(\sigma))\}\geq G(\gamma_{*}(\sigma),\gamma_{*}(\sigma))>0
\eeq
for any $0<s,z\leq\gamma_{*}(\sigma)<\gamma_{**}(\sigma)$.

To prove the following theorem on the unconditional energy stability of the  scheme \eqref{BDF_2},
we further introduce the inverse Laplace operator $\psi=(-\Delta)^{-1}u$ {\color{black}for $u\in L^{2}(\Omega)$ and $\int_{\Omega}u d\x=0$, as defined in \cite{WW11}:}
\bry
\begin{cases}
-\Delta \psi= u,\\
\dps\int_{\Omega}\psi d\x=0.
\end{cases}
\ery

\begin{theorem}\label{st_the}
For $1/2\leq\sigma\leq1$ and $0<\gamma_{n+1}\leq\gamma_{**}(\sigma), n=1,\cdots,M-1$,
 it holds for the scheme \eqref{BDF_2} that
\bq\label{eq2_1}
\dps \widetilde{E}^{n+1}-\widetilde{E}^{n}\leq0,
\eq
where the  discrete modified  energy $\widetilde{E}^{n}$ is defined by:
\beq
\dps \widetilde{E}^{n}:=
\dps\overline{E}(\phi^{n},r^{n})+\frac{(2\sigma-1)\gamma_{n+1}^{\frac{3}{2}}}{2+2\gamma_{n+1}}\frac{\|\nabla^{-1}(\phi^{n}-\phi^{n-1})\|^{2}}{\tau_{n}}.
\eeq
\end{theorem}
\begin{proof}
Taking the inner products of \eqref{BDF_2_1} and \eqref{BDF_2_2} with $(-\Delta)^{-1}\delta_{t}\phi^{n+1}$, $\delta_{t}\phi^{n+1}$ respectively, and multiplying \eqref{BDF_2_3} with $2r^{n+1}$, we obtain
\bry
&\dps\Big(\nabla^{-1}F^{n+\sigma}_{2}\phi,\nabla^{-1}\delta_{t}\phi^{n+1})\Big)=-(\mu^{n+\sigma},\delta_{t}\phi^{n+1}),\\[8pt]
 &\dps\Big(\mu^{n+\sigma},\delta_{t}\phi^{n+1}\Big)
 =\big((\Delta+\beta)\phi^{n+\sigma},(\Delta+\beta)\delta_{t}\phi^{n+1}\big)+S\big(\phi^{n+\sigma},\delta_{t}\phi^{n+1}\big)\\
 &\hspace{2.8cm}\dps+\frac{r^{n+1}}{\sqrt{E_{1}^{n}+C_{0}}}\Big(F'(\phi^{*,n+1}),\delta_{t}\phi^{n+1}\Big), \\[9pt]
 &\dps\frac{2r^{n+1}(r^{n+1}-r^{n})}{\Dt_{n+1}}=\frac{r^{n+1}}{\sqrt{E_{1}^{n}+C_{0}}}\int_{\Omega}F'(\phi^{*,n+1})
\frac{\delta_{t}\phi^{n+1}}{\Dt_{n+1}}d\x.
\ery
From the above equations, it follows that
\brr\label{eqt1}
&\Big(\nabla^{-1}F^{n+\sigma}_{2}\phi,\nabla^{-1}\delta_{t}\phi^{n+1}\Big)+\big((\Delta+\beta)\phi^{n+\sigma},(\Delta+\beta)\delta_{t}\phi^{n+1}\big)\\
\\
&+S\big(\phi^{n+\sigma},\delta_{t}\phi^{n+1}\big)+2r^{n+1}(r^{n+1}-r^{n})=0.
\err
Combining the inequality \eqref{ideq1} and
 \brr\label{Id}
&2a^{k+\sigma}(a^{k+1}-a^{k})
=|a^{k+1}|^{2}-|a^{k}|^{2}+(2\sigma-1)|a^{k}-a^{k-1}|^{2},\\[9pt]
&2a^{k+1}(a^{k+1}-a^{k})
=|a^{k+1}|^{2}-|a^{k}|^{2}+|a^{k+1}-a^{k}|^{2},
\err
we can easily derive the modified {discrete} energy law \eqref{eq2_1} {from \eqref{eqt1}}. Then we complete the proof.
\end{proof}

\subsection{Fully discrete scheme with Fourier spectral method}
In this work, we only present the results in the two dimensional space. The results in the three dimensional space can be obtained in a similar way.

Let $\Omega:=(0,L)^{2}$ and we use the uniform Fourier mode number of $2N+1$ for each direction. The trial and test function space of the Fourier spectral method is defined by
\beq
X_{N}:=\mbox{span}\{e^{i2(kx+ly)\pi/L}:-N\leq k,l \leq N\},
\eeq
where $i=\sqrt{-1}$.
By taking the Fourier spectral discretization in space, we can have the following fully discrete scheme of the semi-discrete  approach \eqref{BDF_2}: finding $\phi^{n+1}_{N}\in X_{N}$ and $r^{n+1}$ such that
 \begin{subequations}\label{BDF2_2}
 \begin{align}
 \dps (F^{n+\sigma}_{2}\phi_{N},p_{N})=&\dps-(\nabla\mu^{n+\sigma}_{N},\nabla p_{N}),~~~\forall p_{N}\in X_{N},\label{BDF2_2_1}\\
 \dps(\mu^{n+\sigma}_{N},q_{N})=&\dps\big((\Delta+\beta)\phi^{n+\sigma}_{N},(\Delta+\beta)q_{N})+S(\phi^{n+\sigma}_{N},q_{N})\nonumber\\
 &\dps+\frac{r^{n+1}}{\sqrt{E_{1,N}^{n}+C_{0}}}(F'(\phi^{*,n+\sigma}_{N}),q_{N}),~~~\forall q_{N}\in X_{N}\label{BDF2_2_2}\\
 \dps\frac{r^{n+1}-r^{n}}{\Dt_{n+1}}=&\dps\frac{1}{2\sqrt{E_{1,N}^{n}+C_{0}}}\int_{\Omega}F'(\phi^{*,n+\sigma}_{N})
\frac{\phi^{n+1}_{N}-\phi^{n}_{N}}{\Dt_{n+1}}d\x, \label{BDF2_2_3}
 \end{align}
 \end{subequations}
 where $E^{n}_{1,N}:=E_{1}(\phi^{n}_{N})$, $\phi^{n+\sigma}_{N}:=\sigma\phi^{n+1}_{N}+(1-\sigma)\phi^{n}_{N}$ and $\phi^{*n+\sigma}_{N}:=\phi^{n}_{N}+\sigma\gamma_{n+1}(\phi^{n}_{N}-\phi^{n-1}_{N})$.

 Since the proof of the unconditional energy stability for \eqref{BDF2_2} is essentially similar to that of the semi-discrete scheme \eqref{BDF_2}, we only state the following theorem and omit the proof.
  \begin{theorem}\label{st_full}
For the fully discrete  scheme \eqref{BDF2_2}, if $1/2\leq\sigma\leq1$ and $0<\gamma_{n+1}\leq\gamma_{**}(\sigma), n=1,2,\cdots,M-1$,
 we have
\bq\label{stab_ful2}
\dps \widetilde{E}^{n+1}_{N}-\widetilde{E}^{n}_{N}\leq0,
\eq
where $\widetilde{E}^{n}_{N}$ is defined by:
\beq
\dps \widetilde{E}^{n}_{N}:=\overline{E}(\phi^{n}_{N},{\color{black}r^{n}})+
\dps\frac{(2\sigma-1)\gamma_{n+1}^{\frac{3}{2}}}{2+2\gamma_{n+1}}\frac{\|\nabla^{-1}(\phi^{n}_{N}-\phi^{n-1}_{N})\|^{2}}{\tau_{n}}.
\eeq
\end{theorem}

\section{Error analysis}\label{sec:sect3}
\setcounter{equation}{0}
In this section, we establish a rigorous error estimate for the fully discrete  scheme \eqref{BDF2_2} with a mild smoothness requirement on the exact solution.
For simplicity, we only focus our attention on the investigation of the numerical solutions for the special case with $\sigma=1$ in what follows, and the obtained results are also right for the general parameter $\sigma\in[1/2, 1].$  We denote $\gamma_{*}:=\gamma_{*}(1)$.
From \eqref{stab_ful2}, we can deduce that for any $\phi_{0}(\x)\in H^{2}(\Omega)$ and $0<\gamma_{n}\leq\gamma_{*}$,
\bq\label{H2d}
\|(\Delta+\beta)\phi^{n}_{N}\|+\|\phi^{n}_{N}\|+|r^{n}|\leq \mathcal{M}, ~~~ n=1,2,\cdots,M
\eq
for the fully discrete second order scheme \eqref{BDF2_2}.
Here $\mathcal{M}$ is a positive constant only depending on $\Omega$ and the initial condition $\phi_{0}(\x)$.
This indicates the $H^{2}$ bound of the numerical solution and the uniform bound of the auxiliary variable $r^{n}$ for all $n=1,2,\cdots, M.$

We recall the $L^2$ orthogonal projection $\Pi_{N}$: $L^{2}(\Omega)\rightarrow X_{N}$, defined by
\bq\label{project}
(\Pi_{N}u-u,v)=0,~~\forall v\in X_{N}.
\eq
It is obvious that $\Pi_{N}u$ is the truncated Fourier series, namely
\beq
\Pi_{N}u=\sum_{k,l=-N}^{N}\widehat{u}_{k,l}e^{i2(kx+ly)\pi/L},
\eeq
with
\beq
\widehat{u}_{k,l}=(u,e^{i2(kx+ly)\pi/L})=\frac{1}{|\Omega|}\int_{\Omega}u(x,y) {\color{black}e^{-i2(kx+ly)\pi/L}}dxdy.
\eeq
For this $L^{2}$ orthogonal projection $\Pi_{N}$, we have the following optimal estimate, seeing also \cite{STW10},
\bq\label{proj_err}
\|\Pi_{N}u-u\|_{\mu}\leq C N^{\mu-m}\|u\|_{m},
\eq
for any $u\in H^{m}_{per}(\Omega)$ and $0\leq\mu\leq m$. Here $H^{m}_{per}(\Omega)$ is the subspace of $H^{m}(\Omega)$, which consists of functions with derivatives of order up to $m-1$ being $L$-periodic.
We denote $\overline{e}^{n}_{u}:=u^{n}_{N}-\Pi_{N}u(t_{n})$ with $\overline{e}^{0}_{u}=0$, $\widehat{e}^{n}_{u}:=\Pi_{N}u(t_{n})-u(t_{n})$, $e^{n}_{r}:=r^{n}-r(t_{n})$  and $\phi(t_{*,n+1}):=\phi(t_{n})+\gamma_{n+1}(\phi(t_{n})-\phi(t_{n-1})).$ Then we have
$$e^{n}_{\phi}:=\phi^{n}_{N}-\phi(t_{n})=\overline{e}^{n}_{\phi}+\widehat{e}^{n}_{\phi}, ~e^{n}_{\mu}:=\mu^{n}_{N}-\mu(t_{n})=\overline{e}^{n}_{\mu}+\widehat{e}^{n}_{\mu}.$$
For simplicity, we also denote $\delta_{t}\overline{e}^{n}_{u}:=\overline{e}^{n}_{u}-\overline{e}^{n-1}_{u}$.
The error estimates of $\phi^n_N$ and $r^n$ are derived in the following theorem.
\begin{theorem}\label{th1}
 Assume that $\phi^{0}:=\phi(\x,0)\in H^{2}(\Omega), 0<\gamma_{n}\leq\gamma_{*},$ $\Dt_{1}\leq C \Dt^{\frac{4}{3}}$, and $F(\cdot)\in C^{3}(\mathbb{R})$. Provided that
 \brr\label{sm_cd}
&\dps\phi\in L^{\infty}(0,T;H^{m}_{per}), \phi_{t}\in L^{\infty}(0,T;H^{m}_{per})\cap L^{2}(0,T;H^{1}),\\[7pt]
&\dps {\phi_{tt}\in L^{2}(0,T;H^{1}), \phi_{ttt}\in L^{2}(0,T;H^{-1})}
\err
for $m\geq2$, it holds
\brr\label{equa3}
&\dps\frac{\gamma^{\frac{3}{2}}_{n+2}}{2(1+\gamma_{n+2})}\frac{\|\nabla^{-1}(e^{n+1}_{\phi}-e^{n}_{\phi})\|^{2}}{\Dt_{n+1}}
+\frac{1}{2}\|(\Delta+\beta) e^{n+1}_{\phi}\|^{2}+\frac{S}{2}\|e^{n+1}_{\phi}\|^{2}+|e^{n+1}_{r}|^{2}\\[9pt]
\leq &\dps C\exp(T)\Big[\frac{\Dt^{2}}{C_{0}}\int_{0}^{T}\big(\|\phi_{t}(s)\|^{2}_{H^{1}}+\|\phi_{tt}(s)\|^{2}\big)ds+\Dt^{4}\int_{0}^{T}\big(\|\phi_{tt}(s)\|^{2}_{H^{1}}\\[9pt]
&\dps+\|\phi_{ttt}(s)\|^{2}_{H^{-1}}\big)ds+T\Big(N^{-2m}\|\phi_{t}\|^{2}_{L^{\infty}(0,T;H^{m})}+N^{4-2m}\|\phi\|_{L^{\infty}(0,T;H^{m})}\Big)\Big].
\err
Moreover, we have
\brr\label{equa4}
&\dps\frac{\gamma^{\frac{3}{2}}_{n+2}}{2(1+\gamma_{n+2})}\frac{\|\nabla^{-1}(e^{n+1}_{\phi}-e^{n}_{\phi})\|^{2}}{\Dt_{n+1}}
+\frac{1}{2}\|(\Delta+\beta)e^{n+1}_{\phi}\|^{2}+\frac{S}{2}\|e^{n+1}_{\phi}\|^{2}\\[11pt]
\leq &\dps C\exp(T)\Big[\frac{\Dt^{2}}{C^{2}_{0}}\int_{0}^{T}\big(\|\phi_{t}(s)\|^{2}_{H^{1}}+\|\phi_{tt}(s)\|^{2}\big)ds
+\Dt^{4}\int_{0}^{T}\big(\|\phi_{tt}(s)\|^{2}_{H^{1}}\\[11pt]
&\dps+\|\phi_{ttt}(s)\|^{2}_{H^{-1}}\big)ds+T\Big(N^{-2m}\|\phi_{t}\|^{2}_{L^{\infty}(0,T;H^{m})}+N^{4-2m}\|\phi\|_{L^{\infty}(0,T;H^{m})}\Big)\Big],
\err
and if we set the positive constant $C_{0}$ large enough such that $C_{0}\geq \frac{1}{\Dt}$, it holds
\brr\label{equa4_1}
&\dps\frac{\gamma^{\frac{3}{2}}_{n+2}}{2(1+\gamma_{n+2})}\frac{\|\nabla^{-1}(e^{n+1}_{\phi}-e^{n}_{\phi})\|^{2}}{\Dt_{n+1}}
+\frac{1}{2}\|(\Delta+\beta) e^{n+1}_{\phi}\|^{2}+\frac{S}{2}\|e^{n+1}_{\phi}\|^{2}\\[11pt]
\leq &\dps C\exp(T)\Big[\Dt^{4}
\int_{0}^{T}\big(\|\phi_{t}(s)\|^{2}_{H^{1}}+\|\phi_{tt}(s)\|^{2}_{H^{1}}+\|\phi_{ttt}(s)\|^{2}_{H^{-1}}\big)ds\\[11pt]
&\dps+T\Big(N^{-2m}\|\phi_{t}\|^{2}_{L^{\infty}(0,T;H^{m})}+N^{4-2m}\|\phi\|_{L^{\infty}(0,T;H^{m})}\Big)\Big].
\err

\end{theorem}
\begin{proof}
 From \eqref{H2c}, \eqref{H2d} and the Sobolev embedding theorem, $H^{2}\subset L^{\infty}$, it follows that $\|\phi(t)\|_{L^{\infty}},$ $\|\phi^{n}\|_{L^{\infty}}\leq C,$ where the positive constant $C$ only depends on $\phi^{0}, \Omega$ and $T$. Combining the definition of $F(\cdot)$, we have
\bq\label{equ1}
\|F(\cdot)\|_{L^{\infty}},\|F'(\cdot)\|_{L^{\infty}},\|F''(\cdot)\|_{L^{\infty}},\|F'''(\cdot)\|_{L^{\infty}}\leq C,
\eq
for $\phi$ and $\phi^{n}, n=0,1,2,\cdots, M$. In what follows, we
use $\mathcal{A}\lesssim\mathcal{B}$ to represent that $\mathcal{A}\leq C\mathcal{B}$, where the general positive constant C is independent of $C_{0}$.
 From the definition of $r(t)$ in \eqref{SAV}, we can derive
 \beq
 r_{tt}=-\frac{\big(\int_{\Omega}F'(\phi)\phi_{t}d\x\big)^{2}}{4\sqrt{(E_{1}(\phi)+C_{0})^{3}}}+\frac{\int_{\Omega}\big[F''(\phi)\phi^{2}_{t}+F'(\phi)\phi_{tt}\big]d\x}{2\sqrt{E_{1}(\phi)+C_{0}}}.
 \eeq
Moreover, together with \eqref{eq_c}, \eqref{sm_cd} and \eqref{equ1}, and using H\"{o}lder inequality, we can deduce that
\bq\label{equa1}
\int_{t_{n}}^{t_{n+1}}|r_{tt}|^{2}dt\lesssim \frac{1}{C_0}\int_{t_{n}}^{t_{n+1}}(\|\phi_{t}\|^{2}_{L^{4}}+\|\phi_{tt}\|^{2})dt\lesssim\frac{1}{C_0}\int_{t_{n}}^{t_{n+1}}(\|\phi_{t}\|^{2}_{H^{1}}+\|\phi_{tt}\|^{2})dt,
\eq
where we have used the Sobolev embedding theorem, $H^{1}\subset L^{4}.$
 From \eqref{re_prob}, \eqref{BDF_2} and the definition of the $L^{2}$ orthogonal projection $\Pi_{n}\phi(t_{n})$ in \eqref{project},  we derive the following error equation for $n\geq1$:
 \begin{subequations}\label{err}
 \begin{align}
 &\dps (F^{n+1}_{2}\overline{e}_{\phi},p_{N})+(\nabla\overline{e}^{n+1}_{\mu},\nabla p_{N})=(T_{1}^{n}, p_{N}),\label{err_1}\\
 &\dps(\overline{e}^{n+1}_{\mu}, q_{N})=\big((\Delta+\beta)\overline{e}^{n+1}_{\phi}, (\Delta+\beta)q_{N}\big)+S (\overline{e}^{n+1}_{\phi},q_{N})\nonumber\\
& \hspace{1.9cm}+\frac{e^{n+1}_{r}}{\sqrt{E^{n}_{1,N}+C_{0}}}\big(F'(\phi^{*,n+1}_{N}), q_{N}\big)+\dps (J^{n}_{1}+T^{n}_{2}, q_{N}),\label{err_2}\\
 &\dps e^{n+1}_{r}-e^{n}_{r}=\frac{1}{2\sqrt{E_{1}^{n}+C_{0}}}\int_{\Omega}F'(\phi^{*,n+1}_{N})(\overline{e}^{n+1}_{\phi}-\overline{e}^{n}_{\phi})d\x \nonumber\\
& \hspace{1.9cm}+\int_{\Omega}J^{n}_{2}\cdot
\Big(\Pi_{N}\big(\phi(t_{n+1})-\phi(t_{n})\big)\Big)d\x+J^{n}_{3}-v^{n}_{1}+v^{n}_{2},\label{err_3}
 \end{align}
 \end{subequations}
where 
\bry
J^{n}_{1}:=& \dps r(t_{n+1})\Big[\frac{F'(\phi^{*,n+1}_{N})}{\sqrt{E^{n}_{1,N}+C_{0}}}-\frac{F'(\phi(t_{*,n+1}))}{\sqrt{E_{1}(\phi(t_{n+1}))+C_{0}}}\Big],\\[11pt]
J^{n}_{2}:=&\dps \frac{1}{2}\Big[ \frac{F'(\phi^{*,n+1}_{N})}{\sqrt{E^{n}_{1,N}+C_{0}}}-\frac{F'(\phi(t_{n+1}))}{\sqrt{E_{1}(\phi(t_{n+1}))+C_{0}}}\Big],\\
J^{n}_{3}:=&\dps\int_{\Omega}\frac{F'(\phi(t_{n+1}))}{2\sqrt{E_{1}(\phi(t_{n+1}))+C_{0}}}(\widehat{e}^{n+1}_{\phi}-\widehat{e}^{n}_{\phi})d\x,
\ery
and the truncation errors are defined by
 \begin{subequations}\label{err1}
 \begin{align}
 T^{n}_{1}&\dps=\phi_{t}(t_{n+1})-\partial_{t}(\Pi_{2,n}\phi(t))|_{t=t_{n+1}}, ~T^{n}_{2}\dps=F'(\phi(t_{*,n+1}))-F'(\phi(t_{n+1}))\label{tr_err1}\\
v^{n}_{1}&\dps=r(t_{n+1})-r(t_{n})-\Dt_{n+1}r_{t}(t_{n+1})=\int_{t_{n}}^{t_{n+1}}(t_{n}-s)r_{tt}(s)ds,\label{tr_err3}\\
v^{n}_{2}&\dps=\Big(\frac{F'(\phi(t_{n+1}))}{2\sqrt{E_{1}(\phi(t_{n+1}))+C_{0}}},\int_{t_{n}}^{t_{n+1}}(t_{n}-s)\phi_{tt}(s)ds\Big).\label{tr_err4}
 \end{align}
 \end{subequations}
 Here, $\Pi_{2,n}\phi(t)$ denotes the quadratic interpolation polynomial associated with
$(t_{n-1},\phi(t_{n-1})),$ $(t_{n},\phi(t_{n}))$ and $(t_{n+1},\phi(t_{n+1}))$.
Taking $p_{N}=(-\Delta)^{-1}(\delta_{t}\overline{e}^{n+1}_{\phi})$ and $ q_{N}=\delta_{t}\overline{e}^{n+1}_{\phi}$ into \eqref{err_1} and \eqref{err_2}, respectively, and multiplying \eqref{err_3} with $2e^{n+1}_{r}$, we sum them up to derive
\bry
&\dps \big(\nabla^{-1}F^{n+1}_{2}\overline{e}_{\phi},\nabla^{-1}\delta_{t}\overline{e}^{n+1}_{\phi}\big)+\big((\Delta+\beta)\overline{e}^{n+1}_{\phi}, (\Delta+\beta)\delta_{t}\overline{e}^{n+1}_{\phi})+S (\overline{e}^{n+1}_{\phi},\delta_{t}\overline{e}^{n+1}_{\phi})\\[8pt]
&+2e^{n+1}_{r}(e^{n+1}_{r}-e^{n}_{r})\\[8pt]
=&\dps-(J^{n}_{1}+T^{n}_{2},\delta_{t}\overline{e}^{n+1}_{\phi})+\big(\nabla^{-1}T^{n}_{1},\nabla^{-1}\delta_{t}\overline{e}^{n+1}_{\phi}\big)\\
&\dps+2e^{n+1}_{r}\int_{\Omega}J^{n}_{2}\cdot \Pi_{N}
(\phi(t_{n+1})-\phi(t_{n}))d\x+2e^{n+1}_{r}J^{n}_{3}+2e^{n+1}_{r}(-v^{n}_{1}+v^{n}_{2}).
\ery
Using \eqref{ideq1},  the identities \eqref{Id}, and Young's inequality, we derive
\bry
&\dps\frac{1}{2}\Big[\Big(\frac{\gamma^{\frac{3}{2}}_{n+2}}{1+\gamma_{n+2}}+G(\gamma_{*},\gamma_{*})\Big)\frac{\|\nabla^{-1}\delta_{t}\overline{e}^{n+1}_{\phi}\|^{2}}{\Dt_{n+1}}-\frac{\gamma^{\frac{3}{2}}_{n+1}}{1+\gamma_{n+1}}\frac{\|\nabla^{-1}\delta_{t}\overline{e}^{n}_{\phi}\|^{2}}{\Dt_{n}}\Big]\\[9pt]
&\dps+\frac{1}{2}\big[\|(\Delta+\beta) \overline{e}^{n+1}_{\phi}\|^{2}-\|(\Delta+\beta) \overline{e}^{n}_{\phi}\|^{2}\big]+\frac{S}{2}\big[\|\overline{e}^{n+1}_{\phi}\|^{2}-\|\overline{e}^{n}_{\phi}\|^{2}\big]+|e^{n+1}_{r}|^{2}-|e^{n}_{r}|^{2}\\[11pt]
\leq&\dps C(\gamma_{*})\Dt_{n+1}\big[\|\nabla J^{n}_{1}\|^{2}+\|\nabla^{-1}T^{n}_{1}\|^{2}+\|\nabla T^{n}_{2}\|^{2}\big]+G(\gamma_{*},\gamma_{*})\frac{\|\nabla^{-1}\delta_{t}\overline{e}^{n+1}_{\phi}\|^{2}}{2\Dt_{n+1}}\\[11pt]
&\dps+C\Dt_{n+1}\|\phi_{t}\|_{L^{\infty}(0,T;L^{2})}\big[|e^{n+1}_{r}|^{2}+\|J^{n}_{2}\|^{2}\big]
+\frac{C\Dt_{n+1}}{\sqrt{C_{0}}}\big[|e^{n+1}_{r}|^{2}\\
&\dps+N^{-2m}\|\phi_{t}\|^{2}_{L^{\infty}(0,T;H^{m})}\big]+\Dt_{n+1}|e^{n+1}_{r}|^{2}+\frac{2}{\Dt_{n+1}}(|v^{n}_{1}|^{2}+|v^{n}_{2}|^{2}),
\ery
where we have used
\bry
&\dps 2e^{n+1}_{r}\cdot J^{n}_{3}\\[7pt]
\leq&\dps\frac{\|F'(\phi(t_{n+1}))\|_{L^{\infty}}}{\sqrt{E_{1}(\phi(t_{n+1}))+C_{0}}}\int_{\Omega}e^{n+1}_{r}(\widehat{e}^{n+1}_{\phi}-\widehat{e}^{n}_{\phi})d\x\\[9pt]
\leq&\dps\frac{C}{\sqrt{C_{0}}}\Big[\Dt_{n+1}|e^{n+1}_{r}|^{2}+\frac{1}{\Dt_{n+1}}\int_{\Omega}|\widehat{e}^{n+1}_{\phi}-\widehat{e}^{n}_{\phi}|^{2}d\x\Big]\\[9pt]
=&\dps\frac{C}{\sqrt{C_{0}}}\Big[\Dt_{n+1}|e^{n+1}_{r}|^{2}+\frac{1}{\Dt_{n+1}}\int_{\Omega}\big|\Pi_{N}\big(\phi(t_{n+1})-\phi(t_{n})\big)-\big(\phi(t_{n+1}-\phi(t_{n}))\big)\big|^{2}d\x\Big]\\[9pt]
\leq&\dps\frac{C}{\sqrt{C_{0}}}\Big[\Dt_{n+1}|e^{n+1}_{r}|^{2}+\Dt_{n+1}N^{-2m}\|\phi_{t}\|^{2}_{L^{\infty}(0,T;H^{m})}\Big].
\ery
Combining with \eqref{sm_cd}, we have following error inequality
\brr\label{equ2}
&\dps\frac{1}{2}\Big[\frac{\gamma^{\frac{3}{2}}_{n+2}}{1+\gamma_{n+2}}\frac{\|\nabla^{-1}\delta_{t}\overline{e}^{n+1}_{\phi}\|^{2}}{\Dt_{n+1}}
-\frac{\gamma^{\frac{3}{2}}_{n+1}}{1+\gamma_{n+1}}\frac{\|\nabla^{-1}\delta_{t}\overline{e}^{n}_{\phi}\|^{2}}{\Dt_{n}}\Big]+\frac{1}{2}\big[\|(\Delta+\beta) \overline{e}^{n+1}_{\phi}\|^{2}\\[9pt]
&\dps-\|(\Delta+\beta) \overline{e}^{n}_{\phi}\|^{2}\big]+\frac{S}{2}\big[\|\overline{e}^{n+1}_{\phi}\|^{2}-\|\overline{e}^{n}_{\phi}\|^{2}\big]+|e^{n+1}_{r}|^{2}-|e^{n}_{r}|^{2}\\[11pt]
\lesssim&\dps \Dt_{n+1}\Big[\|\nabla J^{n}_{1}\|^{2}+\|J^{n}_{2}\|^{2}+|e^{n+1}_{r}|^{2}+\|\nabla^{-1}T^{n}_{1}\|^{2}+\|\nabla T^{n}_{2}\|^{2}\\
&\dps+N^{-2m}\|\phi_{t}\|^{2}_{L^{\infty}(0,T;H^{m})}\Big]+\frac{2(|v^{n}_{1}|^{2}+|v^{n}_{2}|^{2})}{\Dt_{n+1}}.
\err
Moreover, we have
\bq\label{equat1}
\dps \nabla J^{n}_{1}
=-r(t_{n+1})\Big[\frac{\nabla I^{n}_{1}}{\sqrt{E^{n}_{1,N}+C_{0}}}+\nabla F'(\phi(t_{*,n+1}))\cdot I^{n}_{2}\Big],
\eq
where
\bq\label{equ6}
I^{n}_{1}:= F'(\phi^{*,n+1}_{N})- F'(\phi(t_{*,n+1})),\ \ I^{n}_{2}:=\frac{1}{\sqrt{E^{n}_{1,N}+C_{0}}}-\frac{1}{\sqrt{E_{1}(\phi(t_{n+1}))+C_{0}}}.
\eq
From \eqref{equ1} and $\phi(t_{*,n+1})\in H^{2}$, it follows that
\bry
\|\nabla F'(\phi(t_{*,n+1}))\|=&\dps\|F^{''}(\phi(t_{*,n+1}))\nabla\phi(t_{*,n+1})\|\\[7pt]
\leq&\dps\|F^{''}(\phi(t_{*,n+1}))\|_{L^{\infty}}\|\nabla\phi(t_{*,n+1})\|\leq C.
\ery
Therefore, together with \eqref{SAV} and \eqref{eq_c}, we deduce that
\bq\label{equ3}
\|\nabla J^{n}_{1}\|^{2}\lesssim[\|\nabla I^{n}_{1}\|^{2}+C_{0}|I^{n}_{2}|^{2}].
\eq
For the term $J^{n}_{2}$, using triangle inequality and \eqref{eq_c}, we have
\brr\label{equ4}
\|J^{n}_{2}\|^{2}=&\dps\frac{1}{2}\Big\|\frac{I^{n}_{1}}{\sqrt{E^{n}_{1,N}+C_{0}}}
+\frac{T^{n}_{2}}{\sqrt{E^{n}_{1,N}+C_{0}}}+F'(\phi(t_{n+1}))\cdot I^{n}_{2}\Big\|^{2}\\[11pt]
\lesssim &\dps \frac{\|I^{n}_{1}\|^{2}}{C_{0}}+\frac{\|T^{n}_{2}\|^{2}}{C_{0}}+|I^{n}_{2}|^{2}.\\[9pt]
\err
Then, we substitute the estimates \eqref{equ3} and \eqref{equ4} into the error inequality \eqref{equ2} to get
\brr\label{equ5}
&\dps\frac{\gamma^{\frac{3}{2}}_{n+2}}{1+\gamma_{n+2}}\frac{\|\nabla^{-1}\delta_{t}\overline{e}^{n+1}_{\phi}\|^{2}}{2\Dt_{n+1}}
-\frac{\gamma^{\frac{3}{2}}_{n+1}}{1+\gamma_{n+1}}\frac{\|\nabla^{-1}\delta_{t}\overline{e}^{n}_{\phi}\|^{2}}{2\Dt_{n}}+\frac{1}{2}\big[\|(\Delta+\beta)\overline{e}^{n+1}_{\phi}\|^{2}\\[9pt]
&\dps-\|(\Delta+\beta) \overline{e}^{n}_{\phi}\|^{2}\big]+\frac{S}{2}\big[\|\overline{e}^{n+1}_{\phi}\|^{2}-\|\overline{e}^{n}_{\phi}\|^{2}\big]+|e^{n+1}_{r}|^{2}-|e^{n}_{r}|^{2}\\[9pt]
\lesssim&\dps \Dt_{n+1}\Big[|e^{n+1}_{r}|^{2}+\frac{1}{C_{0}}\| I^{n}_{1}\|^{2}+\|\nabla I^{n}_{1}\|^{2}+C_{0}|I^{n}_{2}|^{2}
+\|\nabla^{-1}T^{n}_{1}\|^{2}+\frac{1}{C_{0}}\|T^{n}_{2}\|^{2}\\
&\dps+\|\nabla T^{n}_{2}\|^{2}+N^{-2m}\|\phi_{t}\|^{2}_{L^{\infty}(0,T;H^{m})}\Big]+\frac{2}{\Dt_{n+1}}(|v^{n}_{1}|^{2}+|v^{n}_{2}|^{2}).
\err
Moreover, from \eqref{equ6}, \eqref{equ1}, \eqref{proj_err} and \eqref{eq_c},
\brr\label{equ7}
\|I^{n}_{1}\|^{2}=&\dps \|F'(\phi^{*,n+1}_{N})-F'(\phi(t_{*,n+1}))\|^{2}\lesssim \|e^{n}_{\phi}\|^{2}+\|e^{n-1}_{\phi}\|^{2},\\[8pt] \lesssim&\|\overline{e}^{n}_{\phi}+\widehat{e}^{n}_{\phi}\|^{2}+\|\overline{e}^{n-1}_{\phi}+\widehat{e}^{n-1}_{\phi}\|^{2}\\[8pt]
\lesssim&\dps \|\overline{e}^{n}_{\phi}\|^{2}+\|\overline{e}^{n-1}_{\phi}\|^{2}+N^{-2m}\|\phi\|_{L^{\infty}(0,T;H^{m})},\\[8pt]
|I^{n}_{2}|^{2}=&\dps\Big|\frac{E_{1}(\phi(t_{n+1}))-E^{n}_{1,N}}{\mathcal{G}(C_{0})}\Big|^{2}
\lesssim\dps \frac{1}{C^{3}_{0}}\big|E_{1}(\phi(t_{n+1}))-E^{n}_{1,N}\big|^{2}\\[8pt]
\lesssim&\dps  \frac{1}{C^{3}_{0}}\Big(\big|E_{1}(\phi(t_{n+1}))-E_{1}(\phi(t_{n}))\big|^{2}+\big|E_{1}(\phi(t_{n}))-E^{n}_{1,N}\big|^{2}\Big)\\[8pt]
\lesssim&\dps  \frac{1}{C^{3}_{0}}\Big(\Dt_{n+1}\int_{t_{n}}^{t_{n+1}}\|\phi_{t}(s)\|^{2}ds+\|e^{n}_{\phi}\|^{2}\Big)\\[8pt]
\lesssim&\dps  \frac{1}{C^{3}_{0}}\Big(\Dt_{n+1}\int_{t_{n}}^{t_{n+1}}\|\phi_{t}(s)\|^{2}ds+\|\overline{e}^{n}_{\phi}\|^{2}+N^{-2m}\|\phi\|_{L^{\infty}(0,T;H^{m})}\Big),
\err
where the $\mathcal{G}(C_{0})$ is defined by $$\mathcal{G}(C_{0}):=\sqrt{E^{n}_{1,N}+C_{0}}\sqrt{E_{1}(\phi(t_{n+1}))+C_{0}}\Big(\sqrt{E_{1}(\phi(t_{n+1}))+C_{0}}+\sqrt{E^{n}_{1,N}+C_{0}}\Big).$$
From the definition of $I^{n}_{1}$ in \eqref{equ6}, we have
\bq\label{est_gard_I}
\|\nabla I^{n}_{1}\|^{2}=\dps \|I^{n}_{1,1}+I^{n}_{1,2}\|^{2}\leq\|I^{n}_{1,1}\|^{2}+\|I^{n}_{1,2}\|^{2},
\eq
where $I^{n}_{1,1}$ and $I^{n}_{1,2}$ are defined by
\bry
&\dps I^{n}_{1,1}:=F''(\phi^{*,n+1}_{N})\nabla\big(\phi^{*,n+1}_{N}-\phi(t_{*,n+1})\big), \\[7pt]
&\dps I^{n}_{1,2}:=\big(F''(\phi^{*,n+1}_{N})-F''(\phi(t_{*,n+1}))\big)\nabla\phi(t_{*,n+1}).
\ery
Furthermore, using \eqref{equ1}, the H\"{o}lder's inequality and the Sobolev embedding theorem, $H^{1}\subset L^{6}$, we have
\brr\label{est_gard_I1}
\|I^{n}_{1,1}\|^{2}\lesssim&\dps \|\nabla e^{n}_{\phi}\|^{2}+\|\nabla e^{n-1}_{\phi}\|^{2}\\[7pt]
\|I^{n}_{1,2}\|^{2}\lesssim&\|(\phi^{*,n+1}_{N}-\phi(t_{*,n+1}))\nabla\phi(t_{*,n+1})\|^{2}\\[7pt]
\lesssim&\|\phi^{*,n+1}_{N}-\phi(t_{*,n+1})\|^{2}_{L^{6}}\|\nabla\phi(t_{*,n+1})\|^{2}_{L^{3}}\\[7pt]
\lesssim&\|\phi^{*,n+1}_{N}-\phi(t_{*,n+1})\|^{2}_{H^{1}}\|\phi(t_{*,n+1})\|^{2}_{H^{2}}\\[7pt]
\lesssim&\|e^{n}_{\phi}\|^{2}_{H^{1}}+\|e^{n-1}_{\phi}\|^{2}_{H^{1}}.
\err
Combining \eqref{est_gard_I} and \eqref{est_gard_I1} gives
\brr\label{est_I}
\|\nabla I^{n}_{1}\|^{2}
\lesssim&\dps \|\nabla e^{n}_{\phi}\|^{2}+\|\nabla e^{n-1}_{\phi}\|^{2}+\|e^{n}_{\phi}\|^{2}+\|e^{n-1}_{\phi}\|^{2}\\[7pt]
\lesssim&\dps \|\Delta \overline{e}^{n}_{\phi}\|^{2}+\|\Delta \overline{e}^{n-1}\|^{2}+\|\overline{e}^{n}_{\phi}\|^{2}+\|\overline{e}^{n-1}_{\phi}\|^{2}+N^{4-2m}\|\phi\|^{2}_{L^{\infty}(0,T;H^{m})},\\[7pt]
\lesssim&\dps \|(\Delta+\beta) \overline{e}^{n}_{\phi}\|^{2}+\|(\Delta+\beta) \overline{e}^{n-1}_{\phi}\|^{2}+\|\overline{e}^{n}_{\phi}\|^{2}+\|\overline{e}^{n-1}_{\phi}\|^{2}+N^{4-2m}\|\phi\|^{2}_{L^{\infty}(0,T;H^{m})}
\err
in which we have used
\bry
&\dps\|\nabla e^{n+1}_{\phi}\|^{2}=(-\Delta e^{n+1}_{\phi},e^{n+1}_{\phi})\leq \frac{\|\Delta e^{n+1}_{\phi}\|^{2}+\|e^{n+1}_{\phi}\|^{2}}{2},\\[7pt]
&\dps \|\Delta e^{n+1}_{\phi}\|^{2}=\|\Delta \overline{e}^{n+1}_{\phi}+\Delta\widehat{e}^{n+1}_{\phi}\|^{2}\leq\|\Delta \overline{e}^{n+1}_{\phi}\|^{2}+ CN^{4-2m}\|\phi\|^{2}_{L^{\infty}(0,T;H^{m})}.
\ery
For the truncation errors of $T^{n}_{1}, T^{n}_{2}, v^{n}_{1}$, and $v^{n}_{2}$, the following estimates hold, seeing also \cite{CWYZ19,SX18},
\brr\label{equ8}
\|\nabla^{-1}T^{n}_{1}\|^{2}\lesssim &\dps (\Dt_{n}+\Dt_{n+1})^{3}\int_{t_{n-1}}^{t_{n+1}}\|\phi_{ttt}(s)\|^{2}_{H^{-1}}ds,\\
\|T^{n}_{2}\|^{2}\lesssim&\dps \|\phi(t_{n+1})-\phi(t_{*,n+1})\|^{2}\lesssim(\Dt_{n}+\Dt_{n+1})^{3}\int_{t_{n-1}}^{t_{n+1}}\|\phi_{tt}(s)\|^{2}ds,\\
\|\nabla T^{n}_{2}\|^{2}\lesssim&\dps \|\phi(t_{n+1})-\phi(t_{*,n+1})\|^{2}+\|\nabla(\phi(t_{n+1})-\phi(t_{*,n+1}))\|^{2}\\[9pt]
\lesssim&\dps (\Dt_{n}+\Dt_{n+1})^{3}\int_{t_{n-1}}^{t_{n+1}}\|\phi_{tt}(s)\|^{2}_{H^{1}}ds,\\
|v^{n}_{1}|^{2}\leq&\dps \frac{1}{3} \Dt^{3}_{n+1}\int_{t_{n}}^{t_{n+1}}|r_{tt}(s)|^{2}ds\lesssim\frac{\Dt^{3}_{n+1}}{C_0}\int_{t_{n}}^{t_{n+1}}(\|\phi_{t}\|^{2}_{H^{1}}+\|\phi_{tt}\|^{2})dt,\\[11pt]
|v^{n}_{2}|^{2}\lesssim&\dps \Dt_{n+1}^{3}\Big\|\frac{F'(\phi(t_{n+1}))}{2\sqrt{E_{1}(\phi(t_{n+1}))+C_{0}}}\Big\|^{2}\int_{t_{n}}^{t_{n+1}}\|\phi_{tt}(s)\|^{2}ds\\[11pt]
\lesssim&\dps \frac{\Dt^{3}_{n+1}}{C_{0}}\int_{t_{n}}^{t_{n+1}}\|\phi_{tt}(s)\|^{2}ds,
\err
where the estimate for $r_{tt}$ in \eqref{equa1} is used.
Then we take the above estimates \eqref{equ7}-\eqref{equ8} into \eqref{equ5},
and sum up it from 1 to $n$ to get
\brr\label{equ9}
&\dps\frac{\gamma^{\frac{3}{2}}_{n+2}}{2(1+\gamma_{n+2})}\frac{\|\nabla^{-1}\delta_{t}\overline{e}^{n+1}_{\phi}\|^{2}}{\Dt_{n+1}}
+\frac{1}{2}\|(\Delta+\beta) \overline{e}^{n+1}_{\phi}\|^{2}+\frac{S}{2}\|\overline{e}^{n+1}_{\phi}\|^{2}+|e^{n+1}_{r}|^{2}\\[11pt]
\lesssim&\dps\frac{\gamma^{\frac{3}{2}}_{2}}{2(1+\gamma_{2})}\frac{\|\nabla^{-1}\overline{e}^{1}_{\phi}\|^{2}}{\Dt_{1}}+\frac{1}{2}\|(\Delta+\beta) \overline{e}^{1}_{\phi}\|^{2} +\frac{S}{2}\|\overline{e}^{1}_{\phi}\|^{2}+|e^{1}_{r}|^{2}\\[8pt]
&\dps+\sum_{k=1}^{n}\Dt_{k+1}\big[|e^{n+1}_{r}|^{2}
+\|(\Delta+\beta)\overline{e}^{n}_{\phi}\|^{2}+\|(\Delta+\beta)\overline{e}^{n-1}_{\phi}\|^{2}+\|\overline{e}^{n}_{\phi}\|^{2}\\[11pt]
&\dps+\|\overline{e}^{n-1}_{\phi}\|^{2}\big]
+t_{n+1}\big[N^{-2m}\|\phi_{t}\|^{2}_{L^{\infty}(0,T;H^{m})}+N^{4-2m}\|\phi\|_{L^{\infty}(0,T;H^{m})}\big]\\[8pt]
&\dps+\frac{\Dt^{2}}{C_{0}}\int_{t_{1}}^{t_{n+1}}\big(\|\phi_{t}(s)\|^{2}_{H^{1}}+\|\phi_{tt}(s)\|^{2}\big)ds
+\Dt^{4}\int_{0}^{t_{n+1}}\big(\|\phi_{tt}(s)\|^{2}_{H^{1}}+\|\phi_{ttt}(s)\|^{2}_{H^{-1}}\big)ds.
\err
We use the first order scheme \eqref{eq1} for the initial time step to get the following error equation for $n=0$:
 \bry
  &\dps(\frac{\overline{e}^{1}_{\phi}}{\Dt_{1}},p_{N})+(\nabla \overline{e}^{1}_{\mu},\nabla p_{N})=-(\Dt_{1}^{-1}\int_{0}^{t_{1}}s\phi_{tt}(s)ds,p_N),\\
  &\dps(\overline{e}^{1}_{\mu}, q_{N})=\big((\Delta+\beta)\overline{e}^{1}_{\phi}, (\Delta+\beta)q_{N}\big)+S (\overline{e}^{1}_{\phi}, q_{N})+\frac{e^{1}_{r}}{\sqrt{E^{0}_{1,N}+C_{0}}}(F'(\phi^{0}_{N}), q_{N})\\[11pt]
  &\hspace{1.1cm}\dps+\Big(\frac{r(t_{1})F'(\phi^{0}_{N})}{\sqrt{E^{0}_{1,N}+C_{0}}}-\frac{r(t_{1})F'(\phi_{0}(\x))}{\sqrt{E_{1}(\phi(t_{1}))+C_{0}}}, q_{N}\Big) -\Big(F'(\phi_{0}(\x))-F'(\phi(t_{1})),q_{N}\Big),\\[13pt]
 &\dps e^{1}_{r}=\frac{\int_{\Omega}F'(\phi^{0}_{N})\overline{e}^{1}_{\phi}d\x}{2\sqrt{E^{0}_{1,N}+C_{0}}}
 +\frac{1}{2}\int_{\Omega}\Big[\frac{F'(\phi^{0}_{N})}{\sqrt{E^{0}_{1,N}+C_{0}}}-\frac{F'(\phi(t_{1}))}{\sqrt{{\color{black}E_{1}(\phi(t_{1}))+C_{0}}}}\Big]
\Pi_{N}\big(\phi(t_{1})-\phi(t_{0})\big)d\x\\[13pt]
&\hspace{1cm}\dps+\int_{\Omega}\frac{F'(\phi(t_{1}))}{\sqrt{{\color{black}E_{1}(\phi(t_{1}))+C_{0}}}}
(\widehat{e}^{1}_{\phi}-\widehat{e}^{0}_{\phi})d\x-v^{0}_{1}+v^{0}_{2},
\ery
where $\phi^{0}_{N}:=\Pi_{N}\phi_{0}(\x).$
Similar to the argument for the case $n\geq1$, the following estimate can be easily derived with the assumption $\Dt_{1}\leq C \Dt^{4/3}$
\brr\label{equa2}
 \dps\frac{\|\nabla^{-1}\overline{e}^{1}_{\phi}\|^{2}}{2\Dt_{1}}+\frac{1}{2}\|(\Delta+\beta)\overline{ e}^{1}_{\phi}\|^{2}+\frac{S}{2} \|\overline{e}^{1}_{\phi}\|^{2}+|e^{1}_{r}|^{2}
 \lesssim&\dps \Dt_{1}^{3}+\Dt_{1}N^{4-2m}\|\phi_{0}(\x)\|_{H^{m}}\\[8pt]
 \lesssim&\dps \Dt^{4}+\Dt_{1}N^{4-2m}\|\phi\|_{L^{\infty}(0,T;H^{m})}.
\err
Combining with \eqref{equ9},\eqref{equa2} and \eqref{proj_err}, together with the discrete Gronwall's lemma and the triangle inequality, the desired estimate \eqref{equa3} is obtained.

Next, we will prove the estimate \eqref{equa4}, which can show that the numerical solution $\phi^{n}$ is of second-order accuracy in time with $C_{0}\geq {1}/{\Dt}$. The error equation for $n\geq1$ reads
\bry
&(F^{n+1}_{2}\overline{e}_{\phi},p_{N})+(\nabla(\Delta+\beta)\overline{e}^{n+1}_{\phi}, \nabla(\Delta+\beta)p_{N})+S (\nabla\overline{e}^{n+1}_{\phi},\nabla p_{N})\\[9pt]
=&\dps\big(1-\frac{r^{n+1}}{\sqrt{E^{n}_{1,N}+C_{0}}}\big)(\nabla F'(\phi^{*,n+1}_{N}), \nabla p_{N})-\big(\nabla (I^{n}_{1}+T^{n}_{2}),\nabla p_{N}\big)-(T^{n}_{1}, p_{N}).
\ery
Taking  $p_{N}=(-\Delta)^{-1}(\delta_{t}\overline{e}^{n+1}_{\phi})$, and using the Young's inequality, we have
\brr\label{esti_phi}
&\dps\frac{1}{2}\Big[\big(\frac{\gamma^{\frac{3}{2}}_{n+2}}{1+\gamma_{n+2}}+G(\gamma_{*},\gamma_{*})\big)\frac{\|\nabla^{-1}\delta_{t}\overline{e}^{n+1}_{\phi}\|^{2}}{\Dt_{n+1}}-\frac{\gamma^{\frac{3}{2}}_{n+1}}{1+\gamma_{n+1}}\frac{\|\nabla^{-1}\delta_{t}\overline{e}^{n}_{\phi}\|^{2}}{\Dt_{n}}\Big]\\[9pt]
&\dps+\frac{1}{2}\big[\|(\Delta+\beta)\overline{e}^{n+1}_{\phi}\|^{2}-\|(\Delta+\beta) \overline{e}^{n}_{\phi}\|^{2}\big]+\frac{S}{2}\big[\|\overline{e}^{n+1}_{\phi}\|^{2}-\|\overline{e}^{n}_{\phi}\|^{2}\big]\\[11pt]
\leq&\dps C(\gamma_{*})\Dt_{n+1}\Big[\Big|1-\frac{r^{n+1}}{\sqrt{E^{n}_{1,N}+C_{0}}}\Big|^{2}+\|\nabla I^{n}_{1}\|^{2}+\|\nabla^{-1}T^{n}_{1}\|^{2}+\|\nabla T^{n}_{2}\|^{2}\Big]\\
&\dps+G(\gamma_{*},\gamma_{*})\frac{\|\nabla^{-1}\delta_{t}\overline{e}^{n+1}_{\phi}\|^{2}}{2\Dt_{n+1}}.
\err
Combined the estimate \eqref{equ7} for $I^{n}_{2}$ with \eqref{equa3}, it gives
\brr\label{equa5}
&\dps\Big|1-\frac{r^{n+1}}{\sqrt{E^{n}_{1,N}+C_{0}}}\Big|^{2}\\
=&\dps\Big|\frac{r^{n+1}}{\sqrt{E^{n}_{1,N}+C_{0}}}-\frac{r(t_{n+1})}{\sqrt{E_{1}(\phi(t_{n+1}))+C_{0}}}\Big|^{2}= \Big|\frac{e^{n+1}_{r}}{\sqrt{E^{n}_{1,N}+C_{0}}}+r(t_{n+1})\cdot I^{n}_{2}\Big|^{2}\\[11pt]
\lesssim&\dps \frac{|e^{n+1}_{r}|^{2}}{C_{0}}+ \frac{1}{C^{2}_{0}}\Big(\Dt_{n+1}\int_{t_{n}}^{t_{n+1}}\|\phi_{t}(s)\|^{2}ds+\|\overline{e}^{n}_{\phi}\|^{2}+N^{-2m}\|\phi\|_{L^{\infty}(0,T;H^{m})}\Big)\\
\lesssim&\dps \frac{\Dt^{2}}{C^{2}_{0}}\int_{0}^{T}\big(\|\phi_{t}(s)\|^{2}_{H^{1}}+\|\phi_{tt}(s)\|^{2}\big)ds+\frac{\Dt^{4}}{C_{0}}\int_{0}^{T}\big(\|\phi_{tt}(s)\|^{2}_{H^{1}}+\|\phi_{ttt}(s)\|^{2}_{H^{-1}}\big)ds \\
&\dps+\frac{1}{C^{2}_{0}}\Big(\Dt_{n+1}\int_{t_{n}}^{t_{n+1}}\|\phi_{t}(s)\|^{2}ds+\|\overline{e}^{n}_{\phi}\|^{2}+N^{-2m}\|\phi\|_{L^{\infty}(0,T;H^{m})}\Big).
\err
Consequently, taking \eqref{equa5}, the estimates of $\|I^{n}_{1}\|^{2}$, $\|T^{n}_{1}\|^{2}$ and $\|T^{n}_{2}\|^{2}$ into \eqref{esti_phi}, summing it up from 1 to n, and applying the discrete Gronwall's lemma, we can deduce the desired estimate \eqref{equa4} and \eqref{equa4_1}. Then the proof is completed.
\end{proof}

\begin{remark}
Our error norms contain the terms $\|\nabla^{-1}e^{1}_{\phi}\|^{2}/\Dt_{1}$ for $n=0$ in \eqref{equa2} and $\|\nabla^{-1}\delta_{t}\overline{e}^{n+1}_{\phi}\|/\Dt_{n+1}$ for $n\geq1$ in \eqref{equa4_1}. Thus, in order to guarantee the second order accuracy of the fully discrete  scheme \eqref{BDF2_2} with this type of error norm in \eqref{equa4_1}, it is necessary to employ the first order scheme \eqref{eq1} with $\Dt_{1}\leq \Dt^{4/3}$ to start the time marching. {\color{black}It is worth noting that the constant $C \exp(T)$ in the error estimates (3.5)-(3.7) would be rather large, which may be not optimal estimates particularly for the long-time simulation case. This exponential constant is commonly appeared in the error analysis for the time-stepping schemes using the discrete Gronwall's inequality as an analysis tool.
To avoid using the discrete Gronwall's inequality, another possible way to establish the corresponding error estimates is to adopt the mathematical induction method. However, it is not easy to give a detailed uniform bound for the general positive constant $C$ such as $C\leq 1$ in the derivation of the error estimate, which is a key issue in successfully deriving the error estimates by the mathematical induction approach. It is also an open, interesting problem now and needs more further effort.}
\end{remark}

\section{Numerical results}\label{sec:sect4}
\setcounter{equation}{0}
This section is devoted to numerically validating the theoretical results in terms of stability and accuracy. For simplicity, we set $\beta=1, S=\varepsilon,$  and the errors are computed by $L^{\infty}$-norm throughout the numerical tests.

\subsection{Test of the convergence order}
In this subsection, we investigate the accuracy of second order in time and spectral convergence in space for the fully discrete  scheme \eqref{BDF2_2}. Here, we consider the PFC model with $\varepsilon=0.025$ and the initial data, given by
\beq
\phi_{0}(x,y)=\sin\Big(\frac{\pi x}{16}\Big)\cos\Big(\frac{\pi y}{16}\Big).
\eeq
The computational domain is set to be $(0,32)\times(0,32)$.
At $T=1$, a reference solution of the phase function $\phi^n$ is computed by the fully discrete scheme \eqref{BDF2_2} with a fixed small uniform time step size $\Delta t=1e-5$ and a $256\times256$ space mesh.
We first check the temporal accuracy on the uniform temporal mesh.
In Figure \ref{fig1} (a), the temporal errors of the numerical solutions of the phase function $\phi^n$ with different time step sizes are plotted in log-log scale.
It is shown that the scheme \eqref{BDF2_2} with $C_{0}=1/\Dt$ achieves the expected second order accuracy in time.
We also display the second order approximation of the computed $r^{n+1}/\sqrt{E^{n}_{1,N}+C_{0}}$ to $1$ in Figure \ref{fig1} (b).
 Then we study the accuracy of the spatial discretization.
 In Figure \ref{fig1} (c), we plot the $L^{\infty}$-errors in semi-log scale with respect to the Fourier mode number $N$.
 The error curve is almost a straight line in the semi-log scale, which indicates that the numerical solutions
are exponentially convergent with respect to $N$. This is consistent with our theoretical results.

Next, we investigate the temporal error behaviors of the fully discrete  scheme \eqref{BDF2_2} on the nonuniform temporal mesh to validate our theoretical error estimates in Theorem \ref{th1}.
{\color{black}The nonuniform temporal mesh $\{\widehat{t}_{n}\}_{n=0}^{N}$ used here is produced  by $40\%$ random  perturbation of the uniform mesh $\{t_{n}=n\tau\}_{n=0}^{N}$, i.e., the random time grids are given by
 $$\widehat{t}_{n}:=t_{n}+0.4\tau\cdot rand(\cdot),\quad n=0,1,\cdots,N,$$
 where $rand(\cdot)$ is a random data from -1 to 1.
The convergence order is computed by
\beq
\mbox{Order}=\frac{\log_{10}(\|e^{M}_{\phi}\|_{\infty}/\|e^{2M}_{\phi}\|_{\infty})}{\log_{10}(\tau(M)/\tau(2M))},
\eeq
where $e^{M}_{\phi}$ and $\tau(M)$ denote the $L^{\infty}$-norm error of $\phi$ at the finial time T and the maximum time-step size for a total of M subintervals on $[0,T]$, respectively.}
In Table \ref{table1}, we present convergence rates in time for the phase function $\phi^n$ in \eqref{BDF2_2} at $T=1$ with $\sigma=$ 1/2,  2/3 and 1.
It is observed the proposed  scheme \eqref{BDF2_2} achieves {\color{black} about} second order accuracy for the phase function $\phi^n$ for all tested $\sigma$ on the nonuniform mesh, even for the cases with the adjacent time step radio $\gamma_{n}>4.8645$.
Thus it indicates that $\gamma_{**}=4.8645$ may be not the optimal upper bound on the adjacent time step ratios, and deserves further investigation.
{\color{black}In addition, compared to the numerical convergence results for the case of the uniform temporal meshes in Figure \ref{fig1}, the convergence order of the proposed scheme shown in Table \ref{table1} achieves obvious oscillations due to the randomness of the random nonuniform meshes used.}

\begin{figure*}[htbp]
\begin{minipage}[t]{0.32\linewidth}
\centerline{\includegraphics[scale=0.28]{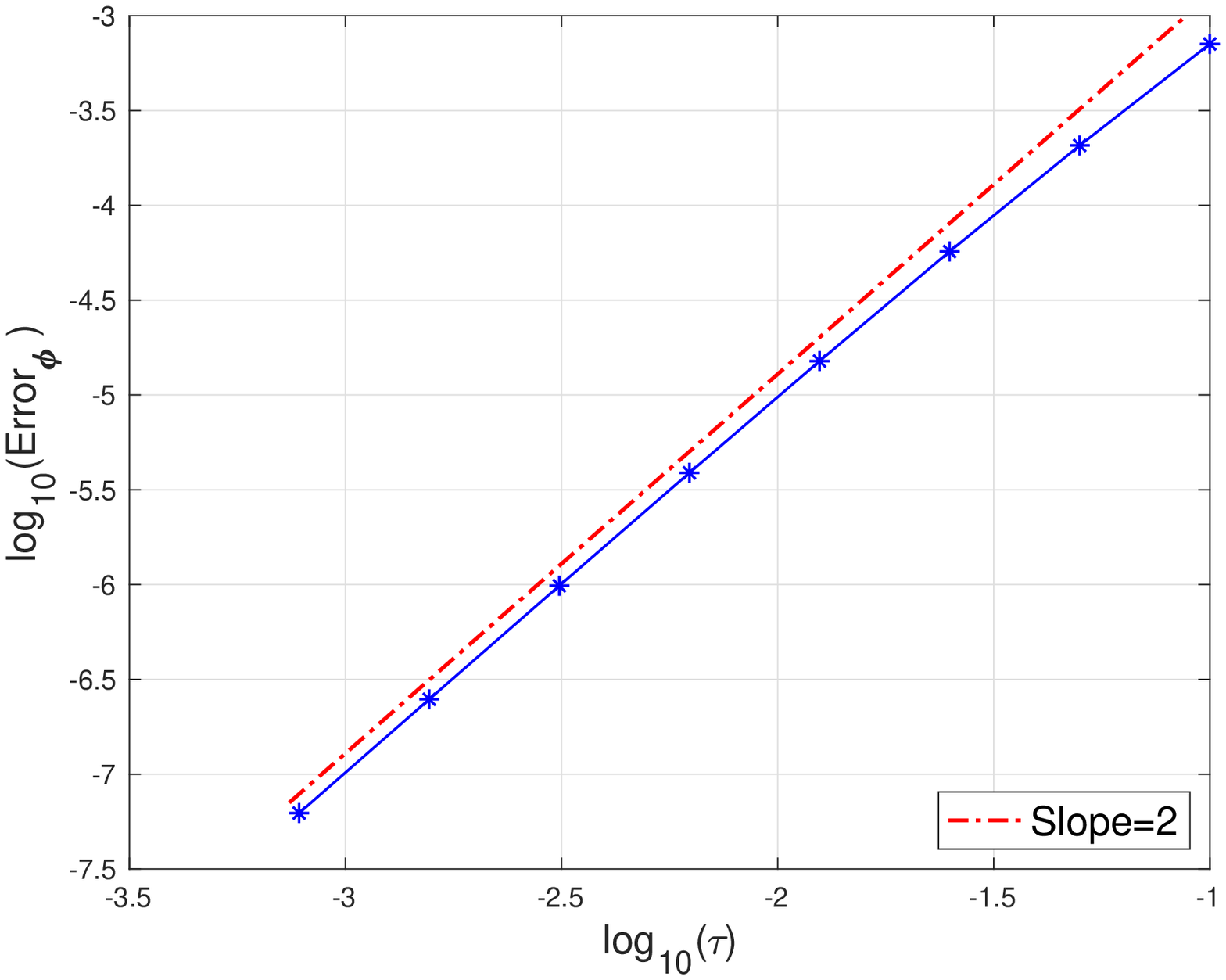}}
\centerline{\tiny (a) Convergence rates in time }
\end{minipage}
\begin{minipage}[t]{0.32\linewidth}
\centerline{\includegraphics[scale=0.28]{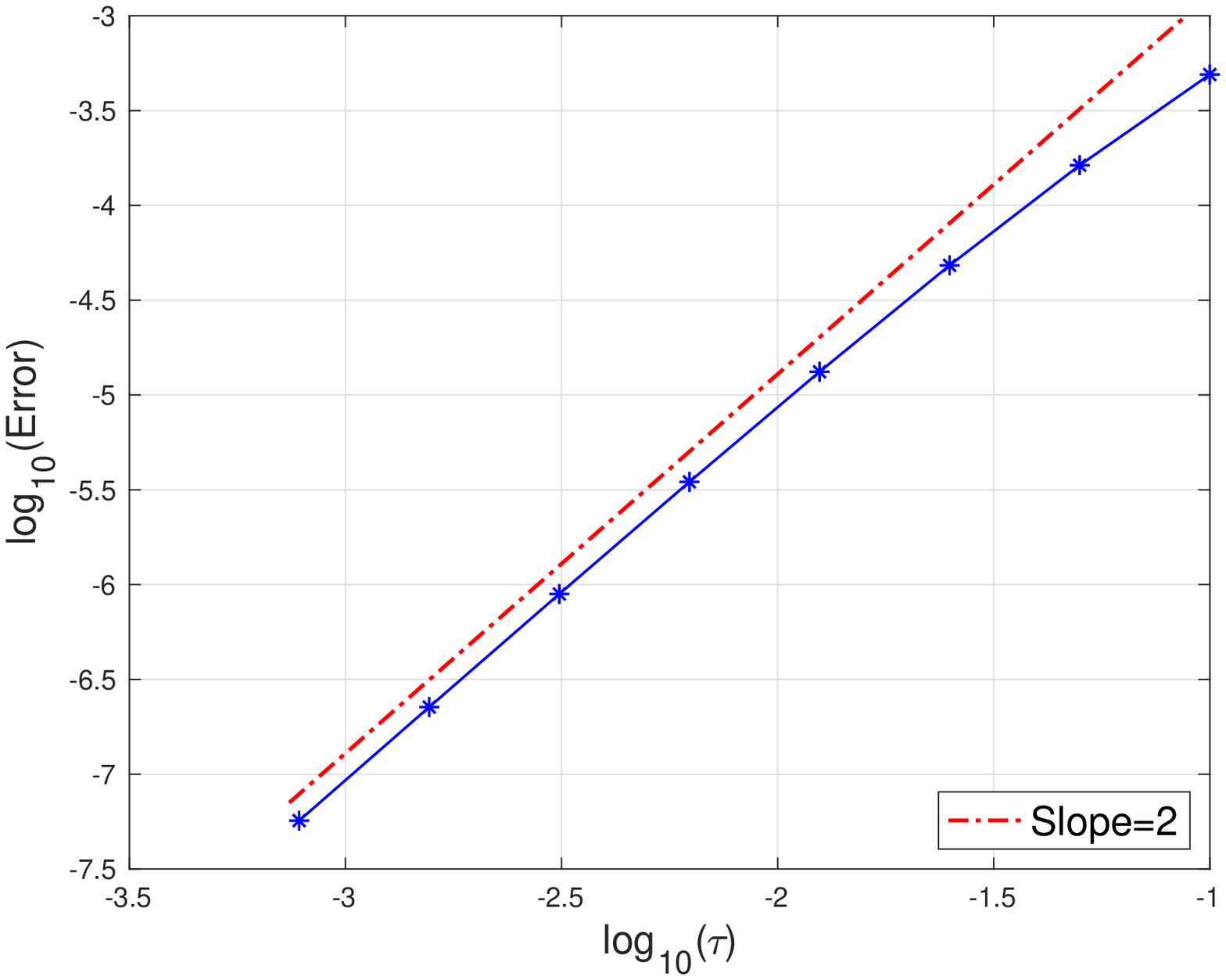}}
\centerline{\tiny(b) $\dps\max_{n}\big|r^{n+1}/\sqrt{E^{n}_{1,N}+C_{0}}-1\big|$ }
\end{minipage}
\begin{minipage}[t]{0.32\linewidth}
\centerline{\includegraphics[scale=0.28]{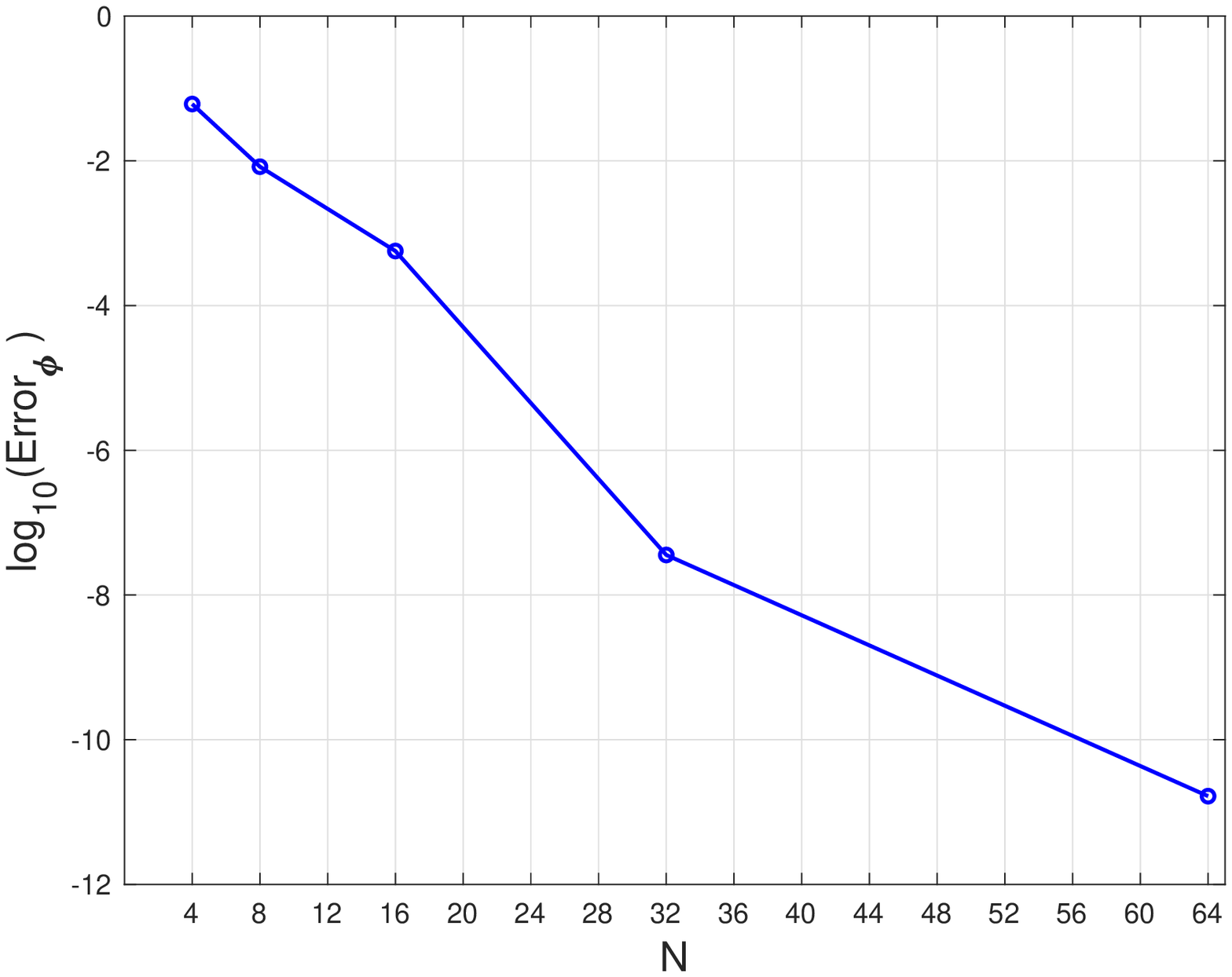}}
\centerline{\tiny(c) Convergence rates in space }
\end{minipage}
\caption{Convergence rates at $T=1$ for the scheme \eqref{BDF2_2} with $C_{0}=1/\Dt$.
}\label{fig1}
\end{figure*}


\begin{table}
  \begin{center}
      \caption{ Convergence rates in time at $T=1$ for the phase function $\phi^n$ in \eqref{BDF2_2} on nonuniform temporal meshes with $C_{0}=1/\Dt$.}\label{table1}
    \begin{tabular}{|ccc|cc|cc|cc|} \hline
    \multicolumn{3}{|c|}{temporal mesh}&\multicolumn{2}{|c|}{$\sigma=
    \frac{1}{2}$ }&\multicolumn{2}{|c|}{$\sigma=\frac{2}{3}$ }&\multicolumn{2}{|c|}{$\sigma=1$ } \\
     $M$&$\tau$&$\max\{\gamma_{n}\}$&error&order&error&order&error&order\\ \hline
     $20$ &7.91e-2&3.64&8.30e-5&-&1.01e-4&-&1.45e-4&-\\
     $40$ &4.12e-2&4.79&1.42e-6&2.71&1.93e-5&2.55&2.95e-5&2.44\\
    $80$ &2.21e-2&4.16&4.49e-6&1.84&4.72e-6&2.26&6.11e-6&2.52\\
    $160$ &1.11e-2&5.41&7.53e-7&2.58&8.57e-7&2.47&1.28e-6&2.25\\
    $320$ &5.55e-3&6.40&8.55e-8&3.16&1.05e-7&3.00&1.52e-7&3.10\\
   $640$ &2.70e-3&5.24&2.18e-8&1.89&2.84e-8&1.81&4.09e-8&1.81\\
   $1280$ &1.39e-3&7.06&5.71e-9&2.03&7.12e-9&2.10&1.00e-8&2.12\\
      \hline
    \end{tabular}
  \end{center}
\end{table}

\subsection{Time adaptive strategy}

In this part, the efficiency of the fully discrete scheme \eqref{BDF2_2} with a time adaptive strategy is investigated through a long time phase transition process. As we know, the phase transition process usually goes through different stages, changing quickly at the beginning and then rather slowly afterwards, until it reaches a steady state after a long time evolution.
Thus it would be very efficient to apply some time adaptive strategy in the simulation based on our proposed nonuniform BDF2 scheme \eqref{BDF2_2}.

We consider the phase transition process governed by the PFC equation \eqref{prob} with $\varepsilon=0.1$ and the initial  data $\phi^{0}_{i,j}=0.08+\eta_{i,j}$, where $\eta_{i,j}$ is a random number between $-0.08$ and $0.08$. Here, the computational domain is set to be $\Omega=(0,256)^{2}$.
The time adaptive strategy proposed in \cite{QZT11} is adopted in the simulation with our fully discrete scheme \eqref{BDF2_2}:
\bq\label{adp}
\dps\Dt_{n+1}=\min\Big(\max\big(\Dt_{min},\frac{\Dt_{max}}{\sqrt{1+\gamma|E^{'}(t)|^{2}}}\big),4.8645\cdot\Dt_{n}\Big),
\eq
where $\Dt_{min},\Dt_{max}$ are predetermined minimum and maximum time step sizes, and
$\gamma$ is a positive constant.
This type of the time adaptive strategy is based on the energy variation.
It will automatically select small time steps when the energy decays rapidly and choose larger time steps otherwise.
In this simulation, we set $\Dt_{min}=0.01$, $\Dt_{max}=5$ and $\gamma=10^{5}$.
Furthermore, $512\times512$ Fourier modes are used for the spatial discretization.
{\color{black} In order to capture the characteristics of the phase transition process with rapid energy changes where the small time step size $\Dt_{min}$ is used, we set $C_{0}=1/\Dt_{min}$ to improve the local accuracy of our proposed numerical scheme at these times of rapid energy changes. It has also been checked that the chosen $C_{0}$ is large enough such that $E_{1}(\cdot)+C_{0}>0.$}

In Figure \ref{fig2_1}, we plot the evolutions of the phase function $\phi^n$ up to $T=5000$ with {\color{black}four different types of temporal meshes, viz., the fixed large time step mesh with $\Dt=5$,  two different ones from the time adaptive strategy \eqref{adp}, and the uniform small time step mesh with $\Dt=0.01$.}
In Figure \ref{fig2_2} (a), it shows that
there is no obvious difference at $t=20$,
and the corresponding energies consist with each other at this stage.
After the energy goes through a large variation stage, the fixed large time step with $\Dt=5$ yields
inaccurate numerical solutions, while
{\color{black}both two tested adaptive time strategies with $\tau_{max}=1$ and $5$} give correct phase transition patterns consistent
with the results obtained by the small time step case $\Dt=0.01$.
{\color{black} It is observed in Figure \ref{fig2_2} (a) and (b) that both the computed modified energy $\widetilde{E}^{n}_{N}$ and the computed original energy $E(\phi^{n}_{N})$ are dissipative in time for the tested four different types of temporal meshes.}
Moreover, as shown in Figure \ref{fig2_2} (b), the energy dissipation of the proposed  scheme \eqref{BDF2_2}  with the adaptive strategy \eqref{adp} consists very well with the one using the fixed small time step $\Dt=0.01$.
The corresponding time step sizes are also plotted in Figure \ref{fig2_2} (c),
which shows the efficiency of our proposed  scheme \eqref{BDF2_2} with the adaptive strategy \eqref{adp} during this long time simulation.
{\color{black}Furthermore, using a smaller $\tau_{max}$ in  adaptive time strategy (4.1) certainly improves the accuracy more or less, but may
lead to considerable increase in the computational cost,  as shown in Figure 3 (c).}
\begin{figure*}[htbp]
\begin{minipage}[t]{0.19\linewidth}
\centerline{\includegraphics[scale=0.16]{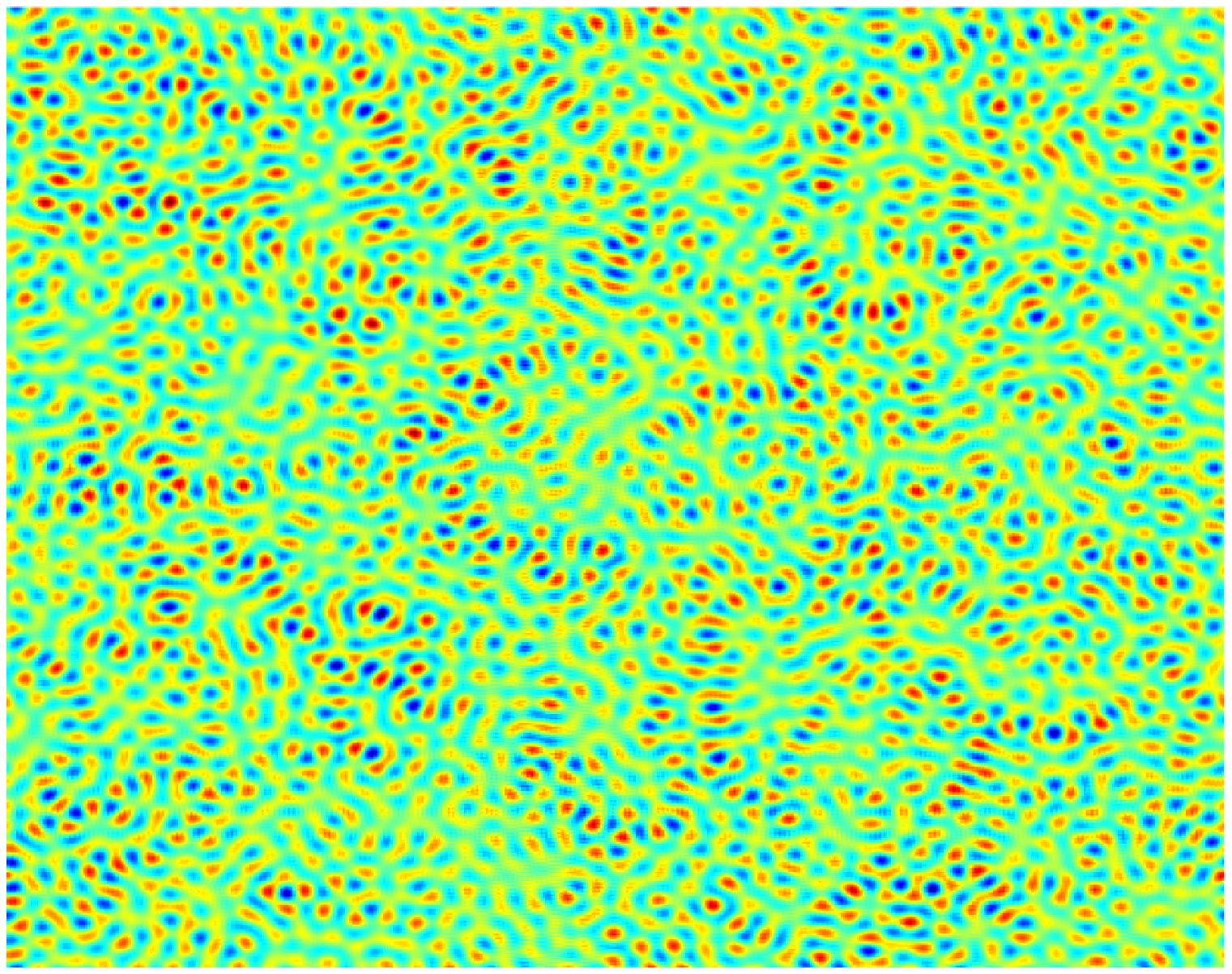}}
\centerline{}
\end{minipage}
\begin{minipage}[t]{0.19\linewidth}
\centerline{\includegraphics[scale=0.16]{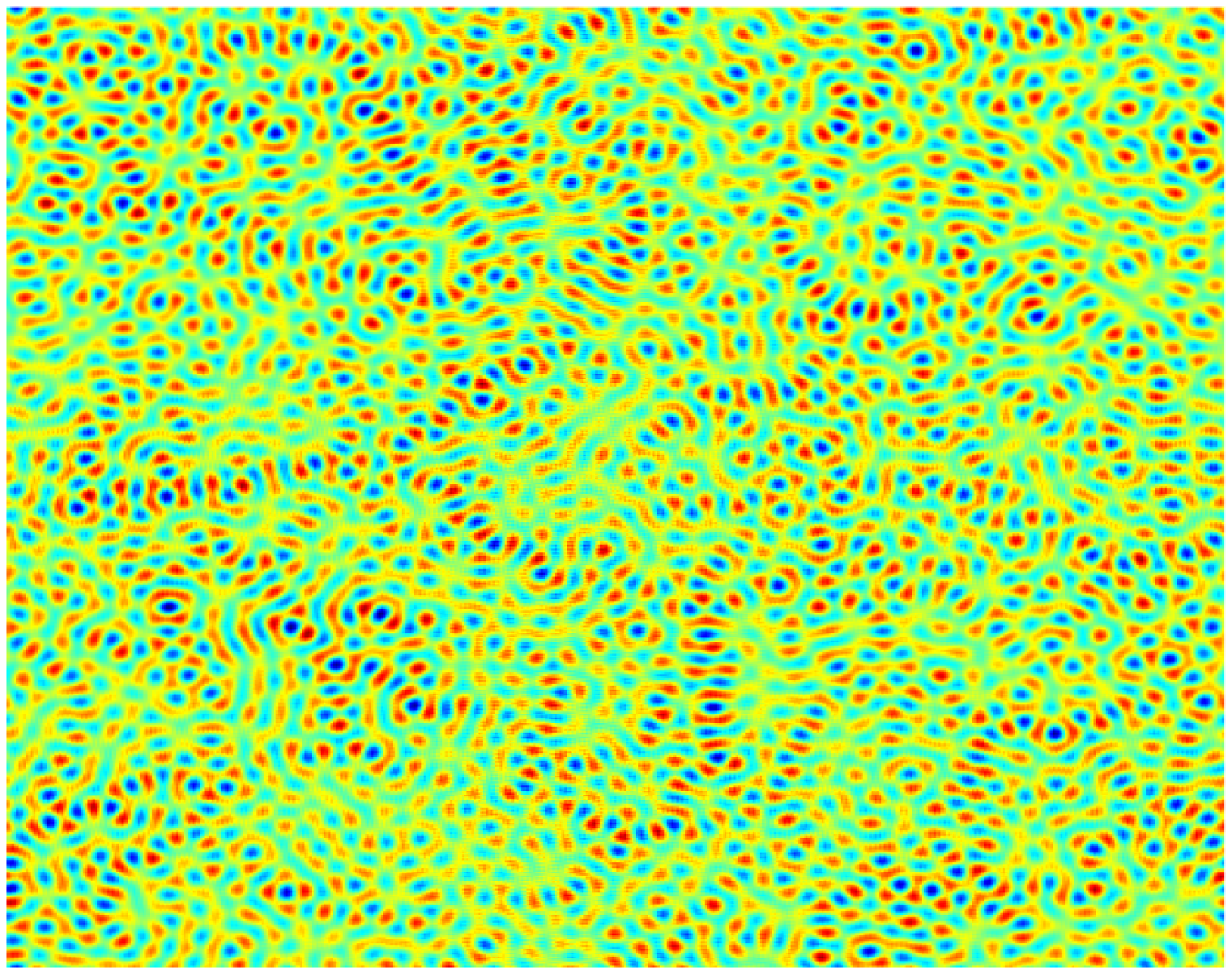}}
\centerline{}
\end{minipage}
\begin{minipage}[t]{0.19\linewidth}
\centerline{\includegraphics[scale=0.16]{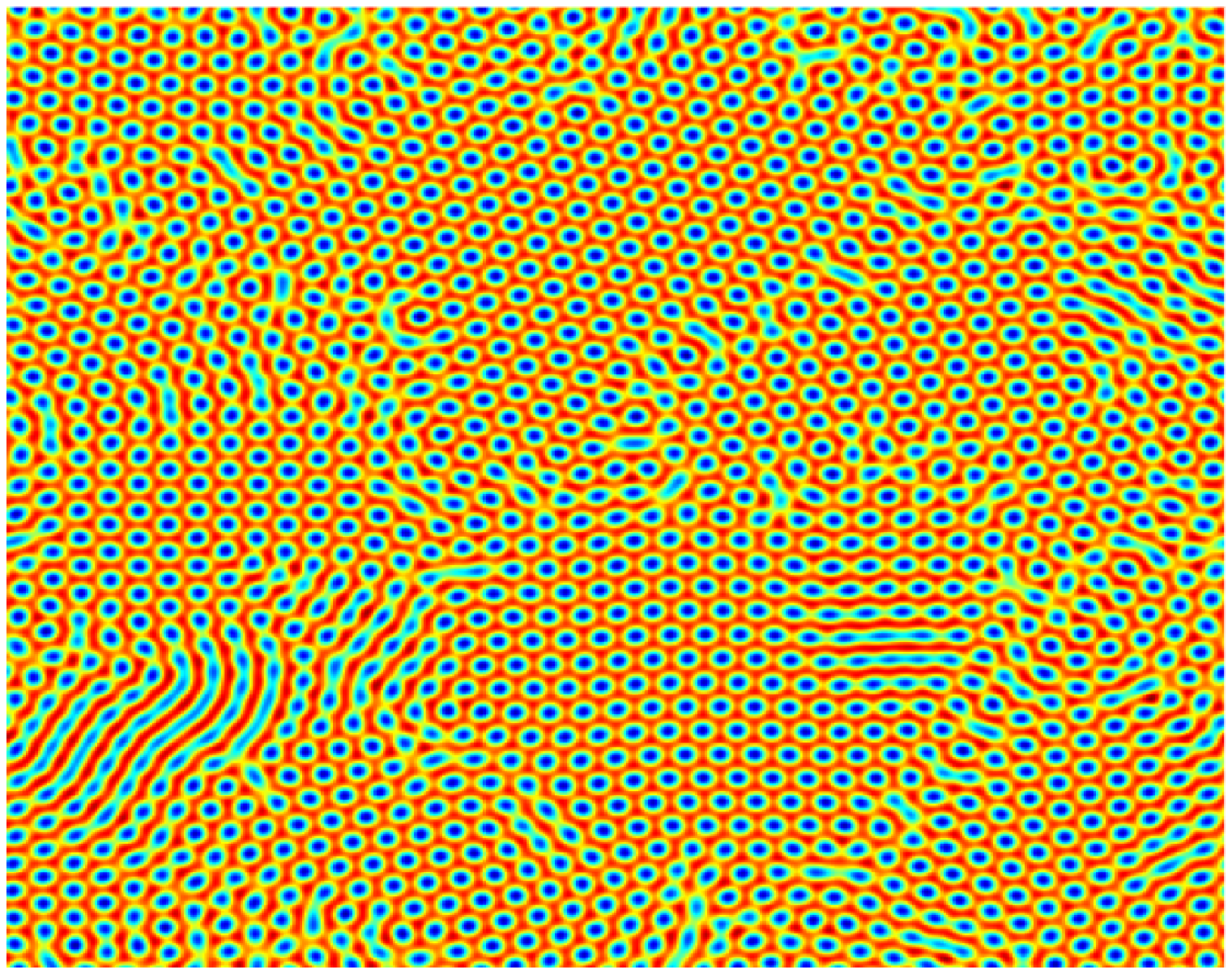}}
\centerline{(a) solutions with fixed large time step size $\tau=5$.}
\end{minipage}
\begin{minipage}[t]{0.19\linewidth}
\centerline{\includegraphics[scale=0.16]{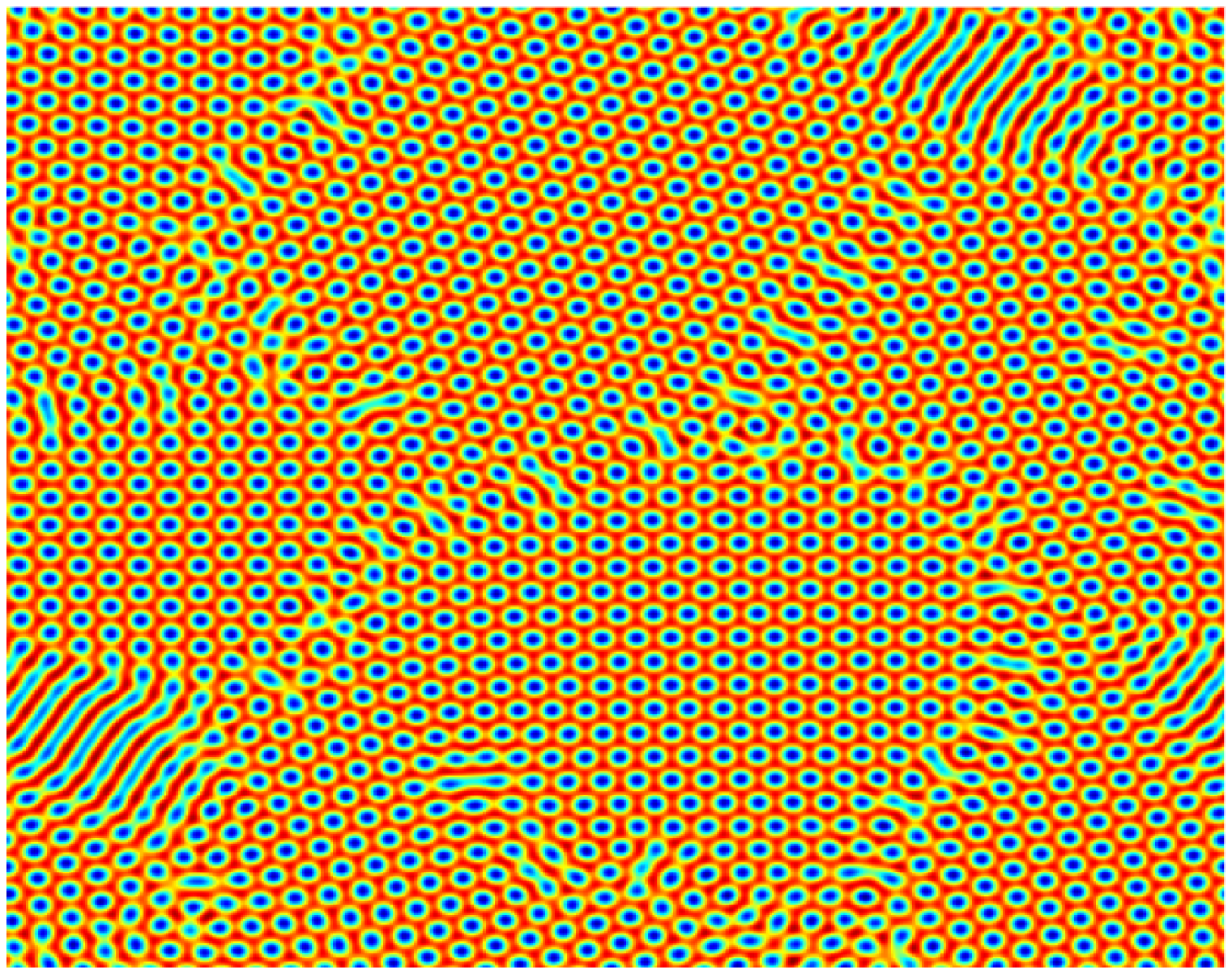}}
\centerline{}
\end{minipage}
\begin{minipage}[t]{0.19\linewidth}
\centerline{\includegraphics[scale=0.16]{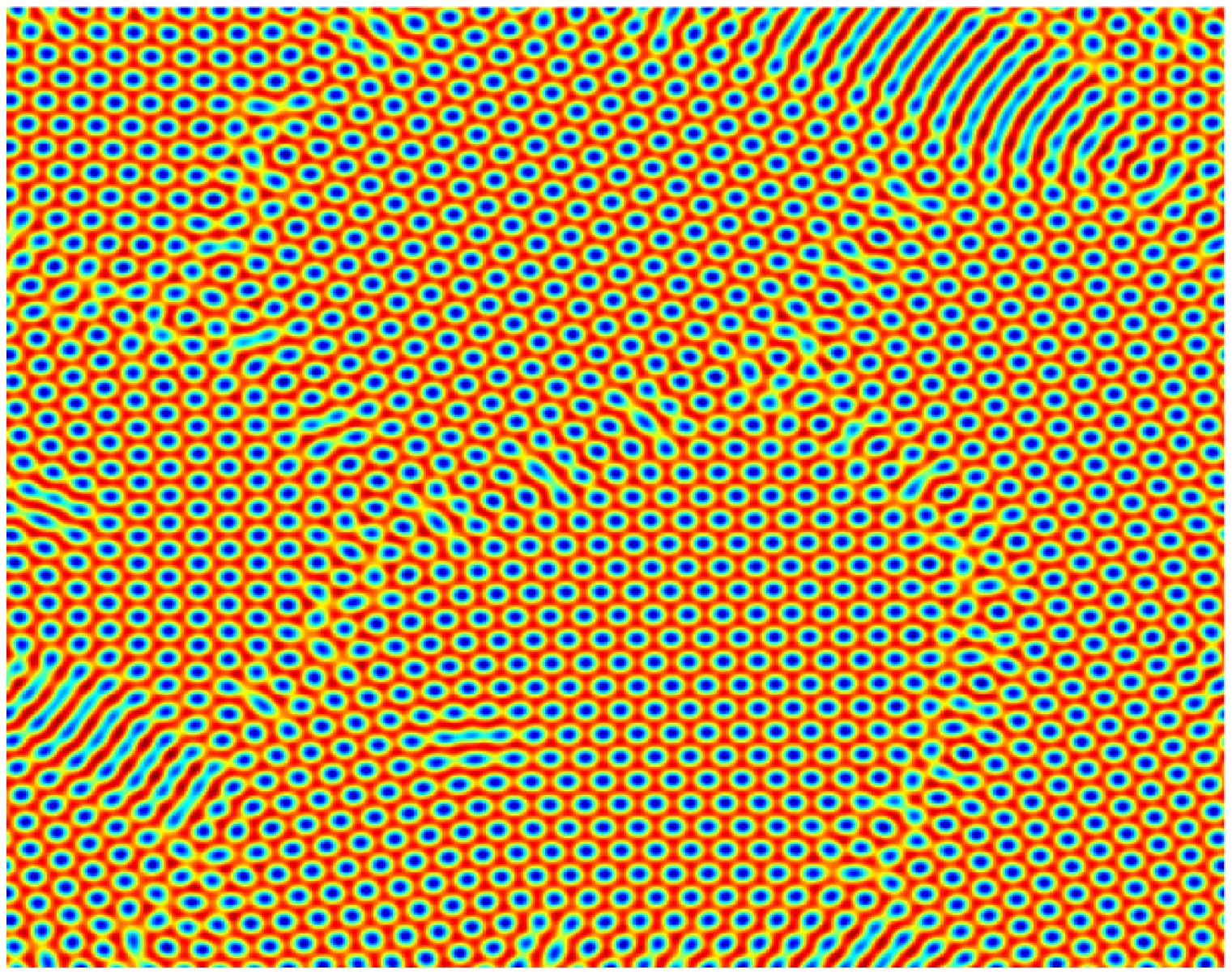}}
\centerline{}
\end{minipage}
\vskip 1mm
\begin{minipage}[t]{0.19\linewidth}
\centerline{\includegraphics[scale=0.16]{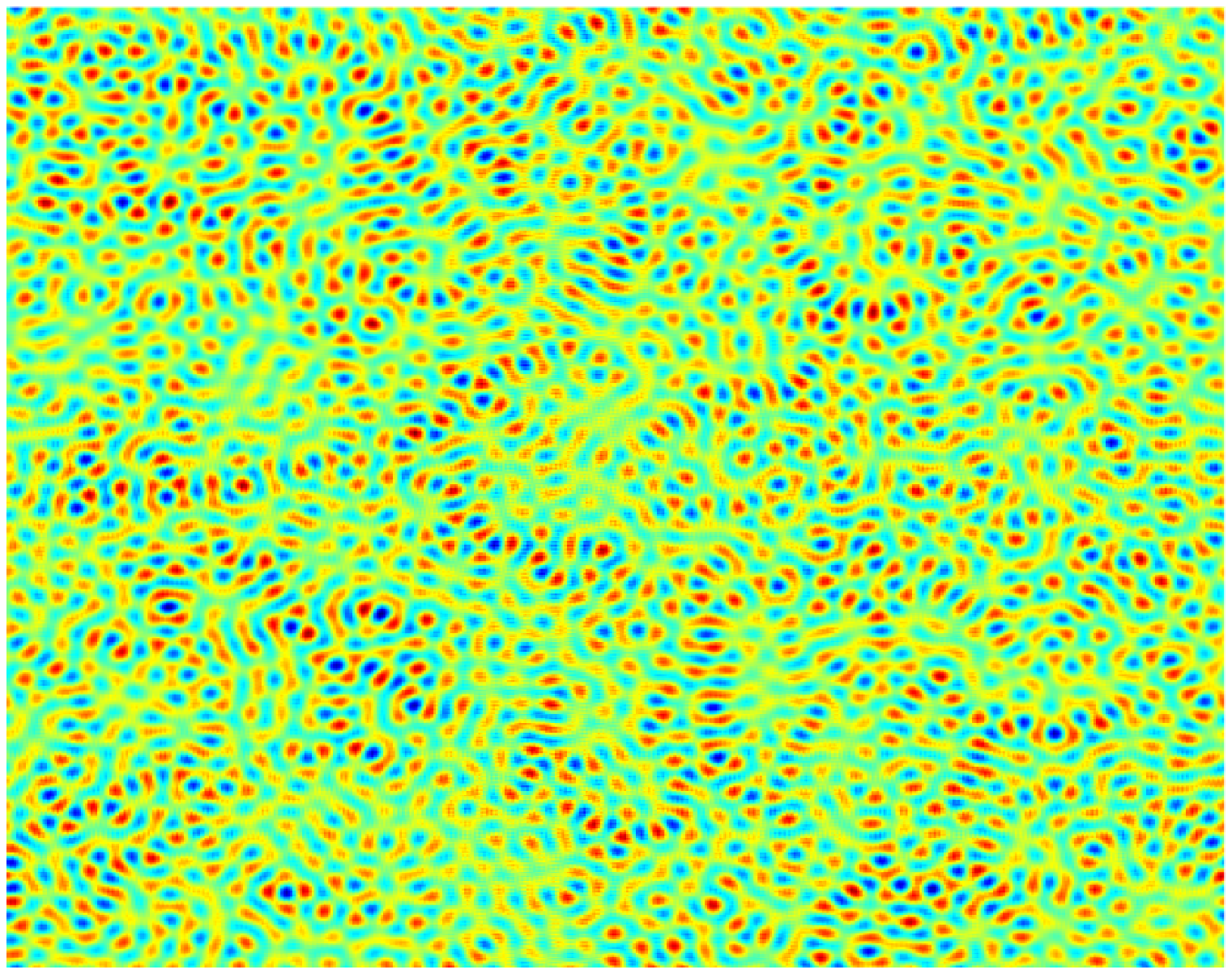}}
\centerline{}
\end{minipage}
\begin{minipage}[t]{0.19\linewidth}
\centerline{\includegraphics[scale=0.16]{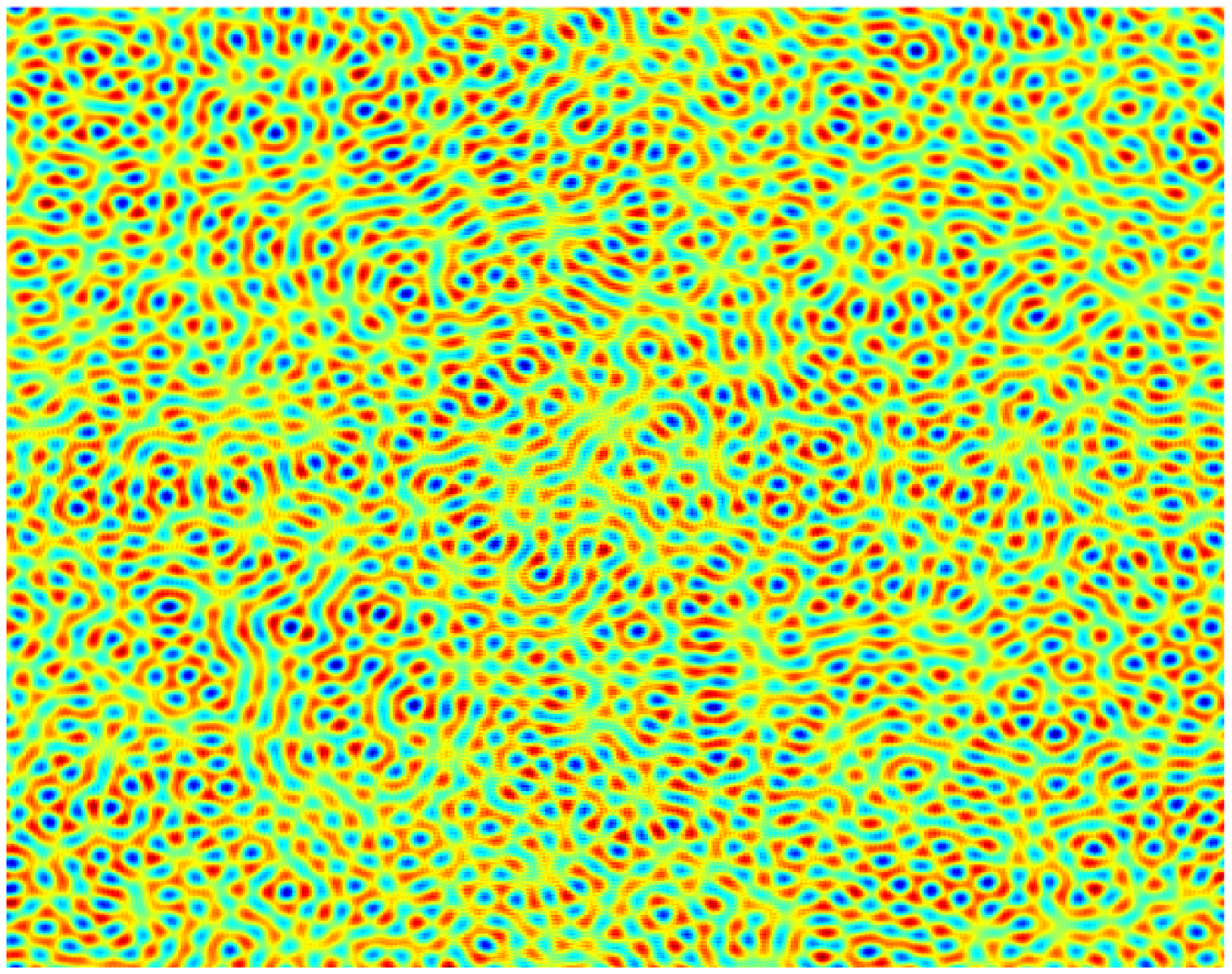}}
\centerline{}
\end{minipage}
\begin{minipage}[t]{0.19\linewidth}
\centerline{\includegraphics[scale=0.16]{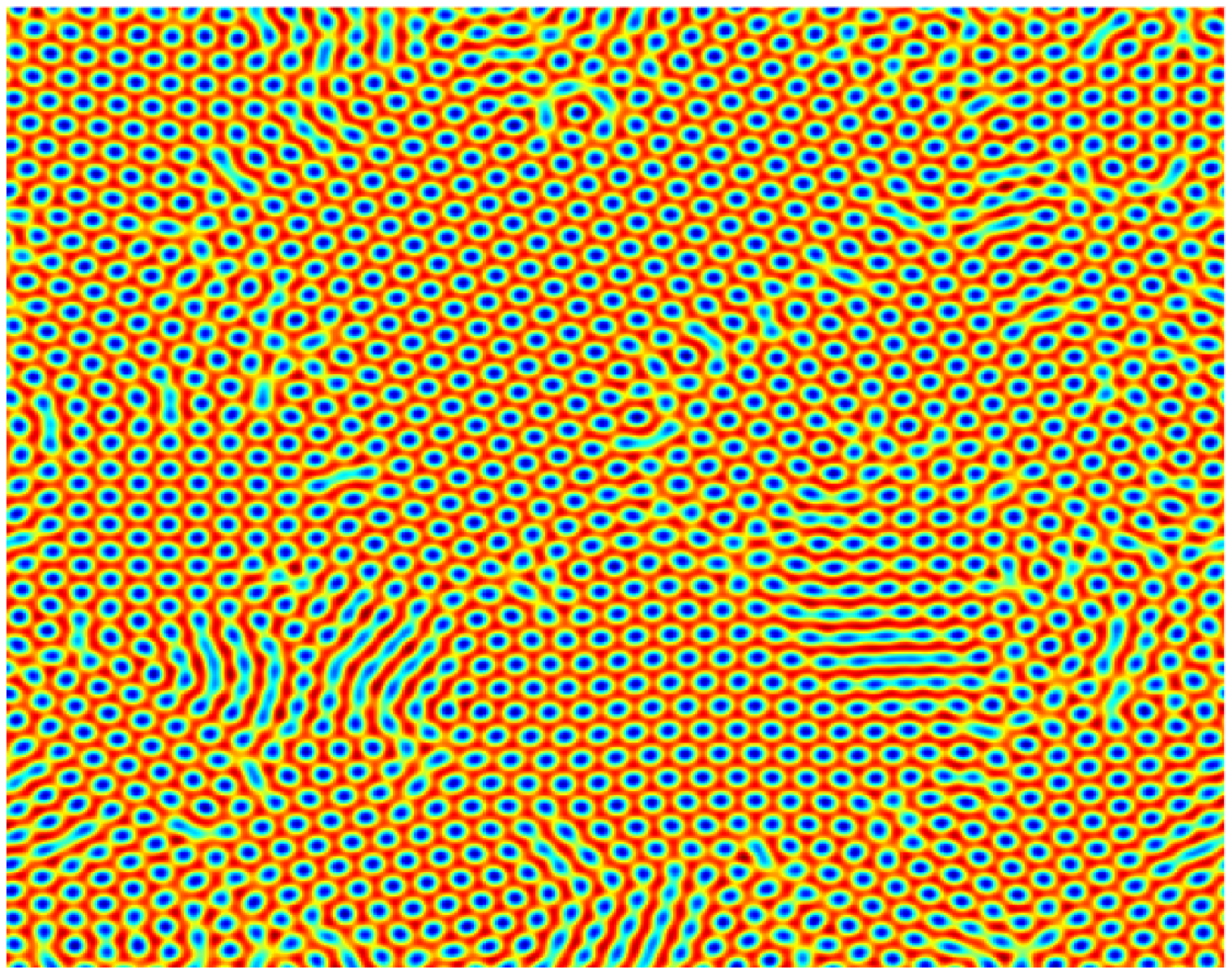}}
\centerline{(b) solutions with adaptive time steps  with $\tau_{min}=0.01, \tau_{max}=5$ and $\gamma=10^{5}.$}
\end{minipage}
\begin{minipage}[t]{0.19\linewidth}
\centerline{\includegraphics[scale=0.16]{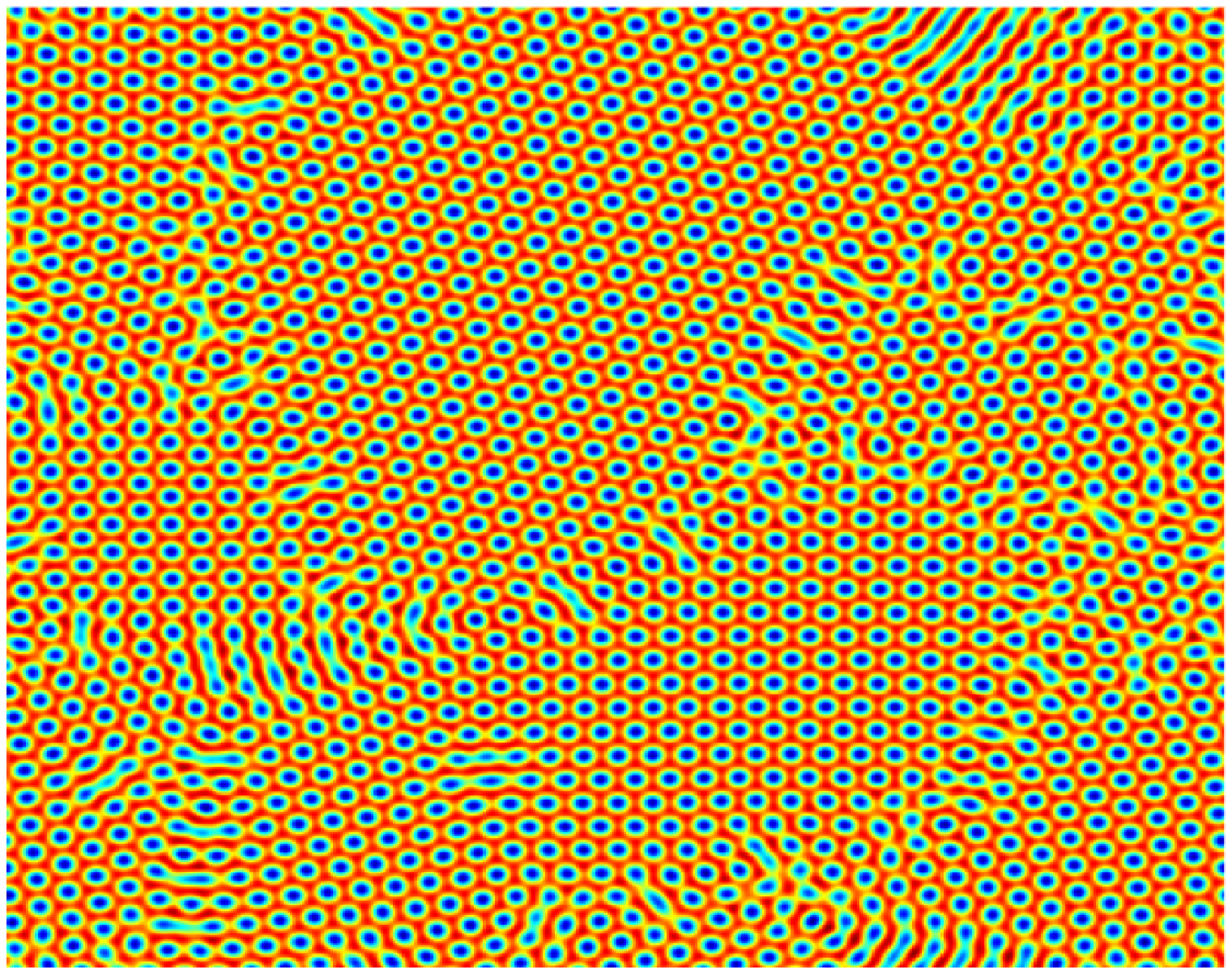}}
\centerline{}
\end{minipage}
\begin{minipage}[t]{0.19\linewidth}
\centerline{\includegraphics[scale=0.16]{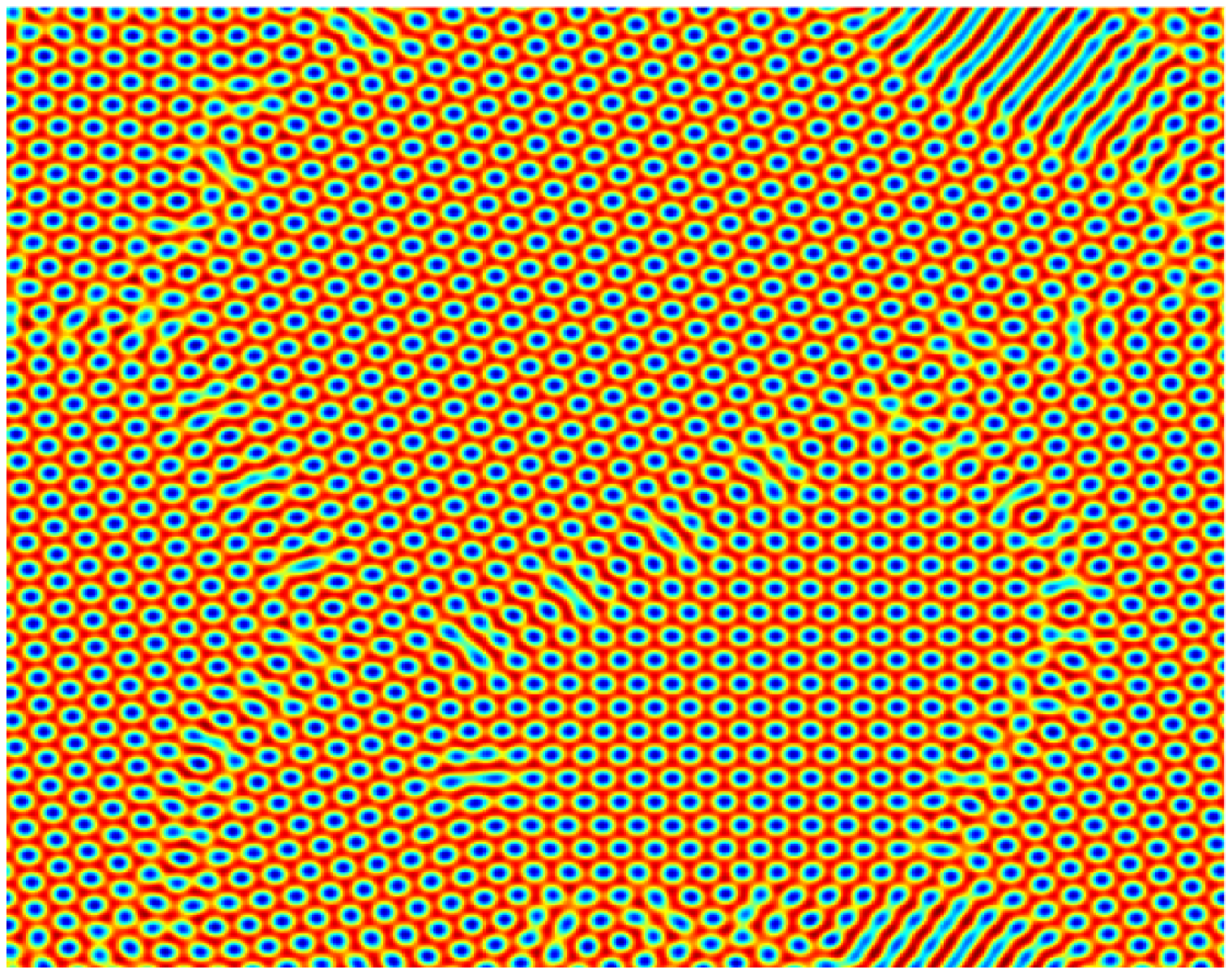}}
\centerline{}
\end{minipage}
\vskip 1mm
\begin{minipage}[t]{0.19\linewidth}
\centerline{\includegraphics[scale=0.16]{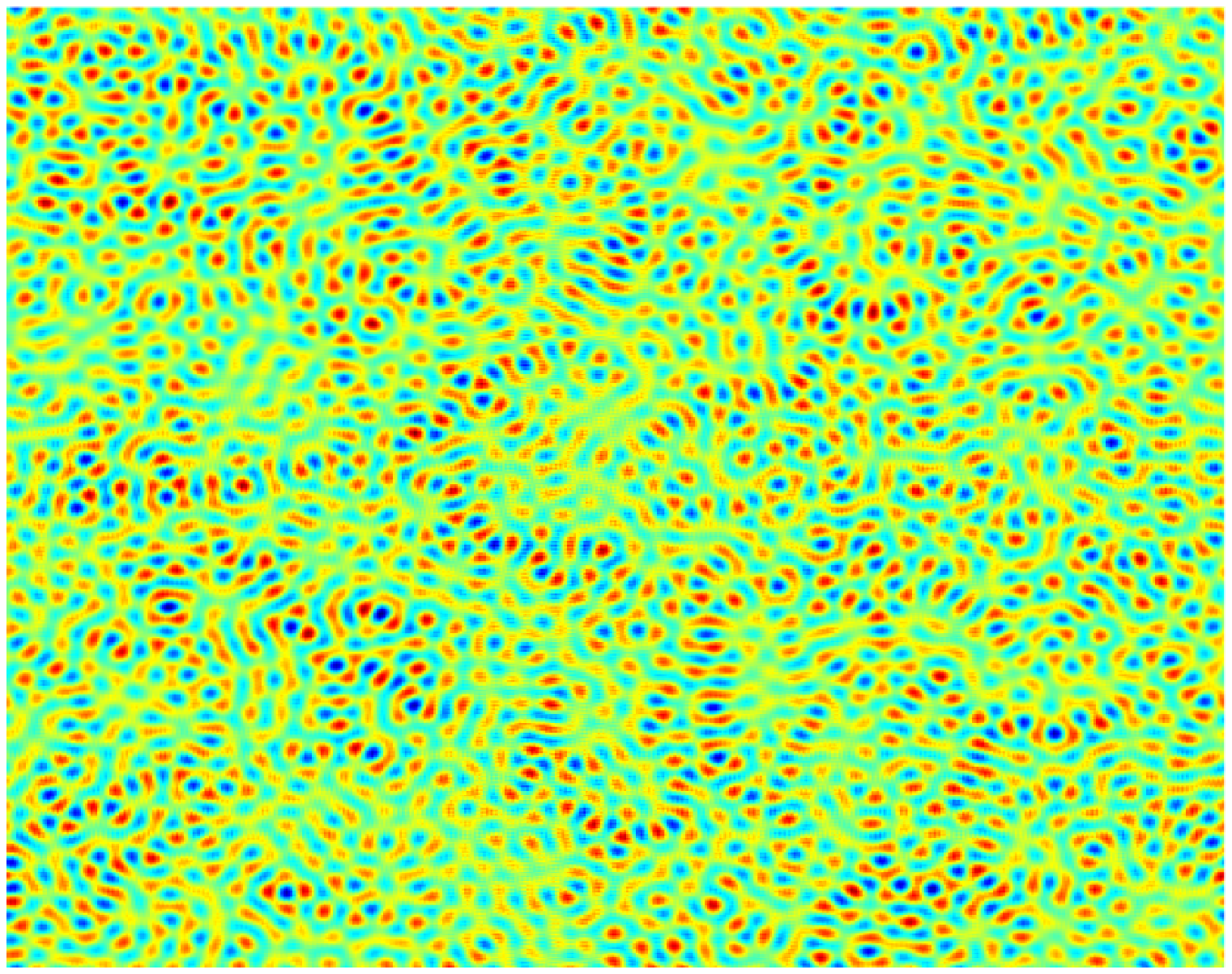}}
\centerline{}
\end{minipage}
\begin{minipage}[t]{0.19\linewidth}
\centerline{\includegraphics[scale=0.16]{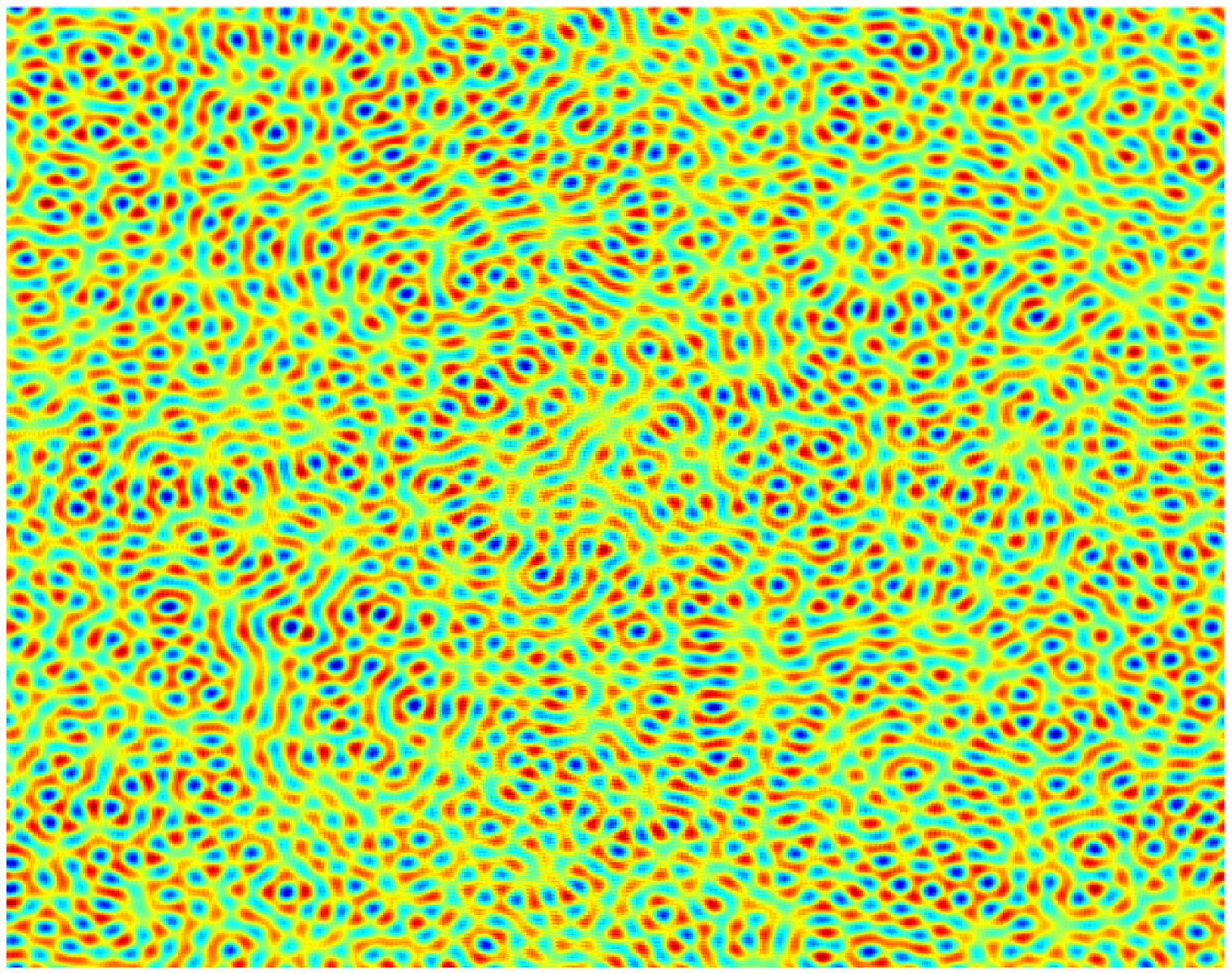}}
\centerline{}
\end{minipage}
\begin{minipage}[t]{0.19\linewidth}
\centerline{\includegraphics[scale=0.16]{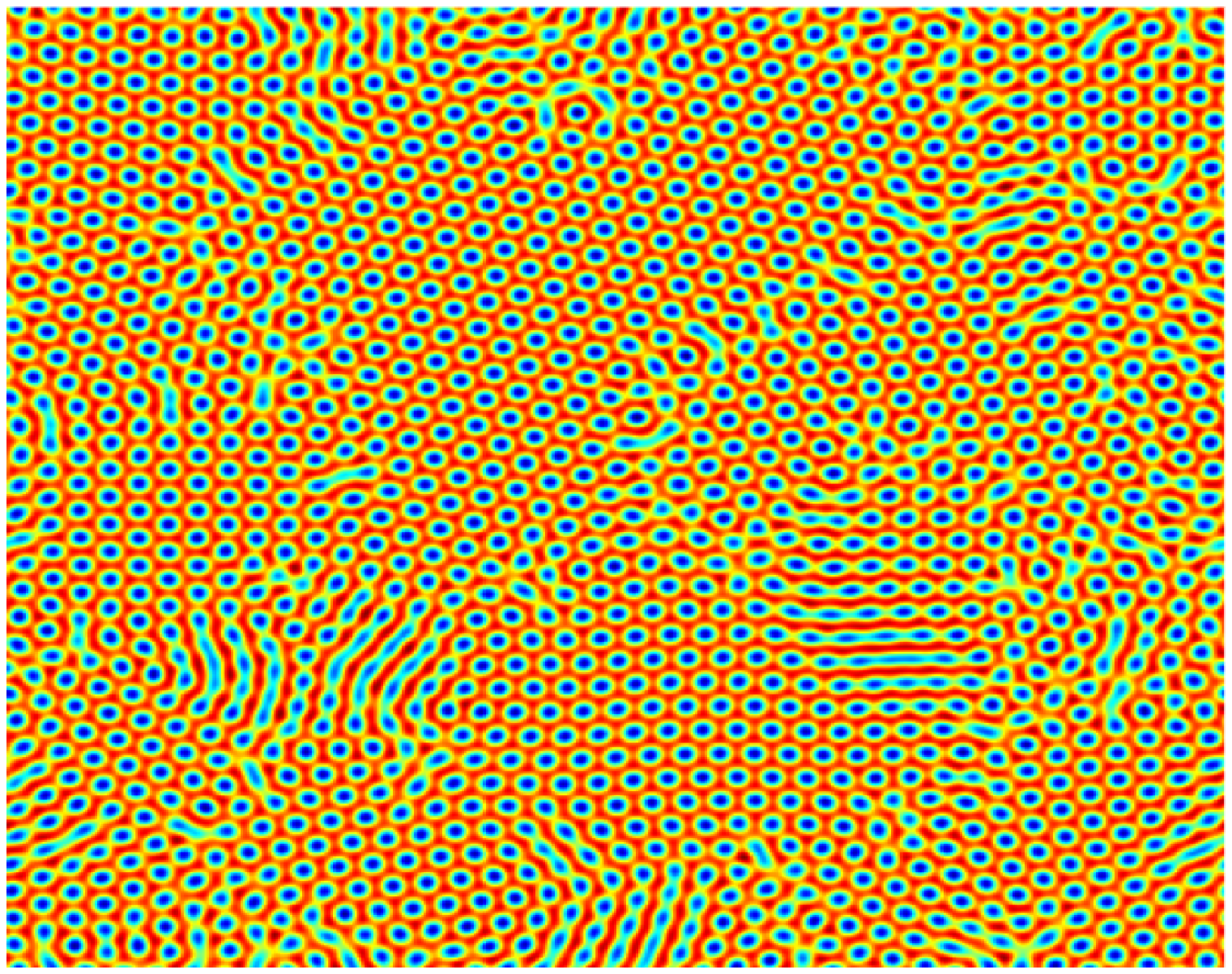}}
\centerline{(c) solutions with adaptive time steps  with $\tau_{min}=0.01, \tau_{max}=1$ and $\gamma=10^{5}.$}
\end{minipage}
\begin{minipage}[t]{0.19\linewidth}
\centerline{\includegraphics[scale=0.16]{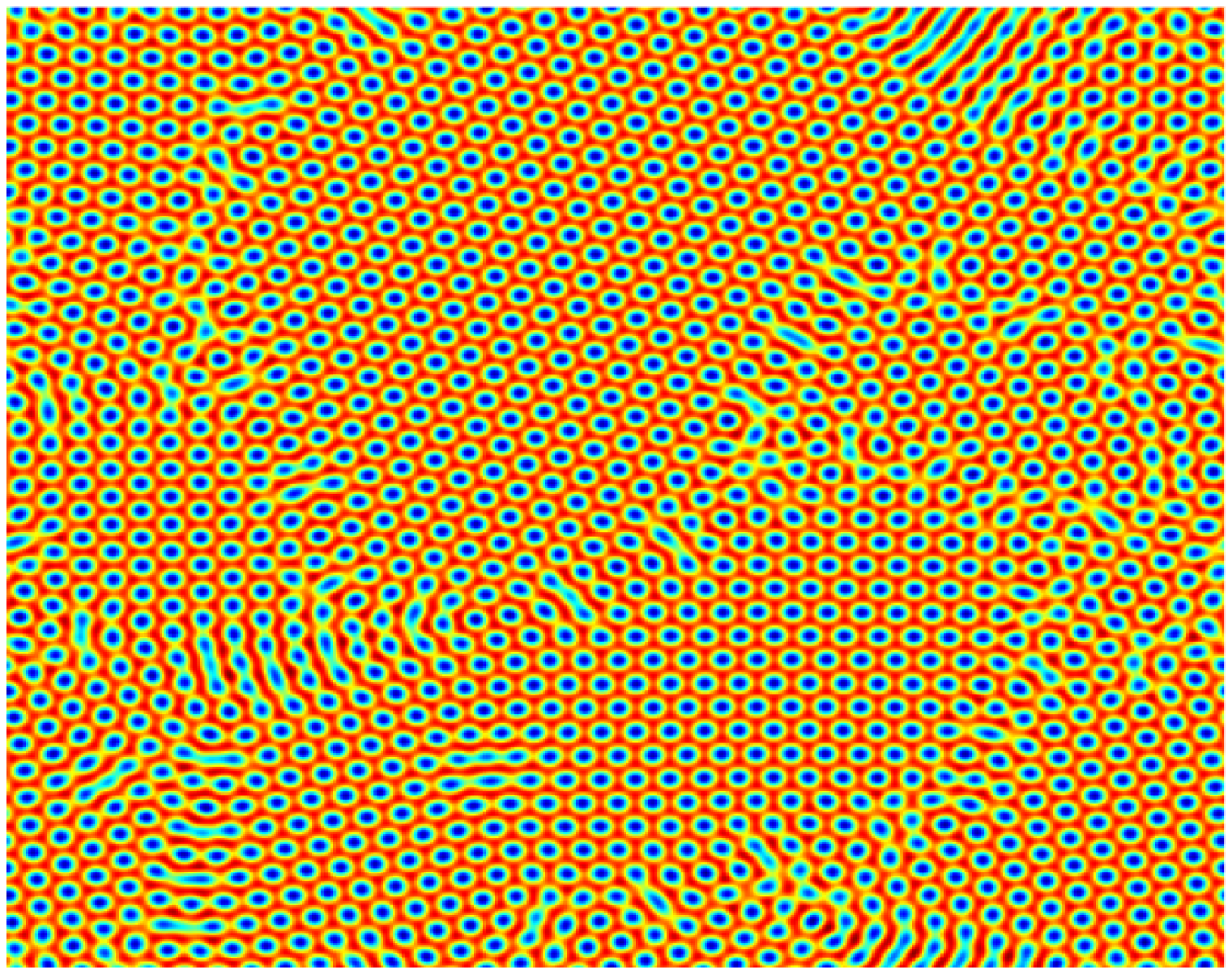}}
\centerline{}
\end{minipage}
\begin{minipage}[t]{0.19\linewidth}
\centerline{\includegraphics[scale=0.16]{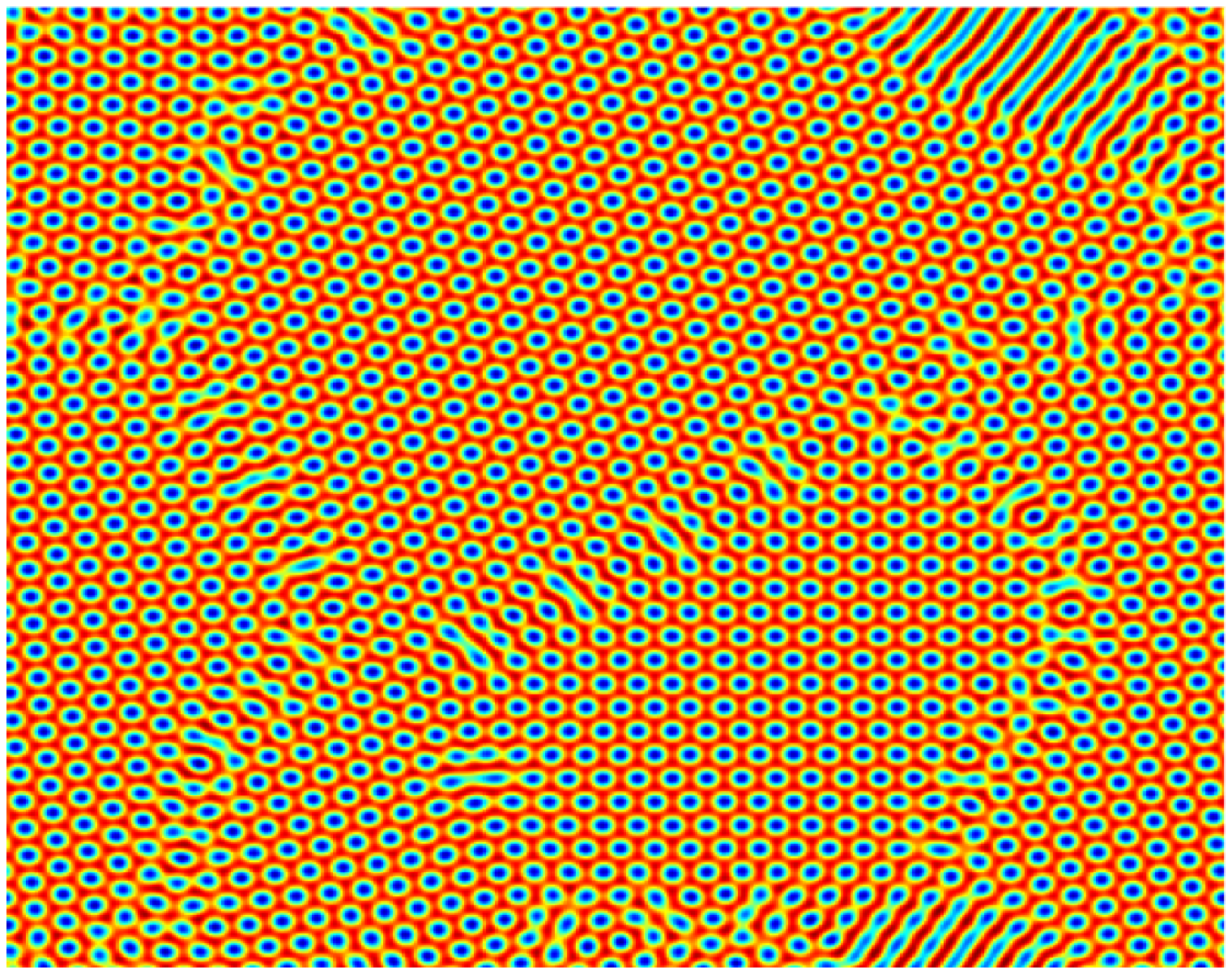}}
\centerline{}
\end{minipage}
\vskip 1mm
\begin{minipage}[t]{0.19\linewidth}
\centerline{\includegraphics[scale=0.16]{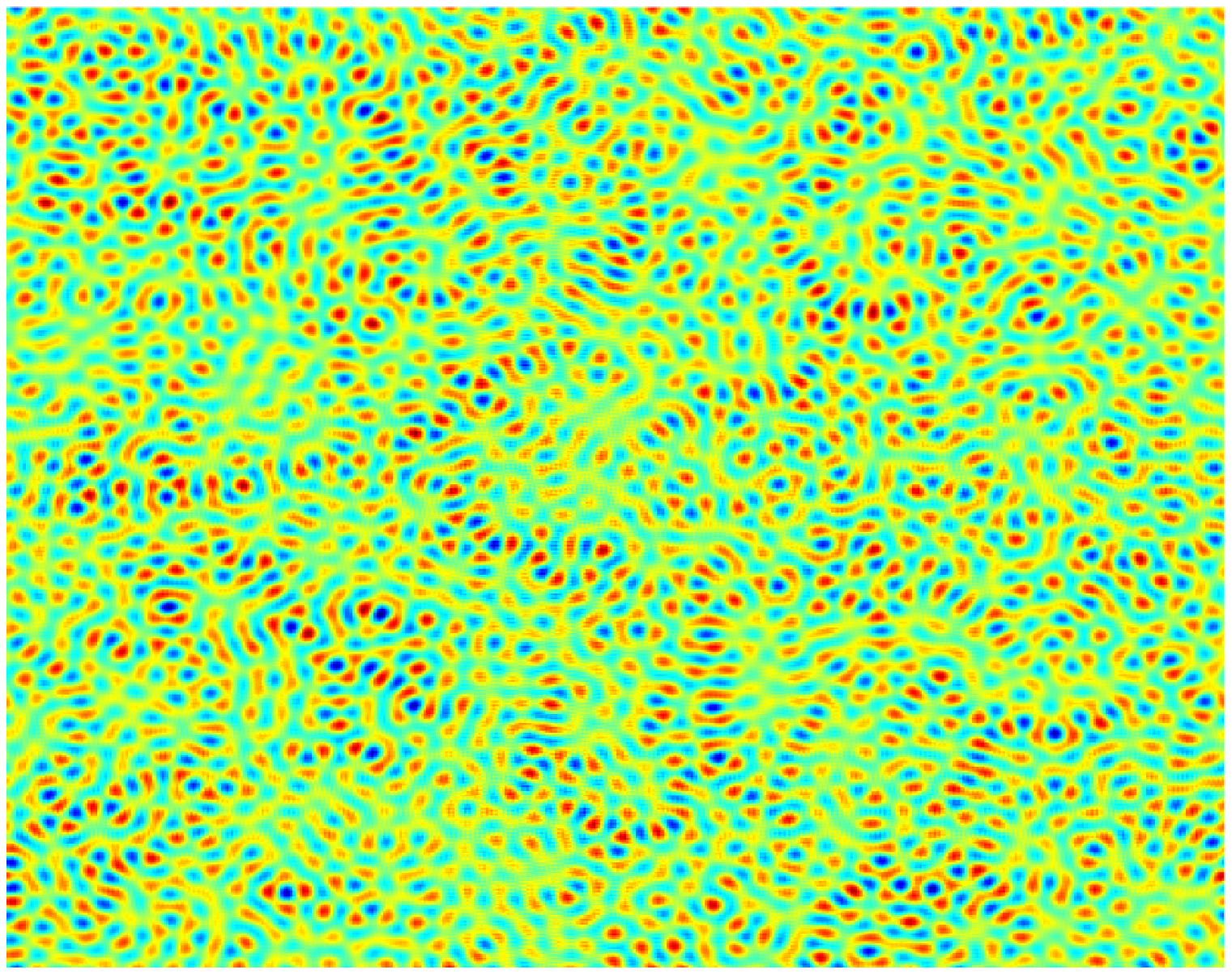}}
\centerline{}
\end{minipage}
\begin{minipage}[t]{0.19\linewidth}
\centerline{\includegraphics[scale=0.16]{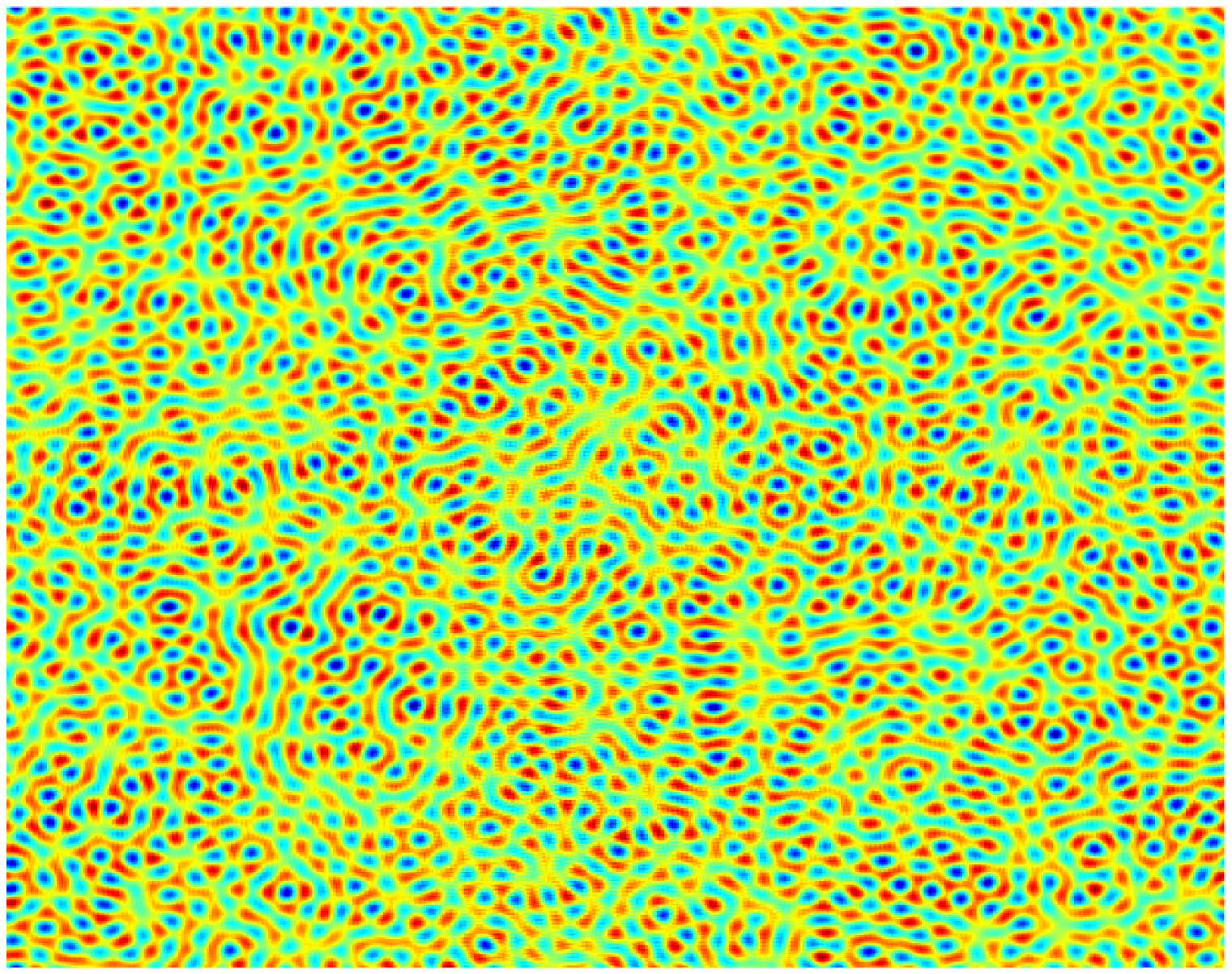}}
\centerline{}
\end{minipage}
\begin{minipage}[t]{0.19\linewidth}
\centerline{\includegraphics[scale=0.16]{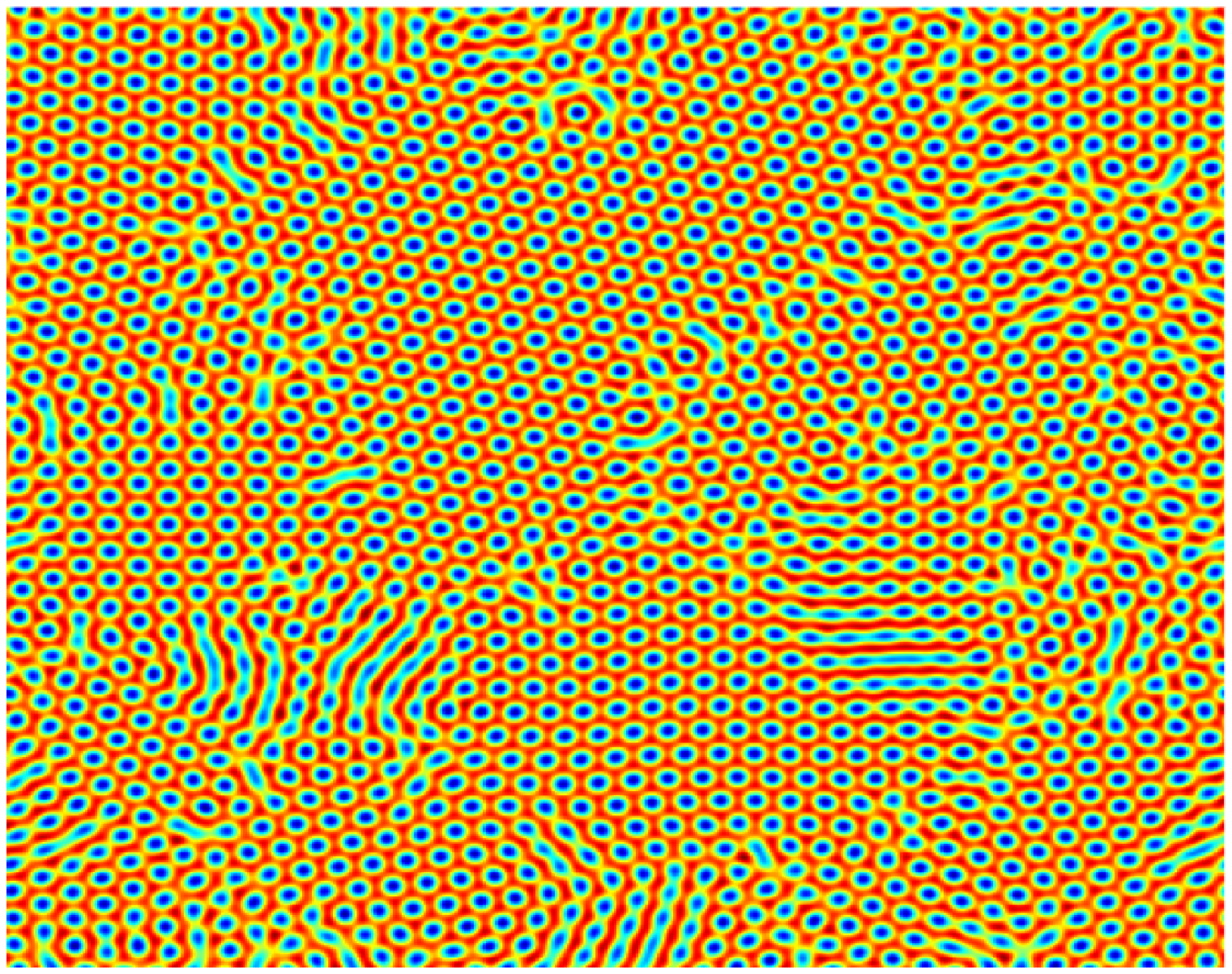}}
\centerline{(d) solutions with fixed small time step size $\tau=0.01$.}
\end{minipage}
\begin{minipage}[t]{0.19\linewidth}
\centerline{\includegraphics[scale=0.16]{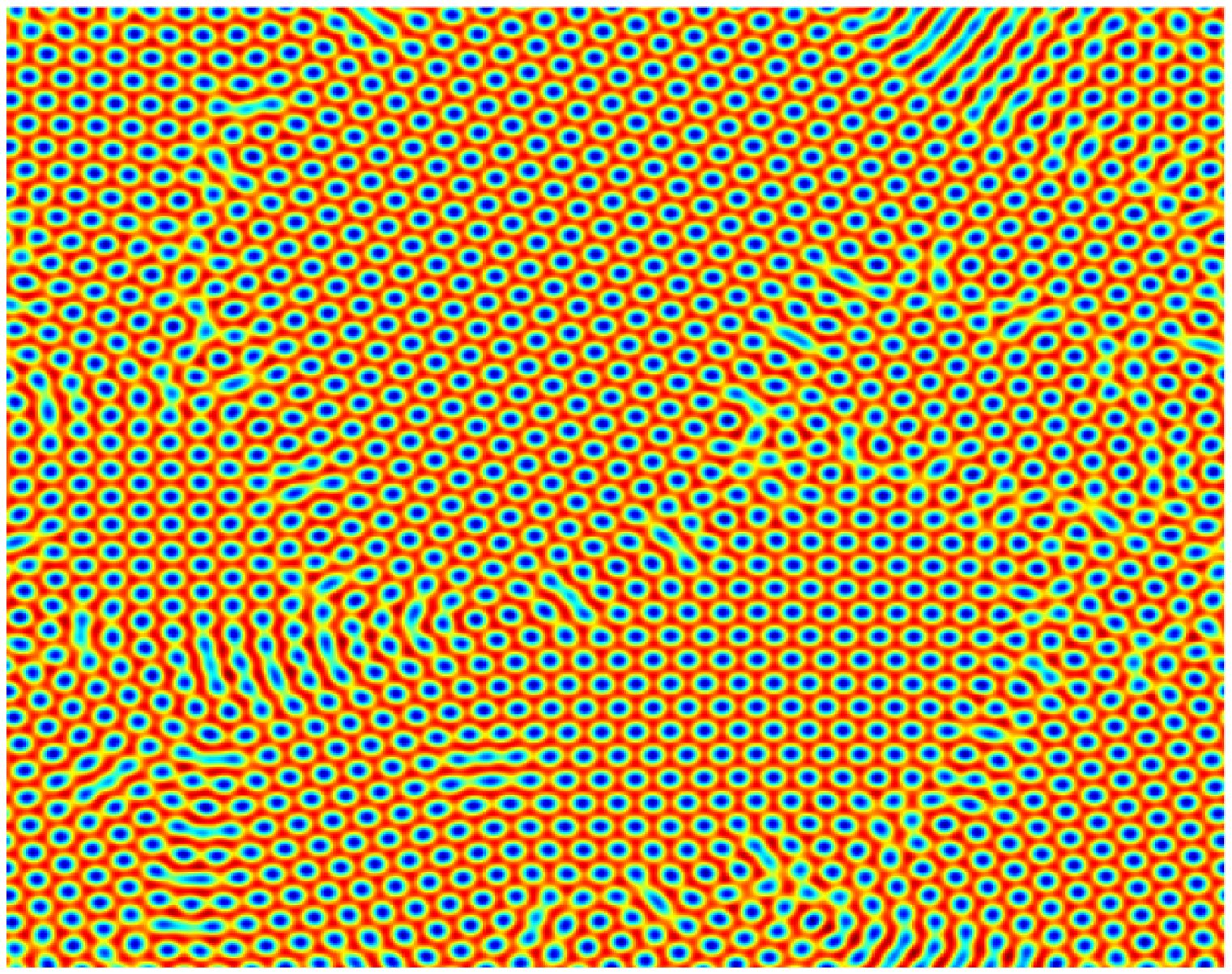}}
\centerline{}
\end{minipage}
\begin{minipage}[t]{0.19\linewidth}
\centerline{\includegraphics[scale=0.16]{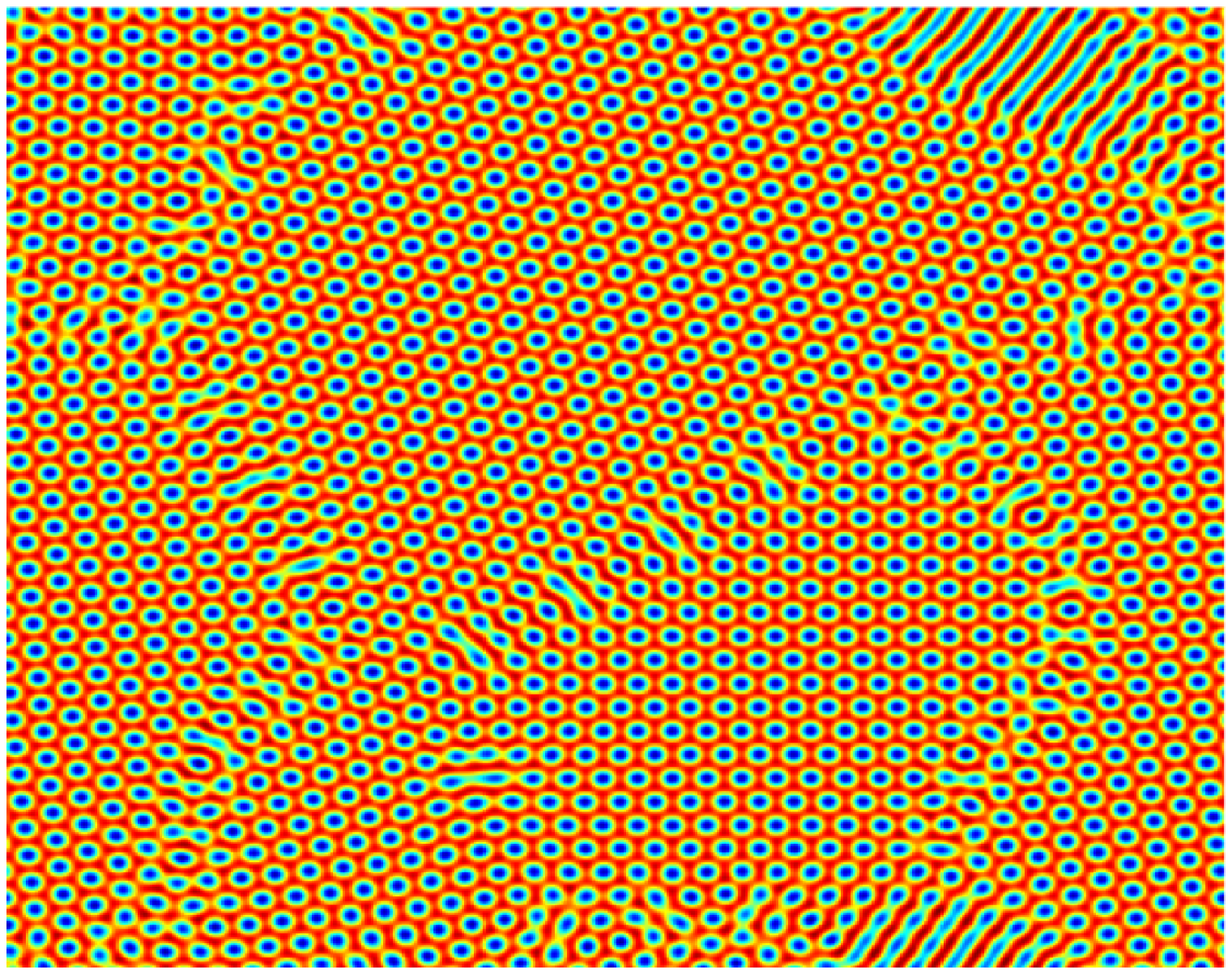}}
\centerline{}
\end{minipage}
\caption{Solution snapshots of the phase transition process for the PFC equation at around $t=20,50, 500, 1500,$ and $5000$, respectively.
}\label{fig2_1}
\end{figure*}

\begin{figure*}[htbp]
\begin{minipage}[t]{0.49\linewidth}
\centerline{\includegraphics[scale=0.4]{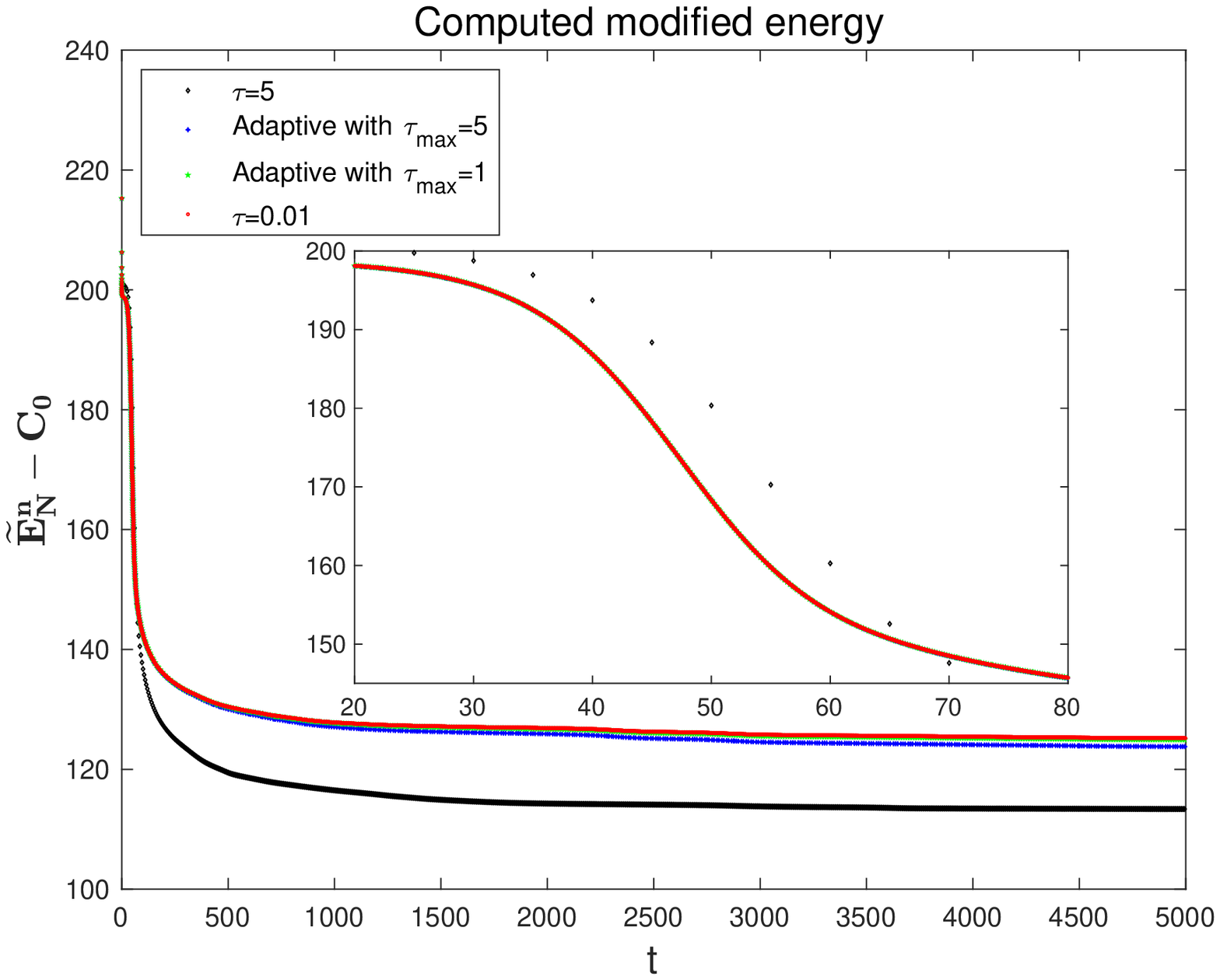}}
\centerline{(a) the computed modified energy $\widetilde{E}^{n}_{N}-C_{0}$}
\end{minipage}
\begin{minipage}[t]{0.49\linewidth}
\centerline{\includegraphics[scale=0.4]{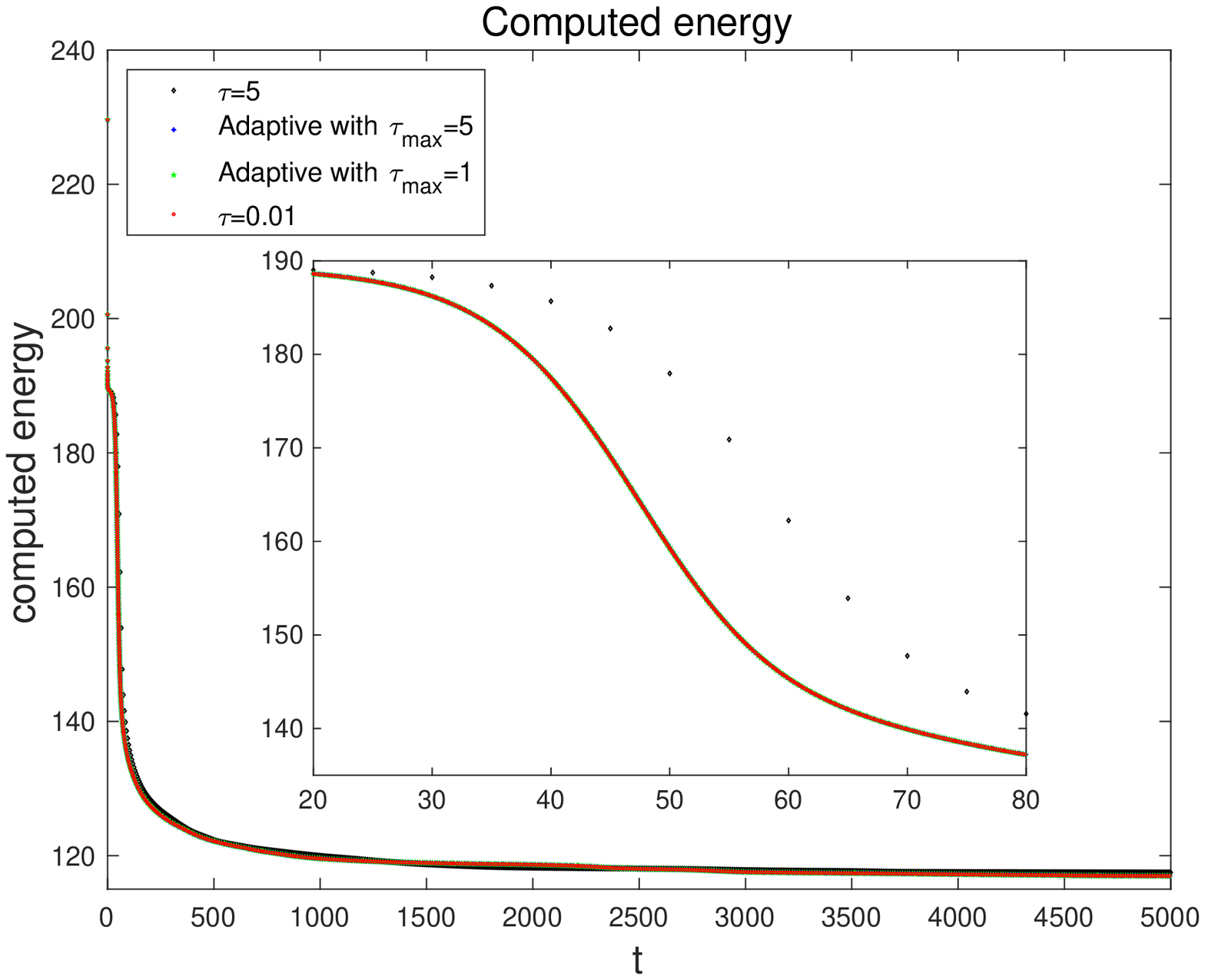}}
\centerline{(b) the computed energy $E(\phi^{n})$}
\end{minipage}
\begin{minipage}[t]{0.49\linewidth}
\centerline{\includegraphics[scale=0.4]{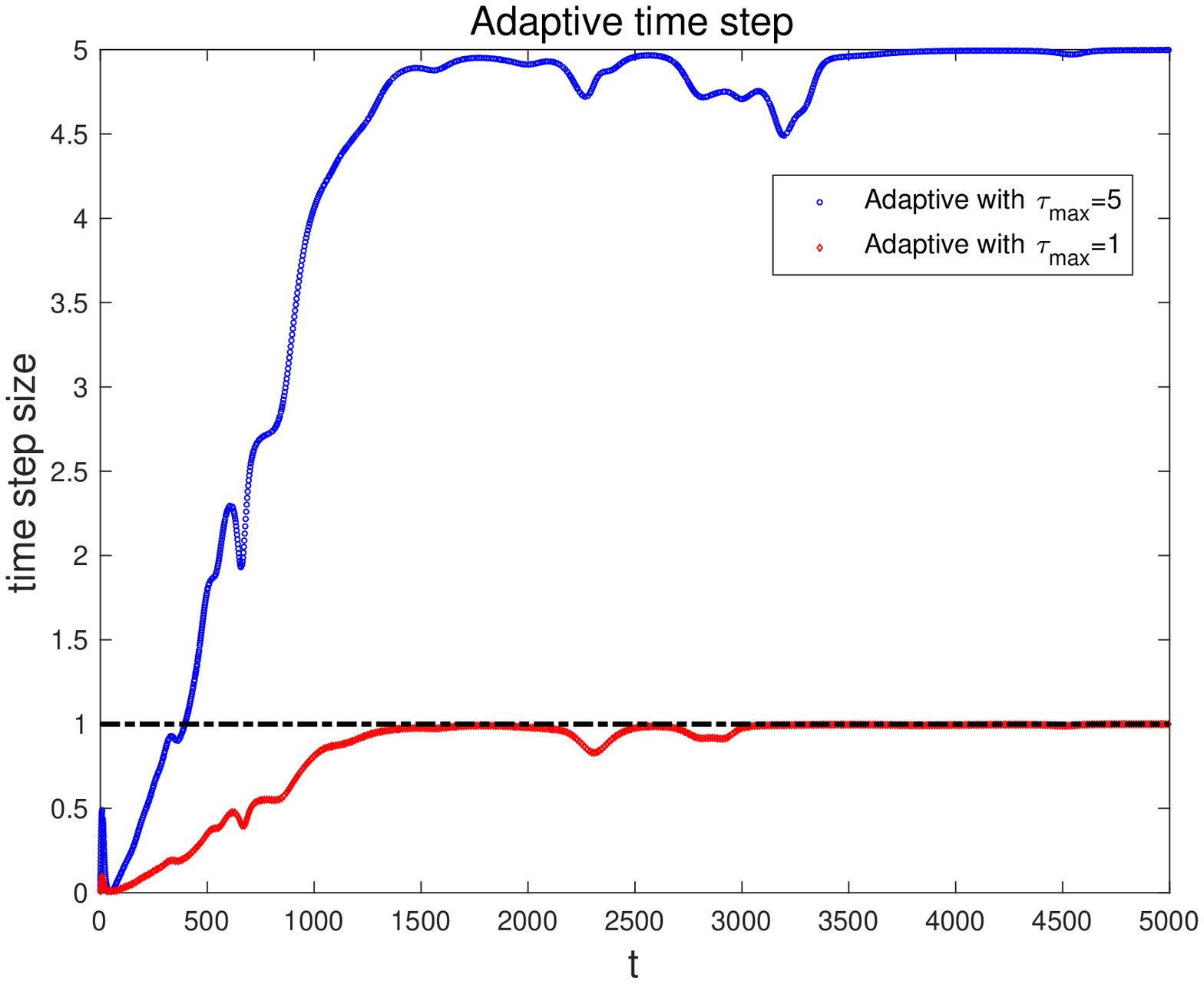}}
\centerline{(c) adaptive time step sizes with $\tau_{max}=1$ and $\tau_{max}=5$ }
\end{minipage}
\caption{Evolution in time of the modified energy, the original energy, and the adaptive time step sizes.
}\label{fig2_2}
\end{figure*}

\subsection{Crystal growth in a supercooled liquid}
In this subsection, we simulate the growth of a polycrystal in a supercooled liquid. Initially, three crystallites with different orientations are given by
\beq
\phi_{0}(x_{l},y_{l})=0.285+0.446\Big(\cos\Big(\frac{0.66}{\sqrt{3}}y_{l}\Big)\cos(0.66x_{l})-0.5\cos\Big(\frac{1.32}{\sqrt{3}}y_{l}\Big)\Big), ~~l=1,2,3,
\eeq
where $x_l$ and $y_{l}$ define a local system of cartesian coordinates that is oriented with the crystallite lattice. Similar numerical examples can be found in \cite{EKHG02,GN12,HWWL09,LS20,YH17}.
The computed domain is $(0,800)^{2}$, $\varepsilon=0.25$, and the initial three crystallites are located in the blocks with length of 40 in the computed domain, seeing Figure \ref{fig3_1} for the case of $t=0$.
In order to obtain three different orientations of the crystallites, we define the local coordinates $(x_{l},y_{l})$ by using the following affine transformation of the global coordinates $(x,y)$,
\bry
\begin{cases}
x_{l}=x\cos(\theta)-y\sin(\theta),\\
y_{l}=x\sin(\theta)+y\cos(\theta),
\end{cases}
\ery
where $\theta=-\pi/4, 0, \pi/4$ for $l=1,2,3$, respectively.
This produces a rotation given by the angle $\theta$.

The simulation is performed by the second order scheme \eqref{BDF2_2} with $\sigma=1$ and $1024\times1024$ Fourier modes.
Here, we use two types of temporal meshes: the fixed small time step with $\tau=0.01$ and the adaptive strategy \eqref{adp} with $\Dt_{min}=0.01, \Dt_{max}=1$ and $\gamma=10.$
In Figure \ref{fig3_1}, we display the evolution of the crystal growth up to $T=1500$ for these two types of temporal meshes, respectively.
It is observed that the snapshots of the crystal growth consist with each other at same time stages for these two temporal meshes.
The growth of the crystalline phase and the motion of well-defined crystal–liquid interfaces are observed.
Moreover, it shows that the different alignment of the crystallites causes defects and dislocations.
 This phenomenon has been also observed in \cite{EKHG02,GN12,HWWL09,LS20,YH17}.
In Figure \ref{fig3_2}, the evolution of {\color{black}the modified energy, the original energy,} and the corresponding adaptive time step sizes are plotted to validate the stability and the efficiency of our proposed scheme \eqref{BDF2_2} with the adaptive strategy \eqref{adp}.
{\color{black}\begin{remark}
As shown in Figure \ref{fig2_2} and Figure \ref{fig3_2}, the evolutions of the computed original energy with the tested time adaptive strategies are in agreement with the one computed by the small time step $\tau=0.01$ better than the case of the computed modified energy. This may be due to the computed auxiliary variable $r^{n}$ and the extra term $\frac{(2\sigma-1)\gamma_{n+1}^{3/2}}{2+2\gamma_{n+1}}\frac{\|\nabla^{-1}(\phi^{n}_{N}-\phi^{n-1}_{N})\|^{2}}{\tau_{n}}$ contained in the modified energy in \eqref{stab_ful2}.
\end{remark}
}

\begin{figure}[htbp]\label{fig3_1}
\begin{minipage}[t]{0.32\linewidth}
\centerline{\includegraphics[scale=0.3]{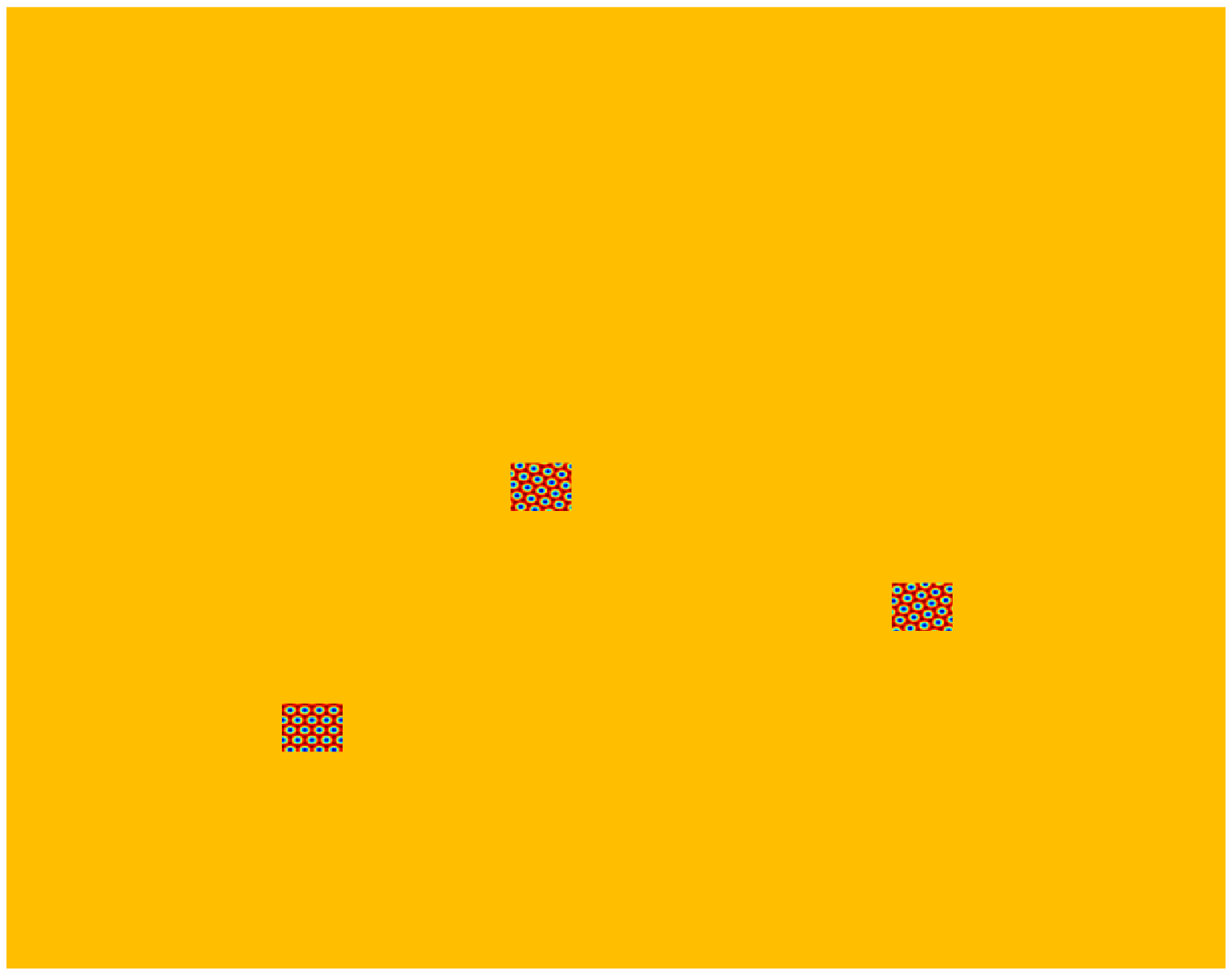}}
\centerline{}
\end{minipage}
\begin{minipage}[t]{0.32\linewidth}
\centerline{\includegraphics[scale=0.3]{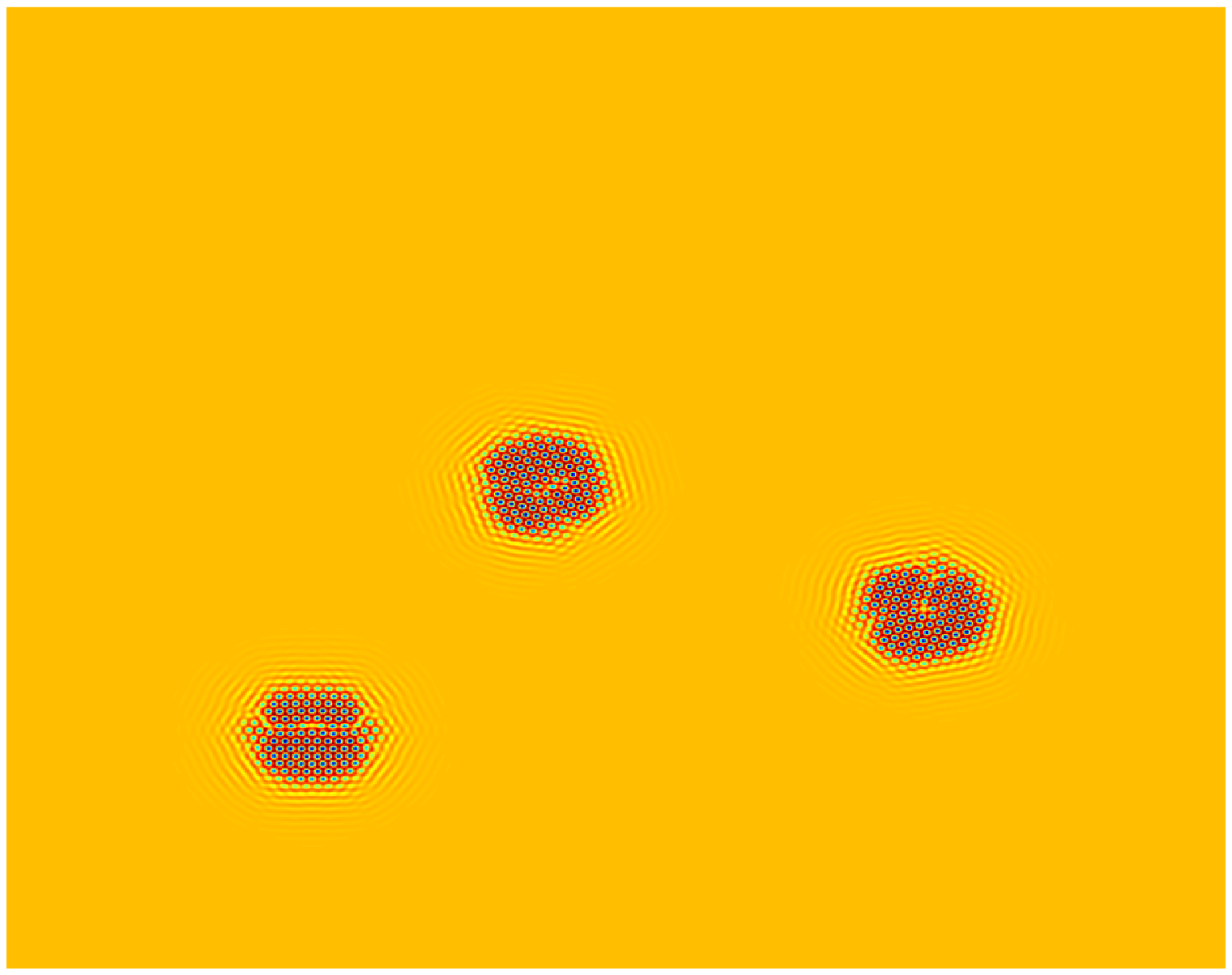}}
\centerline{}
\end{minipage}
\begin{minipage}[t]{0.32\linewidth}
\centerline{\includegraphics[scale=0.3]{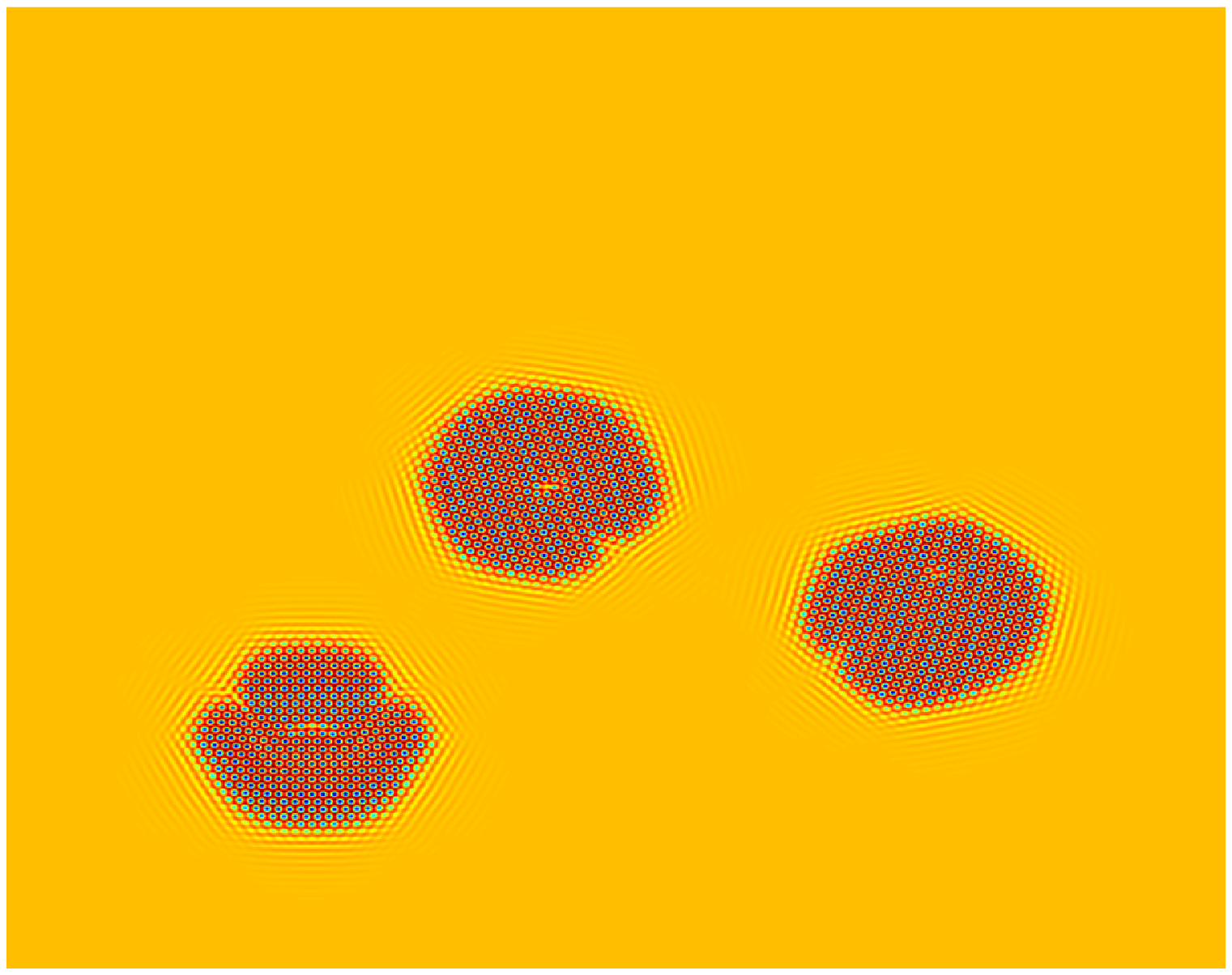}}
\centerline{}
\end{minipage}
\vskip -6mm
\begin{minipage}[t]{0.32\linewidth}
\centerline{\includegraphics[scale=0.3]{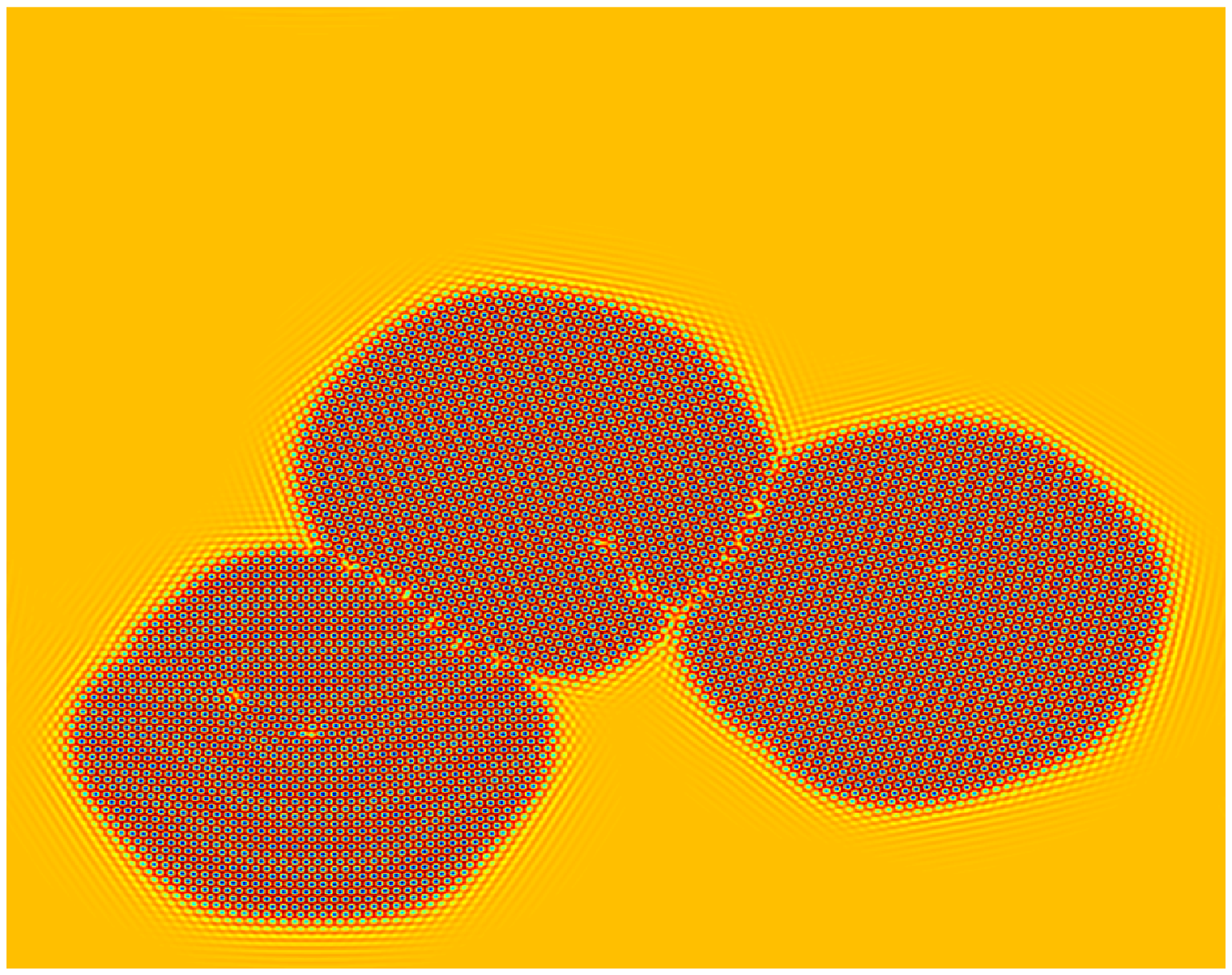}}
\centerline{}
\end{minipage}
\begin{minipage}[t]{0.32\linewidth}
\centerline{\includegraphics[scale=0.3]{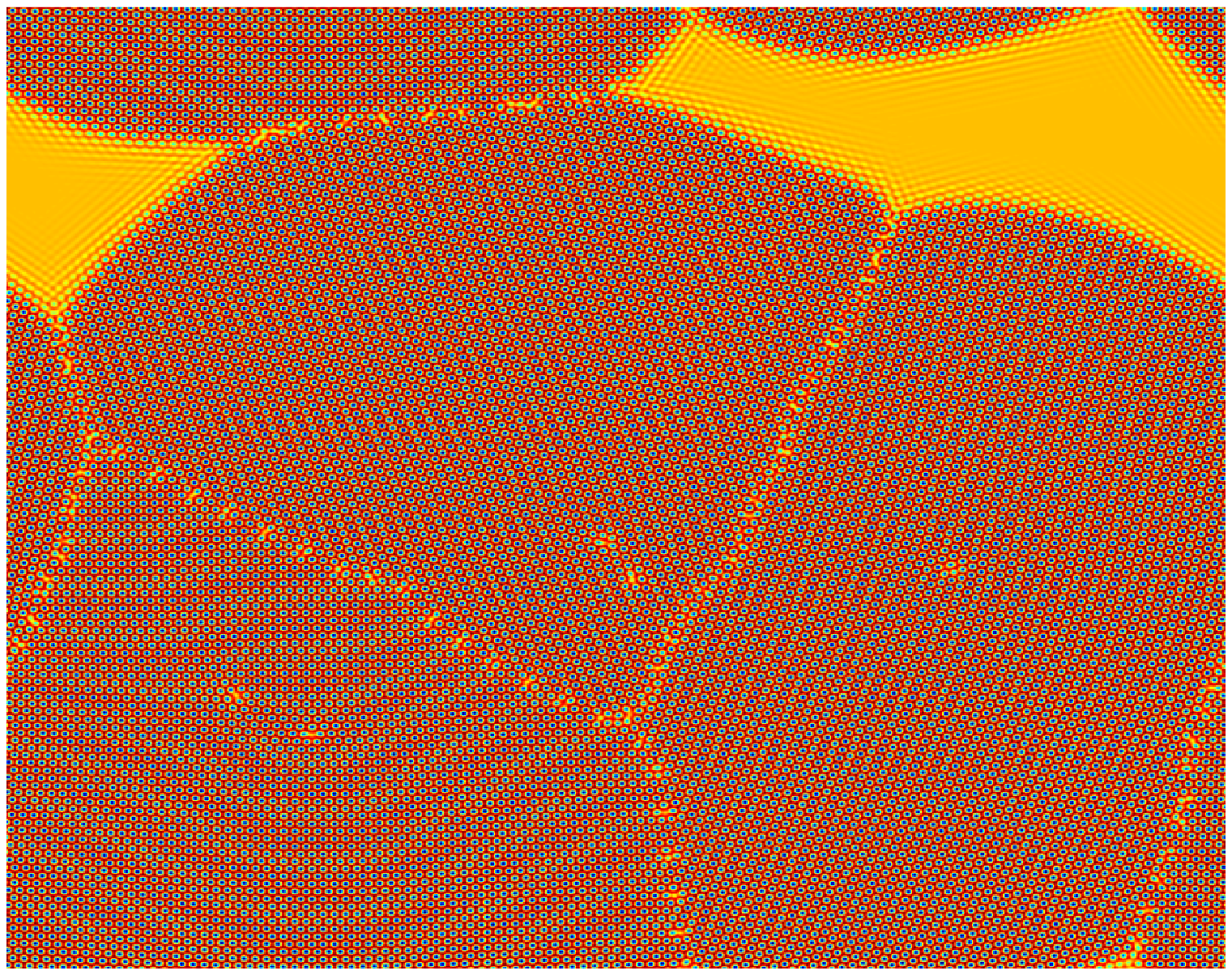}}
\centerline{(a)  Solutions with fixed time step size $\tau=0.01$ }
\end{minipage}
\begin{minipage}[t]{0.32\linewidth}
\centerline{\includegraphics[scale=0.3]{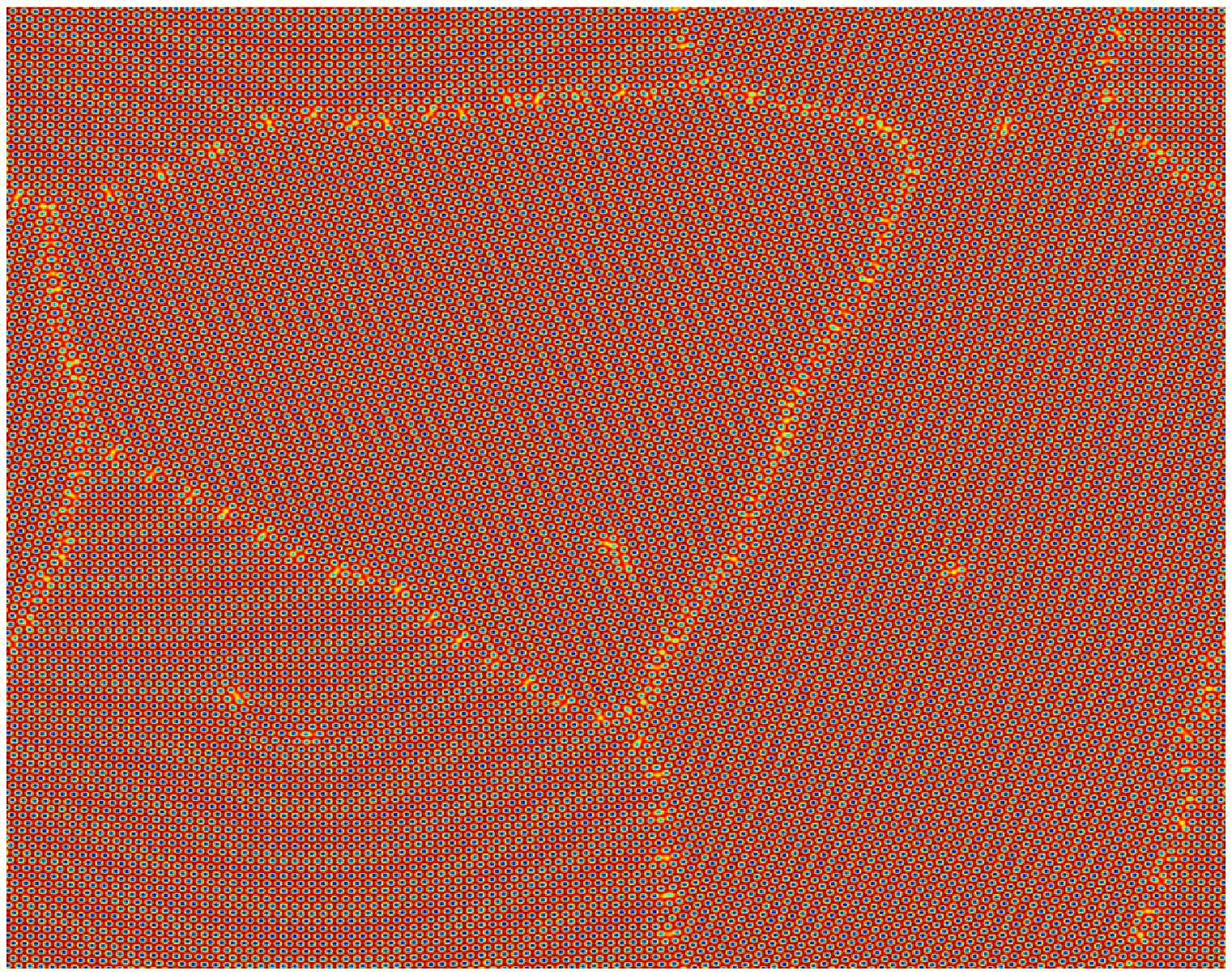}}
\centerline{}
\end{minipage}
\vskip 1mm
\begin{minipage}[t]{0.32\linewidth}
\centerline{\includegraphics[scale=0.3]{3_0.eps}}
\centerline{}
\end{minipage}
\begin{minipage}[t]{0.32\linewidth}
\centerline{\includegraphics[scale=0.3]{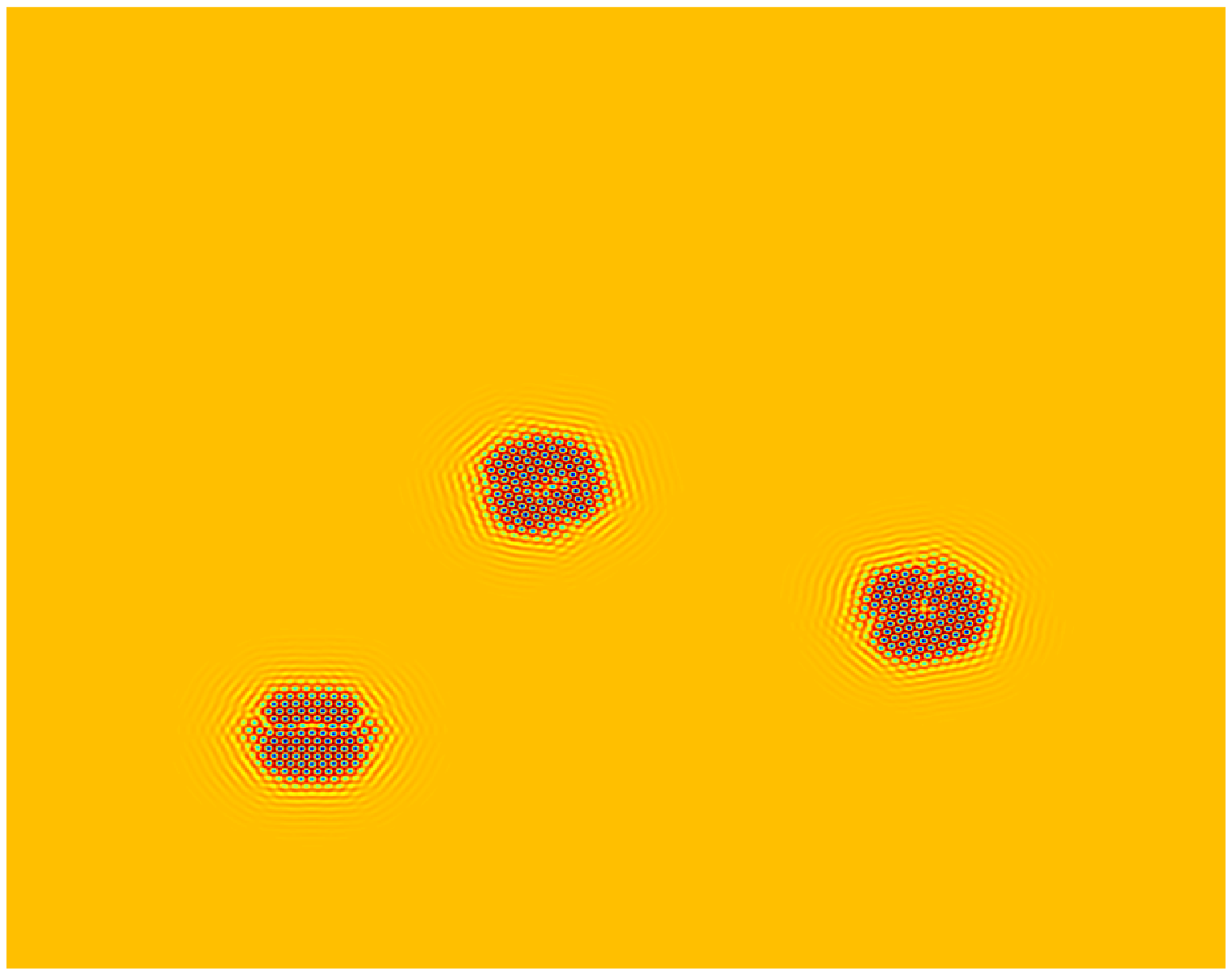}}
\centerline{}
\end{minipage}
\begin{minipage}[t]{0.32\linewidth}
\centerline{\includegraphics[scale=0.3]{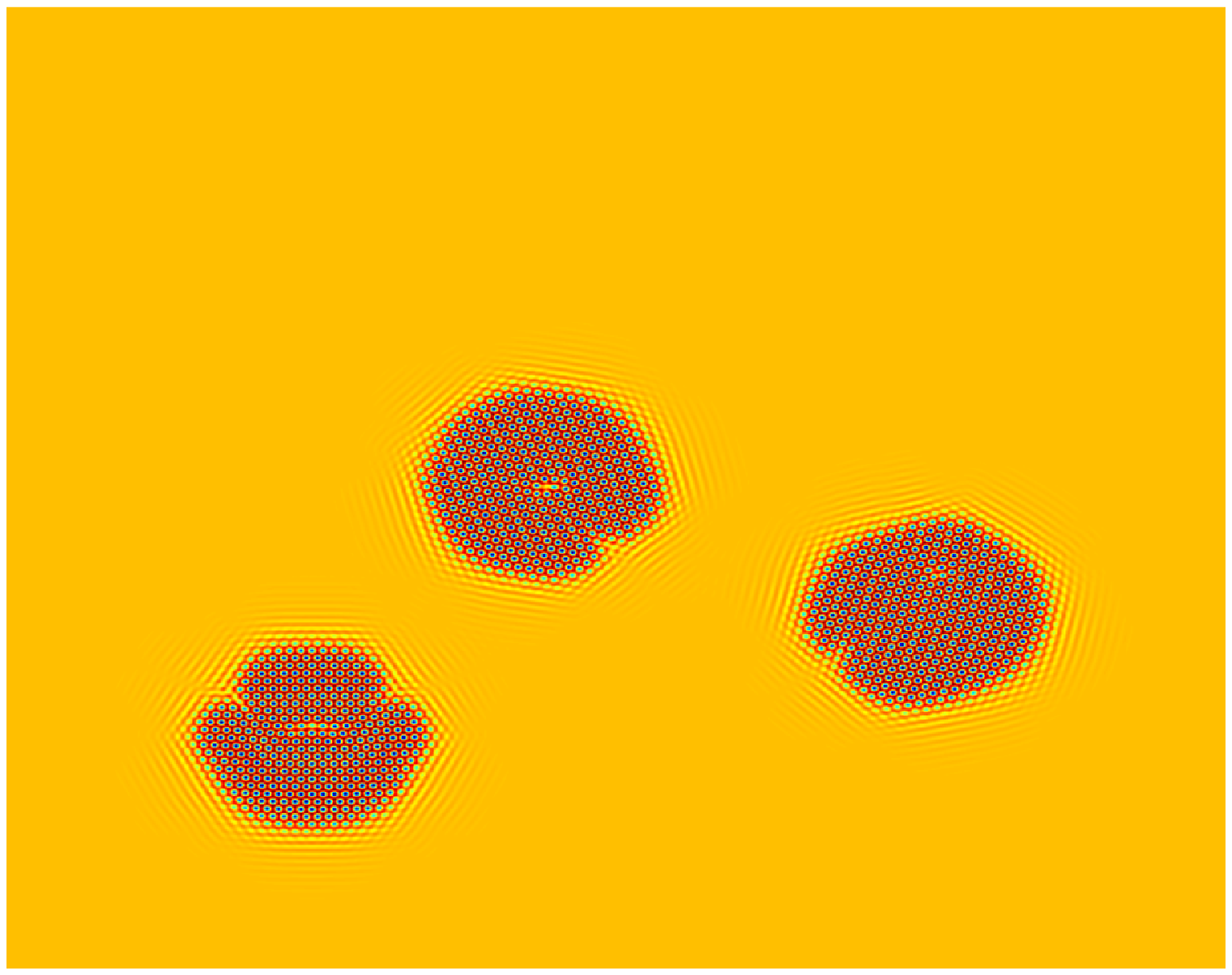}}
\centerline{}
\end{minipage}
\vskip -6mm
\begin{minipage}[t]{0.32\linewidth}
\centerline{\includegraphics[scale=0.3]{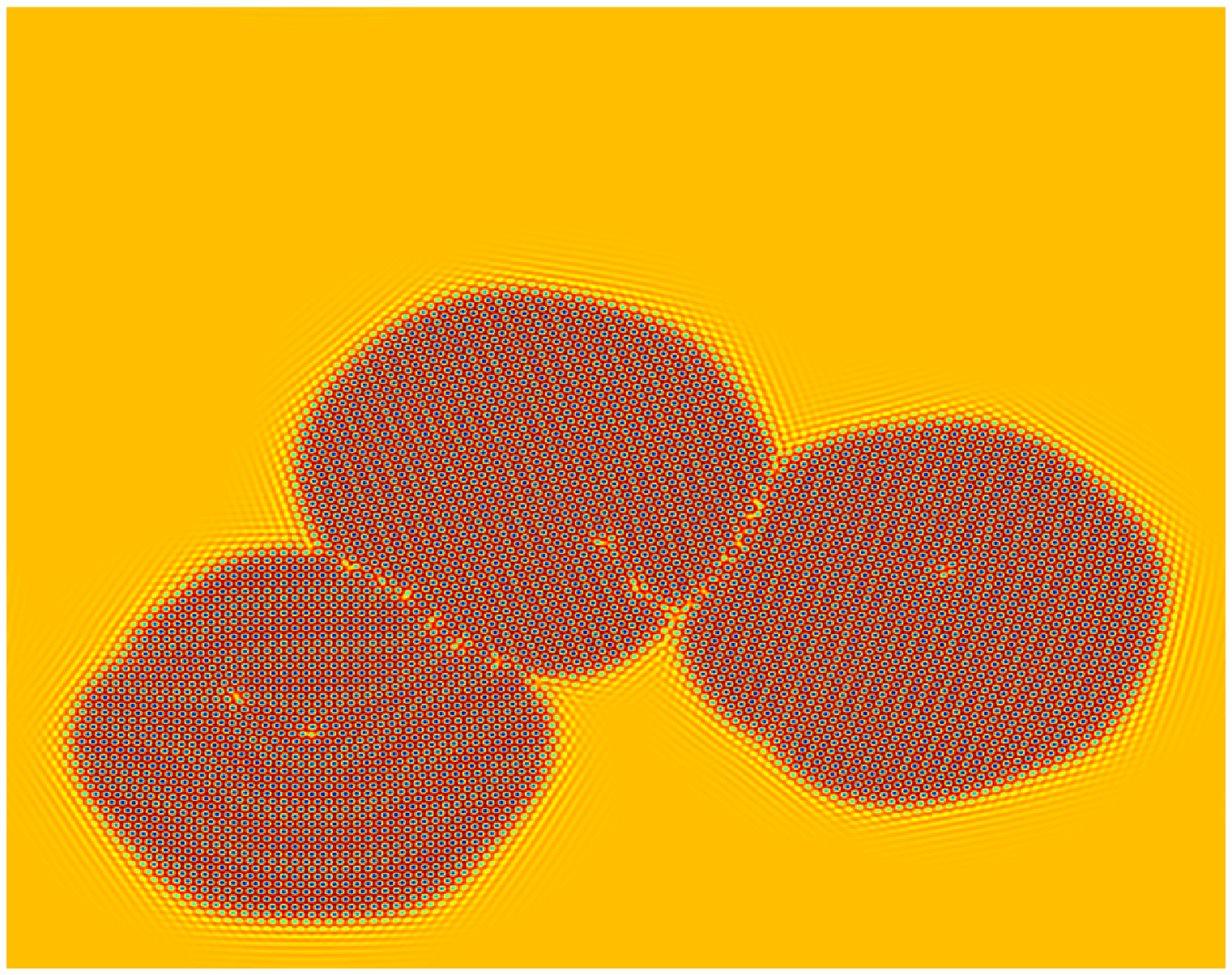}}
\centerline{}
\end{minipage}
\begin{minipage}[t]{0.32\linewidth}
\centerline{\includegraphics[scale=0.3]{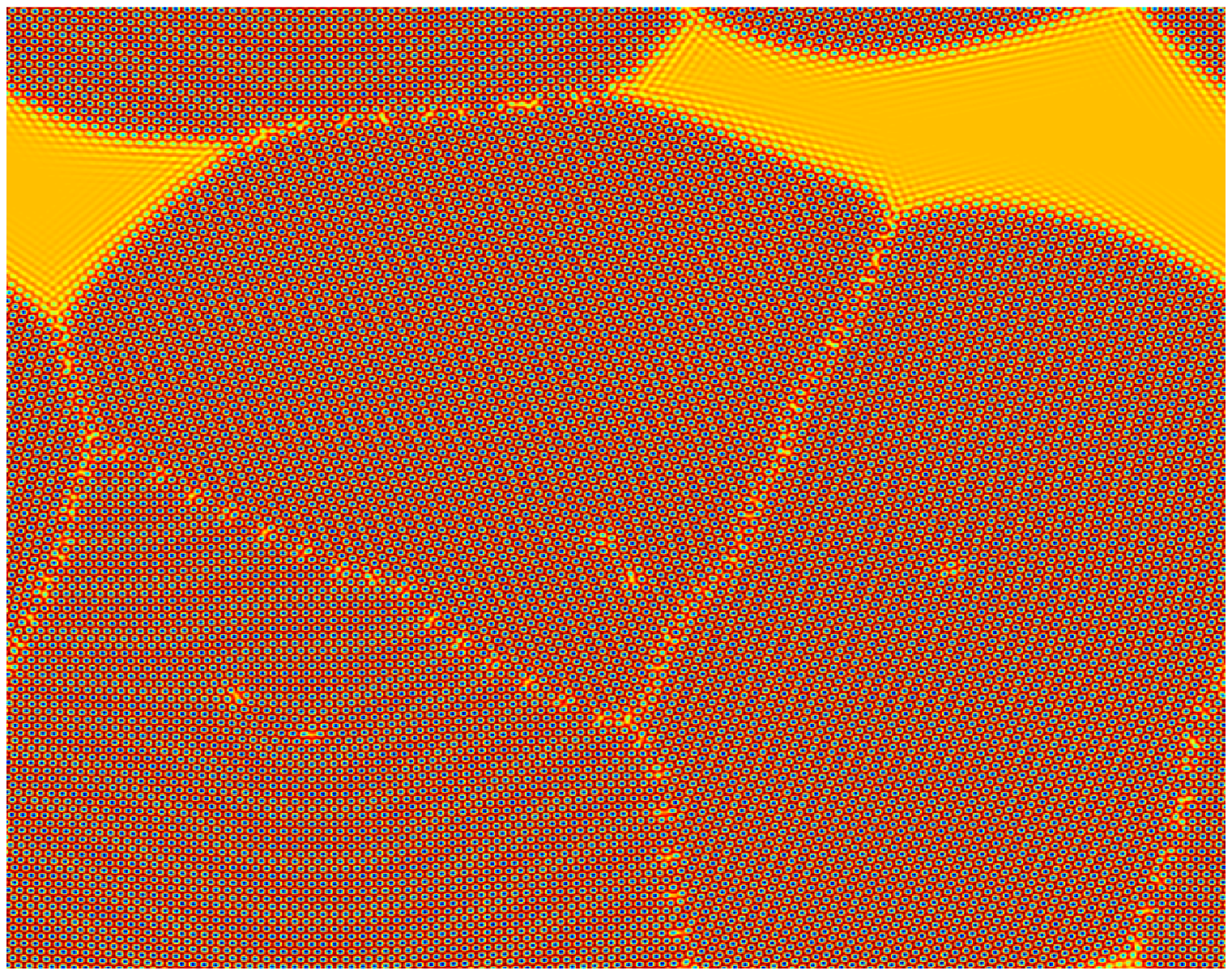}}
\centerline{(b) Solutions with adaptive time steps with $\tau_{min}=0.01, \tau_{max}=1$ and $\gamma=10.$}
\end{minipage}
\begin{minipage}[t]{0.32\linewidth}
\centerline{\includegraphics[scale=0.3]{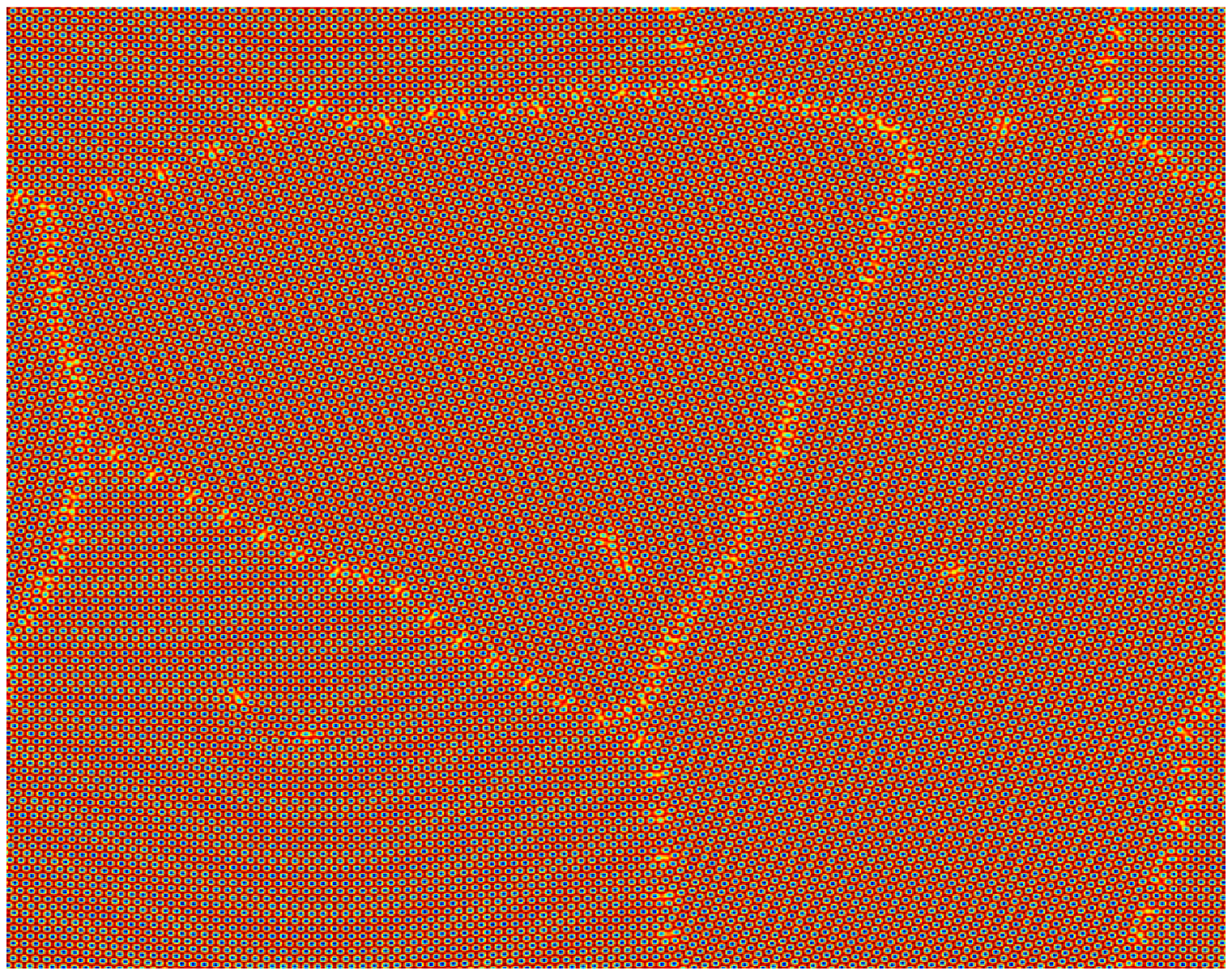}}
\centerline{}
\end{minipage}
\caption{Solution snapshots of crystal growth for the PFC equation at around $t=0, 100, 200, 400, 800$ and $1500$, respectively.
}\label{fig3_1}
\end{figure}

\begin{figure*}[htbp]
\begin{minipage}[t]{0.49\linewidth}
\centerline{\includegraphics[scale=0.4]{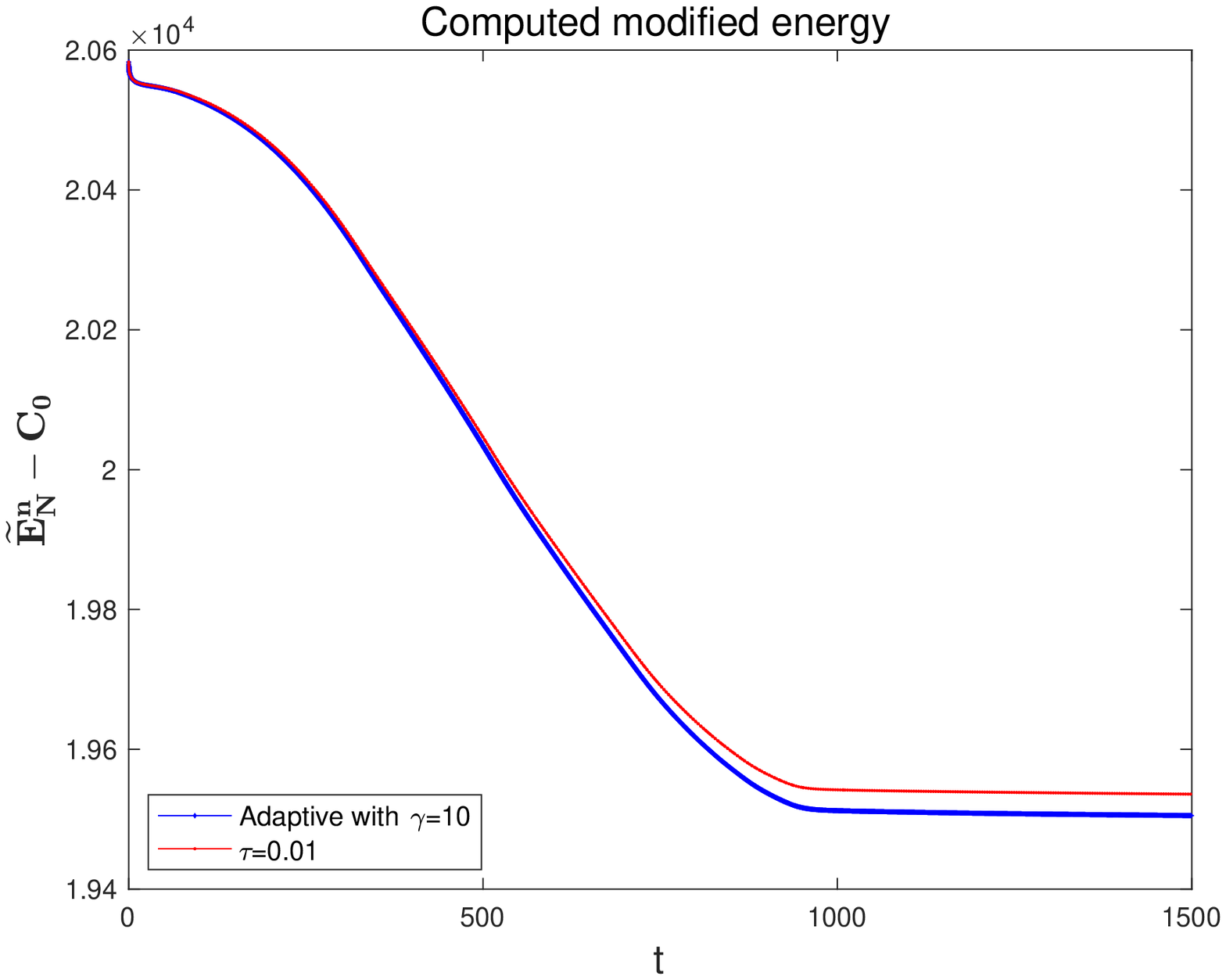}}
\centerline{(a) the computed modified energy $\widetilde{E}^{n}_{N}-C_{0}$}
\end{minipage}
\begin{minipage}[t]{0.49\linewidth}
\centerline{\includegraphics[scale=0.4]{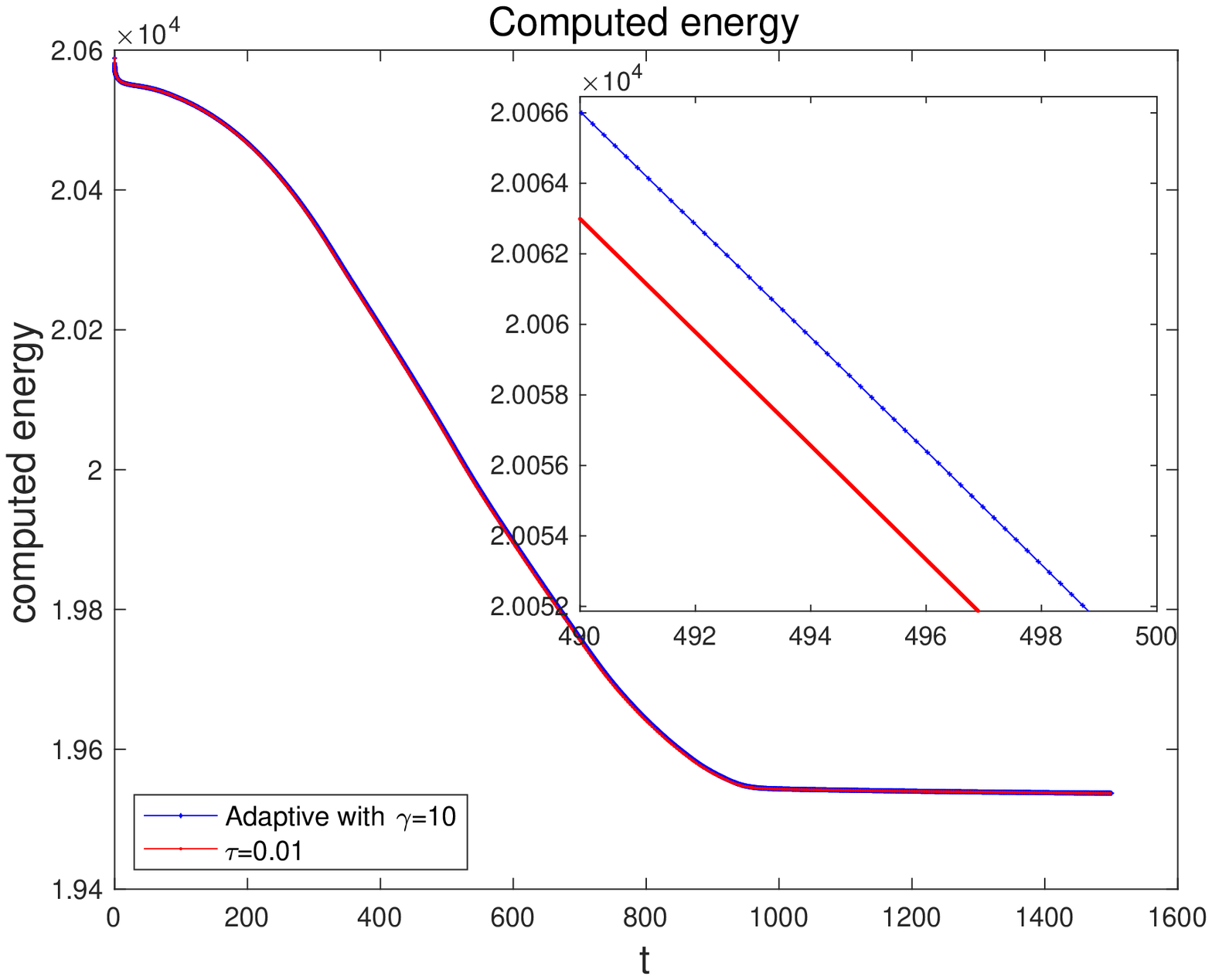}}
\centerline{(b) the computed energy $E(\phi^{n})$}
\end{minipage}
\begin{minipage}[t]{0.49\linewidth}
\centerline{\includegraphics[scale=0.4]{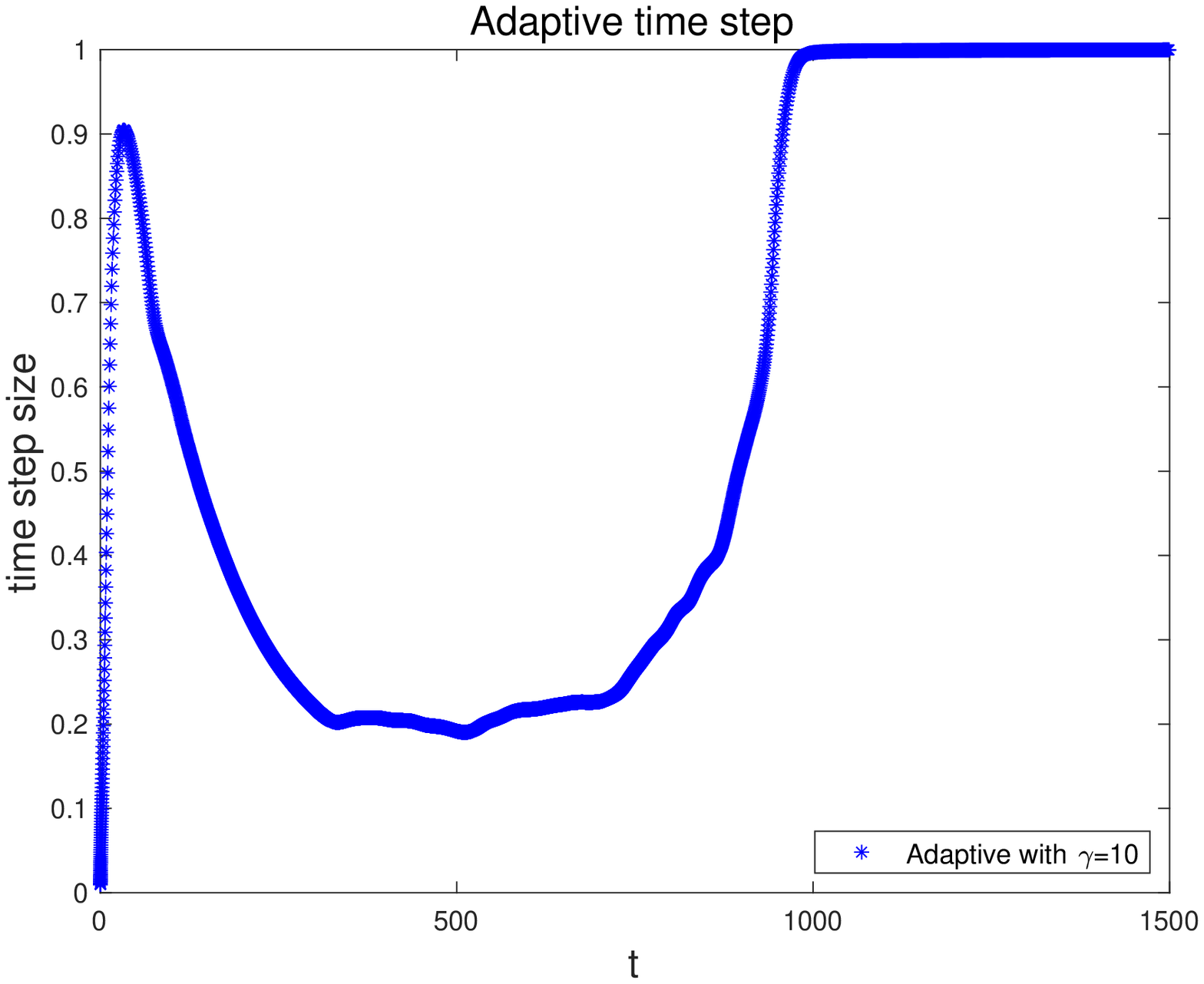}}
\centerline{(c) adaptive time step sizes with $\gamma=10$}
\end{minipage}
\caption{Evolution in time of the energy and the adaptive time step size for crystal growth.
}\label{fig3_2}
\end{figure*}

\section {Concluding remarks}\label{sec:conclusions}
We have developed and analyzed a linear second order nonuniform BDF scheme based on the SAV approach for the PFC model. The  energy stability of the proposed scheme is established for the nonuniform temporal mesh with a mild restriction on the adjacent time step radio $\gamma_{n}\leq4.8645.$
Although we only use a first order method to approximate the auxiliary variable $r(t)$, the second order accuracy of the phase function $\phi^{n}$ can be guaranteed by setting a large enough positive constant $C_{0}$ such that $C_{0}\geq 1/\Dt$. Moreover, a rigorous error estimate of our proposed fully discrete nonuniform BDF2 scheme is established with some proper assumptions on the regularity of the solution.
Finally, a series of  numerical tests have been carried out to validate the theoretical results and the efficiency of the proposed scheme combined with the time adaptive strategy.
\bibliographystyle{plain}
\bibliography{ref}
\end{document}